\documentclass[reqno,autoref,numberwithinsect]{amsart}

\usepackage{hyperref}
\usepackage{amsfonts,amsmath,amssymb,amsthm}
\usepackage{graphicx}
\usepackage{tikz}

\usetikzlibrary{math}
\usepackage{caption}
\usepackage[labelfont=rm]{subcaption}
\usepackage{xspace}
\usepackage{pgfplots}

\usepackage{bm}

\usepackage[normalem]{ulem}

\usepackage{interval}

\usepackage{empheq}

\usepackage[capitalize]{cleveref}

\crefformat{equation}{\textup{#2(#1)#3}}
\crefrangeformat{equation}{\textup{#3(#1)#4--#5(#2)#6}}
\crefmultiformat{equation}{\textup{#2(#1)#3}}{ and \textup{#2(#1)#3}}
{, \textup{#2(#1)#3}}{, and \textup{#2(#1)#3}}
\crefrangemultiformat{equation}{\textup{#3(#1)#4--#5(#2)#6}}%
{ and \textup{#3(#1)#4--#5(#2)#6}}{, \textup{#3(#1)#4--#5(#2)#6}}{, and \textup{#3(#1)#4--#5(#2)#6}}

\Crefformat{equation}{Equation~#2\textup{(#1)}#3}
\Crefrangeformat{equation}{Equations~\textup{#3(#1)#4--#5(#2)#6}}
\Crefmultiformat{equation}{Equations~\textup{#2(#1)#3}}{ and \textup{#2(#1)#3}}
{, \textup{#2(#1)#3}}{, and \textup{#2(#1)#3}}
\Crefrangemultiformat{equation}{Equations~\textup{#3(#1)#4--#5(#2)#6}}%
{ and \textup{#3(#1)#4--#5(#2)#6}}{, \textup{#3(#1)#4--#5(#2)#6}}{, and \textup{#3(#1)#4--#5(#2)#6}}

\usepackage{thmtools, thm-restate}
\usepackage{todonotes}

\usepackage{dsfont}

\newtheorem{theorem}{Theorem}[section]
\newtheorem{lemma}[theorem]{Lemma}
\newtheorem{proposition}[theorem]{Proposition}
\newtheorem{corollary}[theorem]{Corollary} 
\theoremstyle{definition}
\newtheorem{definition}[theorem]{Definition}

\theoremstyle{remark}
\newtheorem{remark}[theorem]{Remark}
\newtheorem{observation}[theorem]{Observation}
\newtheorem{question}[theorem]{Question}
\newtheorem{example}[theorem]{Example}
\newtheorem{conjecture}[theorem]{Conjecture}

\usepackage{algorithmicx}
\usepackage{algorithm}
\usepackage{algpseudocode}

\usepackage[inline,shortlabels]{enumitem}
\setlist[enumerate]{leftmargin=1.5pc}

\algdef{SE}[DOWHILE]{Do}{DoWhile}{\algorithmicdo}[1]{\algorithmicwhile\ #1}
\algrenewcommand{\algorithmiccomment}[1]{\hfill $\rhd$ \emph{#1}}
\algrenewcommand{\algorithmicrequire}{\textbf{Input:}}
\algrenewcommand{\algorithmicensure}{\textbf{Output:}}
\algnewcommand{\Or}{\textbf{or}}
\algnewcommand{\And}{\textbf{and}}
\algnewcommand{\Not}{\textbf{not}\,}

\newcommand\SetOf[2]{\left\{\left.#1\vphantom{#2}\ \right|\ #2\vphantom{#1}\right\}}

\DeclareMathOperator{\supp}{supp}
\DeclareMathOperator{\psupp}{supp^{\oplus}}
\DeclareMathOperator{\nsupp}{supp^{\ominus}}
\newcommand\KK{{\mathbb K}}
\newcommand\HH{{\mathbb H}}
\DeclareMathOperator\val{val}
\newcommand\TT{{\mathbb T}}
\newcommand\TTpm{\mathbb T_{\pm}}

\newcommand\RR{{\mathbb R}}

\newcommand\teq{\mathrel{\vDash}}

\newcommand{\TSS}{\mathbb{S}}
\newcommand{\SignH}{\mathbb{O}}

\DeclareMathOperator{\whull}{\conv_{TC}} 
\DeclareMathOperator{\shull}{\conv_{TO}} 
\DeclareMathOperator{\wcone}{\cone_{\mspace{1mu} {TC}}}

\DeclareMathOperator{\wspan}{span_{\mspace{1mu} {TC}}}
\DeclareMathOperator{\sval}{sval}
\DeclareMathOperator{\conv}{conv}
\DeclareMathOperator{\slog}{slog}

\DeclareMathOperator{\sgn}{sgn}
\DeclareMathOperator{\tsgn}{tsgn}
\DeclareMathOperator{\argmax}{argmax}

\newcommand\Tleq{\TT_{\leq \mathds{O}}}
\newcommand\Tgeq{\TT_{\geq \mathds{O}}}
\newcommand\Tlt{\TT_{< \mathds{O}}}
\newcommand\Tgt{\TT_{> \mathds{O}}}
\newcommand\Tzero{\TT_{\bullet}}
\newcommand\Zero{\mathds{O}}

\newcommand\Uncomp{\mathcal{U}}
\newcommand\HC{\mathcal{H}}
\newcommand\Hclosed{\overline{\mathcal{H}}}
\newcommand\Hsemi{\widetilde{\mathcal{H}}}

\newcommand\puiseux[2]{#1\{#2\}} 

\DeclareMathAlphabet\mathbfcal{OMS}{cmsy}{b}{n}

\newcommand\bHclosed{\overline{\mathbfcal{H}}}
\newcommand\subs{J}
\newcommand{\lif}{\mathsf{li}}
\newcommand{\clif}{\mathsf{cli}}

\DeclareMathOperator{\cone}{cone}

\newcommand{\lplus}{\triangleleft}
\newcommand{\biglplus}{\lhd}

\newcommand{\vertices}{\mathrm{Vert}}
\newcommand{\faces}{\mathrm{Faces}}
\newcommand{\Sym}{\mathrm{Sym}}

\newcommand{\diag}{\mathrm{diag}}

\newcommand{\cl}{\mathrm{cl}}
\newcommand{\inter}{\mathrm{int}}

\newcommand{\ort}{O}
\newcommand{\bort}{\bm O}
\newcommand{\dort}{Q}

\newcommand{\domin}{\mathrm{Argmax}}
\newcommand{\posdomin}{\domin^{+}}

\newcommand{\lc}{\mathrm{lc}}

\newcommand{\ortsig}{\nu}
\newcommand{\coG}{H}
\newcommand{\ccG}{\overline{H}}
\newcommand{\lclam}{\gamma}

\newcommand{\Xperp}{X'}
\newcommand{\Gper}{\bm{G}^{\perp}}
\newcommand{\Gali}{\bm{G}^{=}}

\newcommand{\cM}{\mathcal{M}}

\makeatletter
\def\wrt{w.r.t\@ifnextchar.{}{.\ }}
\makeatother

\makeatletter
\def\input@path{{./Images/}{./}}
\makeatother

\title{Signed tropical halfspaces and convexity}

\author{Georg Loho \and Mateusz Skomra}

\address[Georg Loho]{
  University of Twente,
  Department of Applied Mathematics \\
  The Netherlands \\
}
\email{g.loho@utwente.nl}

\address[Mateusz Skomra]{
  LAAS-CNRS\\
  Universit{\'e} de Toulouse\\
  CNRS\\
  Toulouse\\
  France\\
}
\email{mateusz.skomra@laas.fr}

\begin{document}

\maketitle

\begin{abstract}
  We extend the fundamentals for tropical convexity beyond the tropically positive orthant expanding the theory developed by Loho and V{\'e}gh (ITCS 2020). 
  We study two notions of convexity for signed tropical numbers called \emph{TO-convexity} (formerly `signed tropical convexity') and the novel notion \emph{TC-convexity}. 

  We derive several separation results for TO-convexity and TC-convexity.
  A key ingredient is a thorough understanding of TC-hemispaces -- those TC-convex sets whose complement is also TC-convex. 
  Furthermore, we use new insights in the interplay between convexity over Puiseux series and its signed valuation.

  Remarkably, TC-convexity can be seen as a natural convexity notion for representing oriented matroids as it arises from a generalization of the composition operation of vectors in an oriented matroid.
  We make this explicit by giving representations of linear spaces over the real tropical hyperfield in terms of TC-convexity. 
\end{abstract}


\section{Introduction}

Convexity is a powerful structure which is often behind the existence of efficient algorithms.
In this spirit, tropical convexity, arising from classical convexity by replacing addition with maximization and multiplication by addition, found several applications from optimization~\cite{AllamigeonBenchimolGaubertJoswig:2021}, phylogenetics~\cite{YoshidaZhangZhang:2019} and machine learning~\cite{MaragosCharisopoulosTheodosis:2021}, just to name a few recent references. 
While tropical convexity was mainly considered with a nonnegativity constraint for a long time, the recent paper \cite{LohoVegh:2020} introduced a notion of signed tropical convexity to overcome this restriction.
The main drawback of the convexity introduced in \cite{LohoVegh:2020} is that the closed tropical halfspaces are usually \emph{not} convex.
  This has two major disadvantages. 
  Firstly, there exist signed tropical polyhedra (intersections of finitely many closed tropical halfspaces) that are not convex.
  Secondly, the image by a signed valuation of a convex set defined over Puiseux series may not be convex. 
  In this paper, we extend the fundamentals of signed tropical convexity and thereby resolve limitations of the framework developed in \cite{LohoVegh:2020}.
  To do so, we introduce a second, extended notion of signed tropical convexity.

  The crucial building blocks of the two versions of signed tropical convexity are open and closed tropical halfspaces.
  We coin the name \emph{TO-convexity} (tropical open convexity) for the convexity in which the hull of finitely many points equals the intersection of the \emph{open} tropical halfspaces containing them.
  This is exactly the notion considered in~\cite{LohoVegh:2020}.
  Additionally, we define \emph{TC-convexity} (tropical closed convexity) in such a way that the hull of finitely many points equals the intersection of the \emph{closed} tropical halfspaces containing them.
  In this way, signed tropical polyhedra are TC-convex even though they may not be TO-convex and the same is true about the signed valuations of convex sets.

  Our key contribution are various separation theorems for these convexities:

  \begin{itemize}
  \item Two disjoint TO-convex sets can be separated by a closed tropical halfspace (\cref{th:weak_separation_TO}). 
  \item Each TC-convex set is an intersection of TC-hemispaces; here, a TC-hemispace is a well-structured set sandwiched between an open and a closed tropical halfspace (\cref{thm:separation_hemispaces} and \cref{th:hemispace}). 
  \item The image of a closed semi-algebraic set under the signed valuation can be separated from a point by a closed tropical halfspace (\cref{thm:tropical_polar}). 
  \end{itemize}

    The weak separation of two disjoint TO-convex sets by a closed tropical halfspace
  is an analogue of the hyperplane separation theorem of real convexity and extends the separation theorem known for two tropically convex sets in the nonnegative orthant \cite{GaubertSergeev:2008}.
  Our results imply 
  that a set is TC-convex if and only if it is closed by taking a hull of two points (\cref{thm:TC-hull-intersection-halfspaces}).
  For technical reasons, we actually introduce the TC-convexity via the hull of only two points (\cref{def:weak+hull+segment}) and then show that the hull of a bigger set is an intersection of closed halfspaces.
  Doing the proof in this order also resembles more closely the proof of the hyperplane separation theorem over the reals.
  We combine these insights with further tools to derive the following representations: 

  \begin{itemize}
  \item The closure of a TO-convex set equals the intersection of the closed tropical halfspaces that contain it (\cref{th:TO-closure_separation}).  
  \item The TC-convex hull of a finite set $X$ equals the intersection of the closed tropical halfspaces containing $X$ (\cref{th:conic_M-W} and \cref{thm:affine_M-W}).
  \end{itemize}

  The latter is a signed tropical analogue of the Minkowski--Weyl theorem over the reals, which extends the tropical Minkowski--Weyl theorem known for the nonnegative orthant~\cite{GaubertKatz:2011}.

  Finally, we demonstrate how TC-convexity can be seen as the fundamental notion of tropical convexity for signed tropical geometry.
  We highlight a geometric interpretation of structural results for oriented matroids and matroids over the signed tropical hyperfield.
  In unsigned tropical geometry, the Bergman fan of a matroid arises by considering it as a matroid over the tropical hyperfield and taking its set of vectors there.
  This object turns out to be the tropical convex hull of the circuits (interpreted as tropical vectors) as identified in \cite{MurotaTamura:2001}, see \cite{Joswig:2022} for more details. 
  Considering this over the signed tropical hyperfield leads to the `real Bergman fan' in the sense of~\cite{Celaya:2019}, extending the understanding of the Bergman fan in the positive orthant~\cite{ArdilaKlivansWilliams:2006}, and more generally to linear spaces of matroids over the signed tropical hyperfield.
  We show that such a linear space is actually TC-convex.
  Furthermore, it has a representation, similar to the unsigned ones:
  \begin{itemize}
  \item The linear space of a matroid over the signed tropical hyperfield is the TC-convex hull of the circuits and the intersection of tropical hyperplanes associated with the cocircuits (\cref{thm:representation-vectors-tropical-hyperfields}).
  \end{itemize}
  Since tropical linear spaces are the fundamental building blocks of tropical algebraic geometry, our result reveals that TC-convexity is the right notion of signed tropical convexity for signed tropical geometry.  
  
  \smallskip

To arrive at our main separation results, we need to gather extensive insights into the structure of TC-convex sets and their lifts to Puiseux series.
Notably, our analysis of the signed valuations of closed semialgebraic sets relies on the analogue of the hyperplane separation theorem over Puiseux series, which we recall in \cref{th:hyperplane_sep}.
We use this connection with Puiseux series to strengthen the separation results directly derived for TC-convex sets.
The main insight for these is the fundamental elimination property of TC-convexity given in \cref{prop:sand_glass}.
To combine the separation of TC-convex sets with separation of lifts, we investigate TC-hemispaces -- those TC-convex sets whose complement is also TC-convex.
Apart from showing that any nontrivial TC-hemispace is sandwiched between an open tropical halfspace and its closure, we provide several structural insights on the boundary of non-closed TC-hemispaces. 
These insights are summarized in \cref{le:bigger_pos_argmax}.
This provides a partial generalization of the full characterization of hemispaces in the nonnegative tropical orthant established in~\cite{BriecHorvath:2008,KatzNiticaSergeev:2014,EhrmannHigginsNitica:2016}.
We connect the study of TC-hemispaces with the study of lifts into Puiseux series by showing that every TC-hemispace arises as a signed valuation of some convex set (\cref{prop:hemispace_puiseux}).
Additionally, we derive Carath{\'e}odory-type results for TC-convexity (\cref{le:inter_caratheodory}).

\subsection*{Comparison of TO-convexity and TC-convexity}
From the viewpoint of abstract convexity~\cite{VanDeVel:1993}, TO-convexity is generated by open tropical halfspaces and TC-convexity is generated by closed tropical halfspaces.
This yields that TO-convex sets are also TC-convex.
While TO-convex sets still behave rather well with the algebraic operations, as TO-convex hull can be written using (tropical) convex combinations with hyperoperations, this structure is lost for TC-convexity.

To describe the latter, we use a non-commutative addition, which we call \emph{left sum} (\cref{def:left+sum}). 
It generalizes the composition of vectors in an oriented matroid.
In this way, this operation already appeared in work on Bergman fans for matroid over hyperfields generalizing composition of sign vectors of an oriented matroid \cite{Celaya:2019,Anderson:2019}.
We elaborate on this in \cref{sec:oriented-matroids}. 

Another remarkable difference is the behavior with respect to lifts to Puiseux series.
In \cite[Theorem~3.14]{LohoVegh:2020}, it was shown that the TO-convex hull of finitely many points arises as union of the signed valuations of the convex hulls ranging over all lifts. 
We show in \cref{cor:intersection-of-lifts}, that the TC-convex hull arises as the intersection of these lifts. 

\subsection*{Motivation}

Our work is motivated by the need for a better structural basis underlying (recent) applications of tropical convexity.
New insights on the complexity of classical linear programming based on tropical geometry have been a great success story in recent years~\cite{AllamigeonBenchimolGaubertJoswig:2014,AllamigeonBenchimolGaubertJoswig:2018,AllamigeonBenchimolGaubertJoswig:2021,AllamigeonGaubertVandame:2022}. 
The point of origin for these advances is the tropicalization of linear programs -- where one always had to impose additional nonnegativity constraints. 
A thorough foundation of signed tropical convexity will allow to study the signed tropicalization of general linear programs. 

Furthermore, tropical convexity has an intimate connection with mean payoff games as the feasibility problem for tropical inequality systems (restricted to the tropical nonnegative orthant) is equivalent to mean payoff games \cite{AkianGaubertGuterman:2012}.
The latter have an intriguing complexity status in the intersection of NP and co-NP while no polynomial-time algorithm is known~\cite{ZwickPaterson:1996}; see the recent paper~\cite{ColcombetFijalkowGawrychowskiOhlmann:2022} for a good overview including the flourishing advances on the subclass of parity games. 
We think that signed tropical convexity may enrich the insights in the structure of these games. 

Studying the polar of a tropical polyhedron naturally leads to tropically convex sets without a nonnegativity constraint.
To overcome this, the nonnegativity constraints were modeled by pairs of nonnegative numbers in former work~\cite{AllamigeonGaubertKatz:2011}.
A full characterization characterization of the polars of tropically convex cones in the nonnegative orthant was established in the recent paper~\cite{AkianAllamigeonGaubertSergeev:2023} using the notion of signed bend cones. 
In upcoming work, we study polars in more detail and give a unified view of signed bend cones beyond the nonnegative orthant based on TC-convexity. 

Yet another motivation comes from interest in signed tropicalization of semialgebraic sets \cite{JellScheidererYu:2018}.
The study of unsigned tropicalizations of semialgebraic sets already lead to a fruitful connection between stochastic games and (non-Archimedean) semidefinite programs \cite{AllamigeonGaubertSkomra:2018}.
In particular, the analysis of tropical cones arising in this work resulted in recent universal complexity bounds on value iteration for a large class of games in~\cite{AllamigeonGaubertKatzSkomra:2022}.
Our framework allows to capture also other classes of optimization problems where the nonnegativity condition is a priori not satisfied. 

Finally, already in~\cite{LohoVegh:2020}, it was demonstrated that signed tropical linear inequality systems with unit coefficients are the same as Boolean formulas.
In this sense, more general signed tropical linear inequality systems form a `quantitative' generalization of Boolean formulas.
Therefore, the feasibility problem for linear systems over signed tropical numbers is a natural generalization of SAT.
Note that the solution sets of these systems are actually not TO-convex.
However, they are exactly the finitely generated TC-convex sets. 

\subsection*{Related work}
We refer to the book~\cite{Joswig:2022} for a general overview on (unsigned) tropical convexity. 
Different versions of separation theorems in tropical convexity have been obtained by numerous authors~\cite{Zimmermann:1977,SamborskiiShpiz:1992,CohenGaubertQuadrat:2005,DevelinSturmfels:2004,BriecHorvath:2008,GaubertSergeev:2008}. Likewise, the Carath{\'e}odory theorem for tropical convexity has been discovered independently in different works~\cite{Helbig:1988,BriecHorvath:2004,DevelinSturmfels:2004}. The tropical analogue of the Minkowski--Weyl theorem was proven in \cite{GaubertKatz:2011}. In order to prove our main separation theorem, we rely on works on abstract convexity, which were already applied to the usual tropical convexity in \cite{Horvath:2017}. Furthermore, our proof requires to give a partial characterization of signed tropical hemispaces. A full characterization of tropical hemispaces in one orthant is given in \cite{BriecHorvath:2008,KatzNiticaSergeev:2014,EhrmannHigginsNitica:2016}.

The link between tropical convexity and convexity over Puiseux series was established in the work \cite{DevelinYu:2007} which studies the tropicalization of polyhedra. This was later generalized to spectrahedra in \cite{Yu:2015,AllamigeonGaubertSkomra:2020} and to convex semialgebraic sets in \cite{AllamigeonGaubertSkomra:2019}. The tropicalizations of general semialgebraic sets are studied in \cite{Alessandrini:2013} and this study is extended to signed tropicalizations in \cite{JellScheidererYu:2018} and \cite[Section~4]{AllamigeonGaubertSkomra:2020}.

Signed tropical numbers first appeared in the context of the symmetrized semiring~\cite{Plus:90}. 
The idea that tropical convexity could be extended to signed tropical numbers appeared already in \cite{BriecHorvath:2004,Briec:2015}, where two such extensions (different than the ones considered here) are introduced. Our work is most closely related to the paper \cite{LohoVegh:2020} which introduces and studies TO-convexity.

Inspired by our work, the structure of TC-convexity was identified in the study of curves \cite{LeTexier:2021}. 
Furthermore, there has been progress on extending signed tropical convexity to convexity over more general hyperfields \cite{MaxwellSmith:2023}.

Our representation of linear spaces associated with oriented (valuated) matroids in terms of TC-convexity extends the real Bergman fan introduced in \cite{Celaya:2019}. 
It draws heavily from recent developments for matroids over hyperfields \cite{Anderson:2019,BowlerPendavingh-arxiv:2023,BakerBowler:2019}. 

\section{Preliminaries on signed tropical numbers}

We give a brief overview of necessary notions related to signed numbers; for more see~\cite{AkianGaubertGuterman:2014,Plus:90,LohoVegh:2020}.   
The signed tropical numbers $\TTpm$ are obtained by glueing two copies of $\left(\RR \cup \{-\infty\}\right)$ at the tropical zero element $\Zero = -\infty$ giving rise to the \emph{nonnegative tropical numbers} $\Tgeq = \RR \cup \{\Zero\}$ and the \emph{non-positive tropical numbers} $\Tleq = \SetOf{\ominus x}{x \in \RR} \cup \{\Zero\}$.

The signed numbers can be extended to the \emph{symmetrized semiring} $\TSS$, which forms a semiring containing $\TTpm$, by attaching \emph{balanced numbers} $\Tzero = \SetOf{\bullet x}{x \in \RR} \cup \{\Zero\}$. 
We will often use the norm $|\,.\,|$ on $\TTpm$ which maps each element of $\Tgeq$ to itself, removes the $\ominus$ sign of an element in $\Tleq$ and the $\bullet$ from an element in $\Tzero$. 
This is complemented by the map $\tsgn$ from $\TSS$ to $\{\oplus,\ominus,\bullet,\Zero\}$ keeping only the sign information.

Throughout, we use the notation $[d] = \{1,2,\dots,d\}$ and $[d]_0 = [d] \cup \{0\}$. 
For a vector $z \in \TTpm^d$ we denote its support, positive support and negative support, respectively, by
\begin{equation} \label{eq:support-notation}
\begin{aligned}
 \supp(z) &= \SetOf{i \in [d]}{z_i \neq \Zero} \\
 \psupp(z) &= \SetOf{i \in [d]}{z_i \in \Tgt} \\
 \nsupp(z) &= \SetOf{i \in [d]}{z_i \neq \Tlt} \enspace .
\end{aligned}
\end{equation}

For $x,y \in \TSS$, the addition is defined by
\begin{equation} \label{eq:balanced+addition}
  x \oplus y = \begin{cases}
    \argmax_{x,y}(|x|,|y|) & \text{ if } |\chi|=1 \\
    \bullet \argmax_{x,y}(|x|,|y|) & \text{ else } \enspace .
  \end{cases}
\end{equation}
where $\chi = \SetOf{\tsgn(\xi)}{\xi \in (\argmax(|x|,|y|))}$.
The multiplication is given by
\begin{equation}
  x \odot y = \left(\tsgn(x) * \tsgn(y)\right) (|x| + |y|) \enspace ,
\end{equation}
where the $*$-multiplication table is the usual multiplication of $\{-1,1,0\}$ for $\{\ominus,\oplus,\bullet\}$ with the additional specialty that multiplication with $\Zero$ yields $\Zero$. 
The operation $\oplus$ extends to vectors and matrices componentwise, while $\odot$ extends in analogy to matrix multiplication.  

We use $\oplus$, $\ominus$ and $\bullet$ also as unary operations on $\TSS$ with $\ominus \ominus = \oplus$, $\bullet \ominus = \bullet \bullet = \ominus \bullet = \bullet$ and where $\oplus$ just acts as identity.
The fact that $\TSS$ with these operations forms a commutative semiring justifies the name introduced above.
We also point out that balanced numbers in the symmetrized semiring are equivalent to the use of a multivalued addition as in the theory of hyperfields~\cite{ConnesConsani:2011,BakerBowler:2019}. More precisely, the addition in the symmetrized semiring is equivalent to the multivalued addition in the \emph{real tropical hyperfield} discussed in~\cite{JellScheidererYu:2018,Viro:2010}.

\begin{example}
  One has $2 \odot \ominus 1 = \ominus 3$, $(0 \oplus \ominus 0) \odot \ominus -1 = \bullet 0 \odot \ominus -1 = \bullet -1$, $-1 \odot -1 = -2$, $\ominus 1 \odot \ominus 1 = 2$. 
\end{example}

\bigskip

We recall some relations which serve to order the symmetrized semiring $\TSS$. 
For $x,y \in \TSS$ we set 
\begin{equation} \label{eq:extended+partial+orders}
  \begin{aligned}
    x > y \quad &\Leftrightarrow \quad x \ominus y \in \Tgt  \\
    x \teq y \quad &\Leftrightarrow \quad x \ominus y \in \Tgt \cup \Tzero \\
    x \geq y \quad &\Leftrightarrow \quad x > y \vee x = y \\  
  \end{aligned}
\end{equation}

For $a \in \TSS$, we set 
\begin{equation} \label{eq:incomparable+elements}
\Uncomp(a) = \begin{cases}
\left[\ominus |a|, |a|\right] = \{ x \in \TTpm \mid \ominus |a| \leq x \leq |a| \} & \text{ for } a \in \Tzero \\
\{a\} & \text{ else }
\end{cases}
\enspace .
\end{equation}
 We extend this to vectors by setting $\Uncomp(v) = \prod_{i \in [d]} \Uncomp(v_i)$.

\begin{example}
 One has $2 > \ominus 3$
 but $\bullet 4$ and $\ominus 3$ are incomparable via `$>$' since $\bullet 4 \ominus \ominus 3 = \bullet 4 \oplus 3 = \bullet 4 = \ominus 4 \bullet 4$.
 Though it holds $\bullet 4 \teq \ominus 3$ and $\ominus 3 \teq \bullet 4$ and $3 \teq \bullet 4$ showing, e.g., that `$\teq$' is not anti-symmetric. 
\end{example}

\bigskip

The signed tropical numbers are equipped with the order topology induced by the strict order $<$.
With this, $\TTpm$ is homeomorphic to $\RR$ with the usual order topology via
\[
\slog \colon \RR \to \TTpm, \quad \slog(x) =
\begin{cases}
\log(x) & \text{ for } x > 0 \\
\ominus\log(|x|) & \text{ for } x < 0 \\
\Zero & \text{ for } x = 0 
\end{cases}
\enspace .
\]
This extends to $\TTpm^d$ via the product topology, in particular $\TTpm^d$ is homeomorphic to $\RR^d$. 
Hence, we can use all topological properties of $\RR^d$ also for vectors of signed tropical numbers. 

Recall that a partial order on some set $S$ is \emph{dense} if for every two elements $x,y \in S$ such that $x < y$ there exists $z \in S$ that satisfies $x < z < y$.
With this, the order $<$ on $\TTpm$ is dense.

\subsection{Halfspaces}

The different versions of halfspaces form the building blocks for our convexity notions. 

\begin{figure}[tbh]
\includegraphics{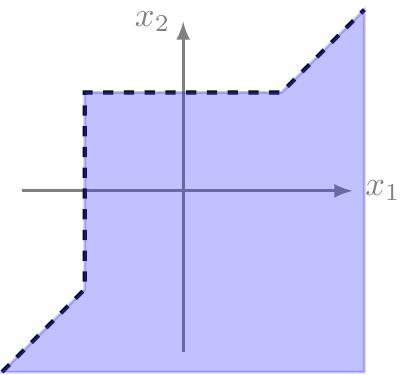} \qquad
\includegraphics{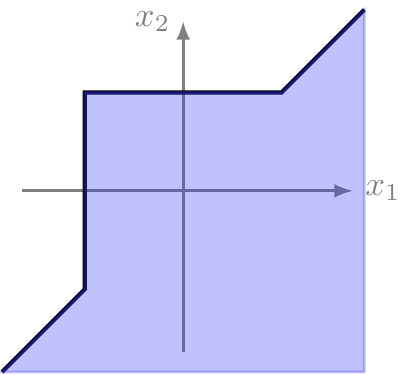}
\label{fig:halfspaces}
\caption{An open and a closed signed tropical halfspace. }
\end{figure}

\begin{definition}
  For a vector $(a_0,a_1,\dots,a_d) \in \TTpm^{d+1}$ such that $(a_1,\dots,a_d) \neq \Zero$ we define the \emph{signed (affine) tropical hyperplane} by
  \begin{equation}
  \HC(a) = \SetOf{x \in \TTpm^d}{a \odot \begin{pmatrix} 0 \\ x \end{pmatrix} \in \Tzero} \enspace ,
\end{equation}
  the \emph{open signed (affine) tropical halfspace} by
\begin{equation}
  \HC^+(a) = \SetOf{x \in \TTpm^d}{a \odot \begin{pmatrix} 0 \\ x \end{pmatrix} \in \Tgt} \enspace ,
\end{equation}
the \emph{closed signed (affine) tropical halfspace} by
\begin{equation}
  \Hclosed^+(a) = \SetOf{x \in \TTpm^d}{a \odot \begin{pmatrix} 0 \\ x \end{pmatrix} \in \Tgt \cup \Tzero} \enspace ,
\end{equation}
and the \emph{semi-closed signed (affine) tropical halfspace} by
\begin{equation}
  \Hsemi^+(a) = \SetOf{x \in \TTpm^d}{a \odot \begin{pmatrix} 0 \\ x \end{pmatrix} \in \Tgeq} \enspace .
\end{equation}
If $a_0 = \Zero$, we call a tropical halfspaces \emph{linear} instead of affine. 

We denote $\HC^-(a) = \HC^+(\ominus a)$ and $\Hclosed^-(a) = \Hclosed^+(\ominus a)$.
Furthermore, given $\subs \subseteq [d]$, we say that a tropical halfspace $\Hclosed^+(a)$ is \emph{of type $\subs$} if $\subs = \SetOf{i \ge 1}{a_{i} \in \Tgt}$.
\end{definition}

To express $\Hclosed^+(a)$ using a more classical notation, we recall that $\max \emptyset = -\infty$.
Then, we have $x \in \Hclosed^+(a)$ if and only if
\begin{equation}\label{eq:hspace}
\max\bigl( a_{0}^{+}, \max_{\tsgn(x_i) = \tsgn(a_i)}(|a_i| + |x_i|) \bigr) \ge \max\bigl( a_{0}^{-}, \max_{\tsgn(x_i) \neq \tsgn(a_i)}(|a_i| + |x_i|) \bigr) \, ,
\end{equation}
where $(a_{0}^{+}, a_{0}^{-}) = (a_0, -\infty)$ for $a_0 \in \Tgeq$ and $(a_{0}^{+}, a_{0}^{-}) = (-\infty, |a_0|)$ otherwise.
We have $x \in \HC^+(a)$ if and only if the inequality in \cref{eq:hspace} is strict.
The following lemma gathers the basic topological properties of tropical halfspaces.

\begin{lemma}\label{le:hspace_topo}
The set $\Hclosed^+(a)$ is closed, its interior is equal to $\HC^+(a)$, and the closure of $\HC^+(a)$ is equal to $\Hclosed^+(a)$.
\end{lemma}
\begin{proof}
  \Cref{eq:hspace} shows that the restriction of $\Hclosed^+(a)$ to any closed orthant of $\TTpm^d$ is closed.
  Therefore, $\Hclosed^+(a)$ is closed and $\HC^+(a) = \TTpm^d \setminus \Hclosed^-(a)$ is open.
  Also, \cite[Lemma~5.2]{LohoVegh:2020} shows that $\Hclosed^+(a)$ is the closure of $\HC^+(a)$.
  In particular, $\HC^+(a)$ is included in the interior of $\Hclosed^+(a)$.
  To finish, we note that the interior of $\Hclosed^+(a)$ cannot contain any point that belongs to $\HC(a) \subseteq \Hclosed^-(a)$ because $\Hclosed^-(a)$ is the closure of $\HC^-(a)$ and $\Hclosed^+(a) \cap \HC^-(a) = \emptyset$.
\end{proof}

\subsection{Basic calculations in $\TSS$}\label{sec:basic_calculation}

We collect and extend basic properties from \cite{LohoVegh:2020}. 

\begin{lemma}[{\cite[Lemma~2.6]{LohoVegh:2020}}] \label{lem:basic+properties+teq}
  Let $a,b,c,d \in \TSS$.
  \begin{enumerate}[(a)]
  \item \label{item:teq+other+side} $a \oplus c \teq b \Leftrightarrow a \teq b \ominus c$
  \item \label{item:teq+summing} $a \teq b \quad \wedge \quad c \teq d \quad \Rightarrow \quad a \oplus c \teq b \oplus d$.
  \item \label{item:restricted+signed+transitivity} If $c \in \TTpm$, then $b \teq c$ and $c \teq a$ imply $b \teq a$.
  \item \label{item:teq+affine+monotony} $a \teq b$ implies $c \odot a \oplus d \teq c \odot b \oplus d$ for $c \in \Tgeq$ and $c \odot b \oplus d \teq c \odot a \oplus d$ for $c \in \Tleq$.
  \end{enumerate}
\end{lemma}

\begin{lemma} \label{lem:stronger+cancellation+teq}
  Let $a,b,c \in \TSS$. Then,
  \begin{equation*}
    a \ominus c \teq \Zero \wedge b \oplus c \teq \Zero \Rightarrow a \oplus b \oplus c \teq \Zero \wedge a \oplus b \ominus c \teq \Zero .
  \end{equation*}
\end{lemma}
\begin{proof}
  We distinguish two cases.

  \textbf{Case 1} ($c \in \Tzero$) Adding the equivalent relations $a \teq c$ and $b \teq \ominus c$ side-wise yields $a \oplus b \teq \bullet c$.
  This is equivalent to $a \oplus b \bullet c \teq \Zero$ and shows the required as in this case $\bullet c = \ominus c = c$.

  \textbf{Case 2} ($c \in \TTpm$) Without loss of generality, we can assume that $c \geq \Zero$. With $a \teq c$ this implies $|a| \geq |c|$.
  Combining $a \teq c$ and $c \teq \ominus b$ yields $a \oplus b \teq \Zero$ in this case.
  With
  \begin{equation*}
    |a \oplus b| = |a| \oplus |b| \geq |a| \geq |c| 
  \end{equation*}
  we get $a \oplus b \teq c \geq \ominus c$ by checking the two possibilities $a \oplus b \in \TTpm$ or $a \oplus b \in \Tzero$. 
\end{proof}

\begin{lemma} \label{lem:iterated-sums-uncomp}
  Let $u,v,w \in \TTpm$, $q \in \Uncomp(v \oplus w)$, $p \in \Uncomp(u \oplus v)$. 
  \begin{enumerate}[(a)]
  \item $\Uncomp(u \oplus q) \subseteq \Uncomp(u \oplus v \oplus w)$ and $\Uncomp(p \oplus w) \subseteq \Uncomp(u \oplus v \oplus w)$,
  \item $\Uncomp(u \oplus q) \cap \Uncomp(p \oplus w) \neq \emptyset$. 
  \end{enumerate}
\end{lemma}
\begin{proof}
  For (a): 
  By commutativity, it suffices to prove the first inclusion.
  For $v \oplus w \in \TTpm$, we have $q = v \oplus w$ and therefore $\Uncomp(u \oplus q) = \Uncomp(u \oplus v \oplus w)$.
  Otherwise, the claim follows from the definition of $\Uncomp(.)$ using
  \begin{equation*}
    |u \oplus q| = |u| \oplus |q| \leq |u| \oplus |u \oplus w| = |u \oplus v \oplus w| 
  \end{equation*}
  by distinguising $|u| > |v \oplus w|$ and $|u| \leq |v \oplus w|$. 
  
  \smallskip

  For (b):
  
  \textbf{Case 1} ($\Uncomp(v \oplus w)$ or $\Uncomp(u \oplus v)$ is singleton.)
  Assume without loss of generality that $\Uncomp(v \oplus w)$ is a singleton.
  This implies $\Uncomp(u \oplus q) = \Uncomp(u \oplus v \oplus w) \supseteq \Uncomp(p \oplus w)$, where the inclusion follows from (a).
  
  \textbf{Case 2} ($v \oplus w$ and $u \oplus v$ are balanced.)
  Here, we have $u = \ominus v = w$.
  This means we have either $p = \ominus w$ and, hence, $\Uncomp(p \oplus w) = \Uncomp(u \oplus v \oplus w)$, or $p \oplus w = w = u$.
  The same applies to $q$ and $u$.
  Therefore, in each of the four combinations of the two possibilities, one has a nonempty intersection. 
\end{proof}

As $\Uncomp(.)$ is defined componentwise, we obtain a higher-dimensional extension of~\Cref{lem:iterated-sums-uncomp}(b). 

\begin{corollary} \label{cor:iterated-sums-uncomp-dim}
    For $u,v,w \in \TTpm^d$ and $q \in \Uncomp(v \oplus w)$, $p \in \Uncomp(u \oplus v)$, one gets $\Uncomp(u \oplus q) \cap \Uncomp(p \oplus w) \neq \emptyset$. 
\end{corollary}

\begin{lemma}[{\cite[Lemma~3.5(b)]{LohoVegh:2020}}]
  If $a \in \Uncomp(x)$, $b\in \Uncomp(y)$, and $c\in \TTpm$, then
$\Uncomp(c\odot a\oplus b)\subseteq \Uncomp(c\odot x\oplus y)$.
\end{lemma}

\begin{observation} 
  Let $a,b,\mu\in \TTpm$ and $s,t \in \TSS$ with $a \in \Uncomp(s)$ and $b \in \Uncomp(t)$. 
  Then $\mu \odot a \in \Uncomp(\mu \odot s)$ and $a \oplus b \in \Uncomp(s \oplus t)$. 
\end{observation}

\subsection{Left sum}

The following notion of one-sided addition of tropical signed numbers is crucial for our study of convexity over signed tropical numbers.

\begin{definition} \label{def:left+sum}
Given two tropical numbers $x,y \in \TSS$ we define their \emph{left sum} as
\[
x \lplus y = \begin{cases}
x &\text{if $|x| = |y|$,} \\
  x \oplus y &\text{otherwise.} 
\end{cases}
\]
We extend this definition to vectors $x, y \in \TSS^d$ by putting $(x \lplus y)_k = x_k \lplus y_k$ for every $k \in [d]$. 
\end{definition}

\begin{remark} \label{rem:relation-left-sum-OM}
 For the special case of vectors with entries in $\{\Zero,0,\ominus 0\}$, the left sum operation is just the \emph{composition} of sign vectors in an oriented matroid \cite{BLSWZ:1993}.
 The generalization to signed vectors with real entries already appears for the description of \emph{real Bergman fans} \cite{Celaya:2019} and even more generally in the context of matroids over tracts \cite[Section 6.2]{Anderson:2019}.
 From this point of view, later on in \cref{prop:local-generation-TC-hull} we will see an extension of the elimination property of oriented matroids.
 We will discuss this connection further in Section~\ref{sec:oriented-matroids}. 
\end{remark}

The left sum operation is not commutative as one can see on the example $0 \oplus (\ominus 0) = 0 \neq \ominus 0 = (\ominus 0) \oplus 0$. 
Nevertheless, it is associative and compatible with the order on $\TTpm$.
We collect some useful properties in the next technical statements. 

\begin{observation} \label{obs:basic-properties-lsum}
  For every $x,y,z \in \TTpm^d$ we have
  \begin{enumerate}[(a)]
  \item $(x \lplus y) \lplus z = x \lplus (y \lplus z)$
  \item $|x \lplus y| = |x| \oplus |y|$.
  \end{enumerate}
\end{observation}

\begin{lemma}\label{le:lsum_order}
  Let $a_1,b_1,a_2,b_2 \in \TTpm$.
  \begin{enumerate}[(a)]
  \item If $a_1 \le b_1$ and $a_2 \le b_2$, then $a_1 \lplus a_2 \le b_1 \lplus b_2$.
  \item if $a_1 < b_1$ and $a_2 < b_2$, then $a_1 \lplus a_2 < b_1 \lplus b_2$.
  \item if $a_1 \ge b_1$ and $a_2 \ge b_2$, then $a_1 \lplus a_2 \ge b_1 \lplus b_2$
  \item if  $a_1 > b_1$ and $a_2 > b_2$, then $a_1 \lplus a_2 > b_1 \lplus b_2$.
  \end{enumerate}
\end{lemma}
\begin{proof}
  It is enough to prove (a) and (b) as (c),(d) follow by replacing $(a_1,a_2,b_1,b_2)$ with $(\ominus a_1, \ominus a_2, \ominus b_1, \ominus b_2)$.
  Hence, we suppose that $a_1 \le b_1$ and $a_2 \le b_2$ (or $a_1 < b_1$ and $a_2 < b_2$). 
  Let $x = a_1 \lplus a_2$ and $y = b_1 \lplus b_2$.

  \smallskip
  
  If $(x,y) \in \{(a_1,b_1), (a_2, b_2)\}$, then the claim is trivial.
  
  If $(x,y) = (a_1, b_2)$, then we have $|b_2| \ge |b_1|$ and $|a_1| \ge |a_2|$.
  If $b_2 \ge \Zero$, then $b_2 \ge b_1 \ge a_2$ (or $b_2 > a_1$ if $a_1 < b_1$).
  If $b_2 < \Zero$, then $a_2 \le b_2 < \Zero$ implies that $|a_1| \ge |a_2| \ge |b_2| \ge |b_1|$.
  Therefore, $a_1 \le b_1$ implies that $a_1 \le \Zero$.
  Since $|a_1| \ge |a_2|$, we get $a_1 \le a_2 \le b_2$ (or $a_1 < b_2$ if $a_2 < b_2$).
  
  If $(x,y) = (a_2, b_1)$, then we have $|a_2| \ge |a_1|$, $|b_1| \ge |b_2|$, and an analogous proof as above applies.
\end{proof}

\begin{lemma} \label{lem:cancel-inequalities-left-sum}
  Fix $a,b \in \TTpm$ with $\tsgn(a) = \tsgn(b)$ and let $x_1^{(1)},\dots,x_n^{(1)} \in \TTpm$, $x_1^{(2)},\dots,x_n^{(2)} \in \TTpm$ be such that
  \[
  a \leq x_1^{(1)} \lplus \dots \lplus x_n^{(1)} \quad \text{ and } \quad b \leq x_1^{(2)} \lplus \dots \lplus x_n^{(2)} \;.
  \]
  Then
  \[
  a \oplus b \leq w_1 \lplus \dots \lplus w_n
  \]
  for all $w_1 \in \Uncomp(x_1^{(1)} \oplus x_1^{(2)}), \dots, w_n \in \Uncomp(x_n^{(1)} \oplus x_n^{(2)})$.

  By multiplying with $\ominus 0$, the implication also holds if one replaces `$\leq$' by `$\geq$'.  
\end{lemma}
\begin{proof}
  Let $p,q \in [n]$ be the smallest indices such that $x_1^{(1)} \lplus \dots \lplus x_n^{(1)} = x_p^{(1)}$ and $x_1^{(2)} \lplus \dots \lplus x_n^{(2)} = x_q^{(2)}$.
  Then \cref{le:lsum_order} implies that $a \lplus b \leq x_p^{(1)} \lplus x_q^{(2)}$ and $b \lplus a \leq x_q^{(2)} \lplus x_p^{(1)}$.
  Due to $\tsgn(a) = \tsgn(b)$, this yields $a \oplus b \leq \min\left(x_p^{(1)} \lplus x_q^{(2)}, x_q^{(2)} \lplus x_p^{(1)}\right)$.
  
  By definition, for any $s,t \in \TTpm$, the interval $\Uncomp(s \oplus t)$ has the boundary points $s \lplus t$ and $t \lplus s$.
  There is a choice of orders $\tau_1,\dots,\tau_n \in \Sym(2)$ which minimizes the expression $\biglplus \left(x_i^{(\tau_i(1))} \lplus x_i^{(\tau_i(2))}\right)$.
  By the former observation and \cref{le:lsum_order} it is smaller than $w_1 \lplus \dots \lplus w_n$ for all $w_1 \in \Uncomp(x_1^{(1)} \oplus x_1^{(2)}), \dots, w_n \in \Uncomp(x_n^{(1)} \oplus x_n^{(2)})$. 
  Let $x_j^{(\ell)}$ for $j \in [n]$ and $\ell \in \{1,2\}$ be the entry which defines the value of this expression, which means the first summand with maximal absolute value.
  If $\ell = 1$ then $j = p$, otherwise $j = q$.
  In either case, we get $\biglplus \left(x_i^{(\tau(1))} \lplus x_i^{(\tau(2))}\right) = x_p^{(1)} \lplus x_q^{(2)}$ or $\biglplus \left(x_i^{(\tau(1))} \lplus x_i^{(\tau(2))}\right) = x_q^{(2)} \lplus x_p^{(1)}$.
  In particular, we obtain
  \[
  w_1 \lplus \dots \lplus w_n \geq \biglplus \left(x_i^{(\tau(1))} \lplus x_i^{(\tau(2))}\right) \geq \min\left(x_p^{(1)} \lplus x_q^{(2)}, x_q^{(2)} \lplus x_p^{(1)}\right) \geq a \oplus b \;. \qedhere
  \]
\end{proof}

\begin{example}
Let $a = 0$, $b = \ominus 0$, $x^{(1)}_1 = \Zero$, $x^{(1)}_2 = 0$, $x^{(2)}_1 = \ominus 0$, $x^{(2)}_2 = \Zero$, $w_1 = \ominus 0$, $w_2 = 0$. Then $a \le x^{(1)}_1 \lplus x^{(1)}_2$, $b \le x^{(2)}_1 \lplus x^{(2)}_2$, $w_1 = x^{(1)}_1 \oplus x^{(2)}_1$, $w_2 = x^{(1)}_2 \oplus x^{(2)}_2$, but $0 = a \lplus b > w_1 \lplus w_2 = \ominus 0$. This example shows that the assumption $\tsgn(a) = \tsgn(b)$ cannot be omitted in the previous lemma. 
\end{example}

\section{Flavors of signed convexity}

In~\cite{LohoVegh:2020}, a notion of convexity for signed tropical numbers was introduced. 
We revisit this notion under the name \emph{TO-convexity} and give several new insights. 
Furthermore, we establish a second notion of convexity for signed tropical numbers in an equally natural way.

Recall that a \emph{cone} is a set $X \subseteq \TTpm^d$ such that $\lambda \odot X \subseteq X$ for all $\lambda \in \Tgt$.
Based on the two convexity notions we will also consider cones over the respective convex sets. 

We will often identify a finite set with the columns of a matrix and vice versa. 

\subsection{TO-convexity}

We collect several results from \cite{LohoVegh:2020}. 
Note that TO-convexity appears under the name `signed tropical convexity' there. 

The \emph{TO-convex hull} of a matrix $A \in \TTpm^{d \times n}$ is 
\begin{equation}\label{eq:hull+U}
   \shull(A) =\bigcup \SetOf{\Uncomp(A\odot x)}{ x\in \Tgeq^n, \bigoplus_{j \in [n]} x_j = 0 } \subseteq \TTpm^d\enspace .
\end{equation}
A set $M \subseteq \TTpm^d$ is \emph{TO-convex} if $\shull(T) \subseteq M$ for all finite $T \subseteq M$. 
The hull operator extends by setting
\begin{equation*}
  \begin{aligned}
    \shull(M) &= \bigcap_{M \subseteq S, S \text{ TO-convex}} S  \\
    &= \bigcup_{T \subseteq M, T \text{ finite}} \shull(T) \enspace . 
  \end{aligned}
\end{equation*}

For a vector $s \in \TSS^d$ with possibly balanced entries, the set $\Uncomp(s) \subset \TTpm^d$ can be seen as an axis-parallel hypercube.
If $s$ has $k$ entries in $\Tzero \setminus \{\Zero\}$, we say that the hypercube $\Uncomp(s)$ has \emph{dimension} $k$.
Furthermore, we say that $\Uncomp(s')$ is a \emph{face} of $\Uncomp(s)$ if $s'_k = s_k$ for every $k$ such that $s_k \in \TTpm$ and $|s'_k| = |s_k|$ for every other $k$.
A face of dimension zero is called a \emph{vertex} of $\Uncomp(s)$.
In particular, a hypercube of dimension $k$ has $2^k$ vertices.

We recall a crucial property of TO-convexity. 
\begin{lemma}[{\cite[Proposition~3.6]{LohoVegh:2020}}] \label{le:TO-hull-intervals}
  A subset $U \subseteq \TTpm^d$ is TO-convex if and only if $\shull(\{p,q\})$ is contained in $U$ for all $p,q \in U$. 
\end{lemma}

The TO-convex hull of a finite set can also be given as intersection of open tropical halfspaces.
This is the origin of the name `TO-convex' derived from `tropical open'.  
Furthermore, it motivates the definition of the TC-convex hull (`tropical closed') in \cref{sec:TC-convexity}. 

\begin{theorem}[{\cite[Theorem~5.1]{LohoVegh:2020}}]
  The TO-convex hull of a finite subset $M \subseteq \TTpm^d$ equals the intersection of all open halfspaces containing $M$, i.e.,
  \begin{equation*}
    \shull(M) = \bigcap_{M \subset \HC^+(a)} \HC^+(a)
  \end{equation*}
\end{theorem}

\subsection{TC-convexity}
\label{sec:TC-convexity}

While the structure of TO-convexity is heavily linked to open tropical halfspaces, we introduce TC-convexity based on closed tropical halfspaces.
For technical reasons, we first only define the TC-convex hull of two points, and then extend it to arbitrary sets using these TC-convex line segments.
Later it turns out that, for an arbitrary finite set, this is actually the same as just taking the intersection  of all closed tropical halfspaces containing it.

For basics on general convexity, we refer to \cite{VanDeVel:1993}. 
We will mainly rely on the notion of a \emph{convexity structure} \cite[\S 1.1]{VanDeVel:1993}.
In particular, we are interested in convexity structures induced by an \emph{interval operator} \cite[\S 4.1]{VanDeVel:1993}. 

\begin{definition} \label{def:weak+hull+segment}
  We define the \emph{TC-convex hull} of two points $x,y \in \TTpm^d$ as
    \begin{equation}
    \whull(x,y) = \bigcap_{x,y \in \Hclosed^+(a)} \Hclosed^+(a) \enspace .
  \end{equation}
\end{definition}

Sets of this form are \emph{TC-convex intervals}. 

\begin{definition} \label{def:TC-convex-whull}
  We say that a set $X \subset \TTpm^d$ is \emph{TC-convex} if $\whull(x,y) \subseteq X$ for all $x,y \in X$.
\end{definition}

We have the following desirable properties. 

\begin{corollary} \label{cor:basic-properties-TC-convex}
  The TC-convex sets form a convexity structure.
  In particular, TC-convex sets are closed by intersection and nested union. 
\end{corollary}

The definition of TC-convexity directly gives two important classes of TC-convex sets. 

\begin{corollary} \label{cor:hspace_weakly_convex}
  A closed signed affine tropical halfspace and a signed affine tropical hyperplane is TC-convex.
  TC-convex intervals, i.e. sets of the form $\whull(x,y)$, for $x,y \in \TTpm^d$, are TC-convex. 
\end{corollary}
\begin{proof}
 Right from the definition we get that closed signed affine tropical halfspaces are TC-convex.
 The TC-convexity of the other sets follows from the intersection property. 
\end{proof}

\Cref{cor:basic-properties-TC-convex} allows us to extend the hull operator routinely to arbitrary sets $M \subseteq \TTpm^d$ by setting

\begin{equation*}
  \whull(M) = \bigcap_{M \subseteq S, S \text{ TC-convex}} S \enspace .
\end{equation*}

In particular, this convexity structure is `domain finite' by \cite[Theorem 1.3]{VanDeVel:1993} which means the following. 
\begin{corollary} \label{cor:convexity-finitary}
  For an arbitrary set $M \subseteq \TTpm^d$ we have
  \begin{equation*}
  \begin{aligned}
  \whull(M) = \bigcup_{T \subseteq M, T \textup{ finite}} \whull(T) \enspace .
  \end{aligned}
\end{equation*}
\end{corollary}

\begin{corollary} \label{cor:TO-contain-TC}
  For an arbitrary set $M \subseteq \TTpm^d$ we have $\whull(M) \subseteq \shull(M)$. 

Furthermore, each TO-convex set is also TC-convex. 
\end{corollary}
\begin{proof}
 For two point $x,y \in \TTpm^d$, \cite[Theorem~5.4]{LohoVegh:2020} implies that $\shull(x,y)$ is an intersection of finitely many closed halfspaces.
 Now, combining with the \cref{def:weak+hull+segment} yields $\whull(x,y) \subseteq \shull(x,y)$.
 Therefore, by definition of a TC-convex set via TC-convex intervals (\cref{def:TC-convex-whull}), each TO-convex set is also TC-convex.
 Hence, using the extension of the hull operators for TO-convexity and TC-convexity to not necessarily finite sets gives the first claim.  
\end{proof}

\begin{figure}[tbh]
  \begin{subfigure}[t]{0.49\textwidth}
    \centering
\includegraphics{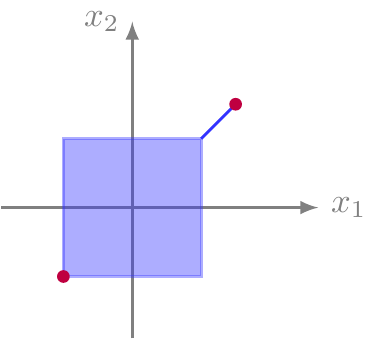}
    \caption{$\shull((0,0),(\ominus -2, \ominus -2))$}
    \label{subfig:tropical+exploded+line+segment}
  \end{subfigure} \hfill
  \begin{subfigure}[t]{0.49\textwidth}
    \centering
\includegraphics{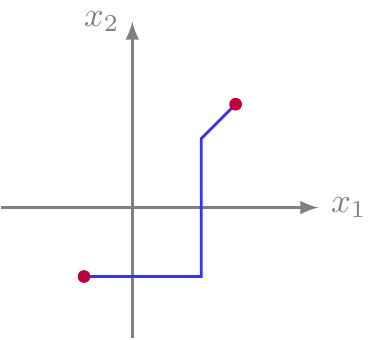}
    \caption{$\shull((0,0),(\ominus -3, \ominus -2))$}
  \end{subfigure}
  \begin{subfigure}[t]{0.49\textwidth}
    \centering
\includegraphics{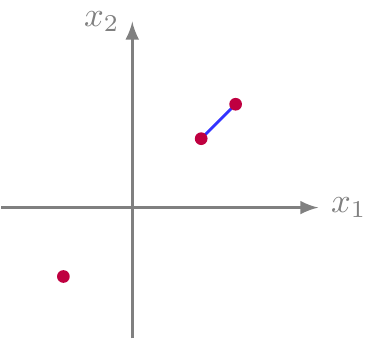}
        \caption{$\whull((0,0),(\ominus -2, \ominus -2))$}
  \end{subfigure} \hfill
  \begin{subfigure}[t]{0.49\textwidth}
    \centering
\includegraphics{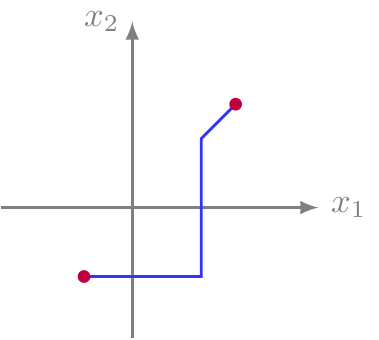}
    \caption{$\whull((0,0),(\ominus -3, \ominus -2))$}
  \end{subfigure}
  \caption{TO-convex intervals and TC-convex intervals in the plane (see \cref{ex:TO+TC-intervals}) }
  \label{fig:line+segments}
\end{figure}

To describe the TC-convex hull explicitly, we need significantly more tools than for the TO-convex hull. 

\begin{definition} \label{def:vert+faces}
Given a set of points $X = \left\{x_1, \dots, x_n\right\} \in \TTpm^d$, we define 
\[
\vertices(X) = \vertices(x_1, \dots, x_n) \coloneqq \SetOf{x_{\sigma(1)} \lplus x_{\sigma(2)} \lplus \dots \lplus x_{\sigma(n)}}{\sigma \in \Sym(n)}  \subset  \TTpm^d \, ,
\]
where $\Sym(n)$ denotes the group of permutations of $[n]$. 
Furthermore, we denote by
\[
\faces(X) \coloneqq \faces(x_1, \dots, x_n)
\]
the union of all faces of $\Uncomp(x_1 \oplus \dots \oplus x_n)$ (considering it as a hypercube) whose vertices belong to $\vertices(x_1, \dots, x_n)$.
\end{definition}

While the convex structure of cancellation for TO-convexity only depends on the balanced outcome, the non-associative structure of TC-convexity is far more subtle.
Therefore, it is not enough to apply a univariate operator like $\Uncomp( . )$ but we have to use the multivariate operator $\faces(.)$.
Though, equipped with this tool, we will be able to describe the TC-hull also in terms of an analog of convex combinations. 

Directly from the definition and \cref{obs:basic-properties-lsum} we get the next. 
\begin{corollary}\label{cor:permutation_vertices}
The set $\vertices(x_1, \dots, x_n)$ is a subset of vertices of the hypercube $\Uncomp(x_1 \oplus \dots \oplus x_n)$. In particular, it contains at most $2^d$ points.
\end{corollary}

\begin{example}
  If we only consider two points, there are essentially three cases.
  We illustrate them on small examples.
  In the first case, the sum does not have a balanced entry. 
  \begin{equation*}
    \vertices\left(\begin{pmatrix} 0 \\ \ominus 0 \end{pmatrix},\begin{pmatrix} 1 \\  -1 \end{pmatrix}\right) =
    \left\{\begin{pmatrix} 1 \\ \ominus 0 \end{pmatrix}\right\} =
    \faces\left(\begin{pmatrix} 0 \\ \ominus 0 \end{pmatrix},\begin{pmatrix} 1 \\  -1 \end{pmatrix}\right) \enspace .
  \end{equation*}

  In the second case, we have 
  
  \begin{equation*}
    \vertices\left(\begin{pmatrix} 0 \\ \ominus 0 \end{pmatrix},\begin{pmatrix} 1 \\  0 \end{pmatrix}\right) =
    \left\{\begin{pmatrix} 1 \\ \ominus 0 \end{pmatrix},\begin{pmatrix} 1 \\  0 \end{pmatrix}\right\}
  \end{equation*}
  which yields 
  \begin{equation*}
    \faces\left(\begin{pmatrix} 0 \\ \ominus 0 \end{pmatrix},\begin{pmatrix} 1 \\  0 \end{pmatrix}\right) =
    \left\{\begin{pmatrix} 1 \\ s \end{pmatrix} \colon s \in [\ominus 0, 0]\right\} \enspace .
  \end{equation*}

  In the last case, there is more than one balanced entry
  \begin{align*}
    \vertices\left(\begin{pmatrix} 0 \\ \ominus 0 \end{pmatrix},\begin{pmatrix} \ominus 0 \\  0 \end{pmatrix}\right) =
    \left\{\begin{pmatrix} 0 \\ \ominus 0 \end{pmatrix},\begin{pmatrix} \ominus 0 \\  0 \end{pmatrix}\right\} =
    \faces\left(\begin{pmatrix} 0 \\ \ominus 0 \end{pmatrix},\begin{pmatrix} \ominus 0 \\  0 \end{pmatrix}\right) \enspace .
  \end{align*}

\end{example}

With this terminology, we can give a representation of the TC-convex hull of two points analogous to the representation by convex combinations. 

\begin{proposition}\label{le:weak_inter}
  We have
  \begin{equation}
  \whull(x,y) = \bigcup\SetOf{\faces(\lambda \odot x, \mu \odot y)}{\lambda, \mu \in \Tgeq, \lambda \oplus \mu = 0} \enspace .
  \end{equation}
\end{proposition}

\begin{proof}
 By \cref{cor:TO-contain-TC}, we know that $\whull(x,y) \subseteq \shull(x,y)$.
  With~\cref{eq:hull+U}, this implies that it suffices to consider which part of $\Uncomp(\lambda \odot x \oplus \mu \odot y)$ is contained in $\whull(x,y)$ for each pair $\lambda, \mu \in \Tgeq, \lambda \oplus \mu = 0$.
  Note that, for different such pairs $\lambda_1,\mu_1$ and $\lambda_2,\mu_2$, the sets $\Uncomp(\lambda_1 \odot x \oplus \mu_1 \odot y)$ and $\Uncomp(\lambda_2 \odot x \oplus \mu_2 \odot y)$ are either disjoint or identical.
  Therefore, it is enough to consider the sets $\Uncomp(\lambda \odot x \oplus \mu \odot y)$ for a fixed pair $\lambda, \mu$.  
  Hence, fixing such a pair, we distinguish three cases.

  For $a = (a_0,\bar{a}) \in \TTpm^{d+1}$ let $\Hclosed^+(a)$ be an affine halfspace containing $x$ and $y$.
  This means that
  \begin{equation} \label{eq:two+points+halfspace}
    \lambda \odot a_0 \oplus \bar{a} \odot \lambda \odot x \teq \Zero \text{ and } \mu \odot a_0 \oplus \bar{a} \odot \mu \odot y \teq \Zero \enspace .
  \end{equation}
  Recall that at least one of $\lambda, \mu$ is $0$. 
  
  \textbf{Case 1} ($\lambda \odot x \oplus \mu \odot y$ has no balanced entry. )
  Adding the relations in~\cref{eq:two+points+halfspace} yields 
  \begin{equation*}
    a_0 \oplus \bar{a} \odot (\lambda \odot x \oplus \mu \odot y) \teq \Zero \enspace .
  \end{equation*}
  That already concludes this case, where indeed $\lambda \odot x \oplus \mu \odot y = \lambda \odot x \lplus \mu \odot y =  \mu \odot y \lplus \lambda \odot x$. 
  
  \textbf{Case 2} ($\lambda \odot x \oplus \mu \odot y$ has exactly one balanced entry. )
  Without loss of generality, we assume that the $d$-th entry is balanced.
  Reformulating~\cref{eq:two+points+halfspace} implies
  \begin{equation*}
    \begin{aligned}
      \lambda \odot a_0 \oplus \bigoplus_{\ell=1}^{d-1} a_{\ell} \odot \lambda \odot x_{\ell} &\teq \ominus \lambda \odot a_d \odot x_d \\
      \mu \odot a_0 \oplus \bigoplus_{\ell=1}^{d-1} a_{\ell} \odot \mu \odot y_{\ell} &\teq \ominus \mu \odot a_d \odot y_d \enspace .
    \end{aligned}
  \end{equation*}
  Using $b := \lambda \odot x_d =  \ominus \mu \odot y_d$, the side-wise addition of the relations yields
    \begin{equation*}
    \begin{aligned}
      a_0 \oplus \bigoplus_{\ell=1}^{d-1} a_{\ell} \odot (\lambda \odot x_{\ell} \oplus \mu \odot y_{\ell}) &\teq a_d \oplus \bullet b \enspace .
    \end{aligned}
    \end{equation*}
    Therefore, each point in
    \begin{equation*}
      \left\{(\lambda \odot x_{1} \oplus \mu \odot y_{1}, \dots, \lambda \odot x_{d-1} \oplus \mu \odot y_{d-1},s) \colon s \in [\ominus |b|, |b|]\right\}
    \end{equation*}
    is contained as claimed.
    Furthermore, the latter set is indeed the line between $\lambda \odot x \lplus \mu \odot y$ and $\mu \odot y \lplus \lambda \odot x$. 
    
  \textbf{Case 3} ($\lambda \odot x \oplus \mu \odot y$ has more than one balanced entry. )
  Without loss of generality, exactly the coordinates $k+1$ to $d$ of $\lambda \odot x \oplus \mu \odot y$ are balanced.
  Then~\cref{eq:two+points+halfspace} amounts to
  \begin{equation*}
    \begin{aligned}
      \left(\lambda \odot a_0 \oplus \bigoplus_{\ell=1}^{k} a_{\ell} \odot \lambda \odot x_{\ell}\right) \oplus \left(\bigoplus_{\ell=k+1}^{d} a_{\ell} \odot \lambda \odot x_{\ell}\right) &\teq \Zero\\
      \left(\mu \odot a_0 \oplus \bigoplus_{\ell=1}^{k} a_{\ell} \odot \mu \odot y_{\ell}\right) \oplus \left(\bigoplus_{\ell=k+1}^{d} a_{\ell} \odot \mu \odot y_{\ell} \right) &\teq \Zero \enspace .
    \end{aligned}
  \end{equation*}
  Using \cref{lem:stronger+cancellation+teq} in the same way as we did for Case 2, we obtain that $\lambda \odot x \lplus \mu \odot y$ and $\mu \odot y \lplus \lambda \odot x$ are also contained in $\Hclosed^+(a)$.

  Up until now, we have proven the first inclusion. Now, we set $b_{\ell} = \lambda \odot x_{\ell} = \ominus \mu \odot y_{\ell}$ for $\ell \in [d] \setminus [k]$. 
  We look at the halfspaces
  \begin{equation} \label{eq:separating+halfs+box}
    \Hclosed^+(\ominus 0,\Zero,\dots,\Zero,\epsilon_1 \odot b_{k+1}^{\odot -1},\dots,\epsilon_{d-k} \odot b_d^{\odot -1}) 
  \end{equation}
  for each $\epsilon \in \{\ominus,\oplus\}^{d-k} \setminus \{(\ominus,\dots,\ominus),(\oplus,\dots,\oplus)\}$.
  By putting $x$ and $y$ in the corresponding relation, we see that they are both contained.    

  Now let $w_1 =  \lambda \odot x \lplus \mu \odot y$ and $w_2 = \mu \odot y \lplus \lambda \odot x$.
  Note that $w_1$ and $x$ have the same sign pattern on the coordinates $k+1$ up to $d$, and the same for $w_2$ and $y$, respectively. 
  We pick any point 
  \begin{equation*}
    z \in \Uncomp(\lambda \odot x_1 \oplus \mu \odot y_1,\dots, \lambda \odot x_k \oplus \mu \odot y_k,\bullet b_{k+1},\dots, \bullet b_d) \setminus \{w_1,w_2\}
  \end{equation*}
  
  \textbf{Case 3a} ($z$ has the same sign pattern as $w_1$ or $w_2$. )
  Without loss of generality, we assume that $z$ has the same sign pattern as $w_1$. 
  Then there is a coordinate of $z$, say the $(k+1)$st, such that $|z_{k+1}| < |\lambda \odot x_{k+1}| = |\mu \odot y_{k+1}|$.
  We let
  \begin{equation*}
    (\epsilon_1,\dots, \epsilon_{d-k}) = (\tsgn(z_{k+1}),\ominus \tsgn(z_{k+2}),\dots,\ominus \tsgn(z_{d-k}))
  \end{equation*}
  As $|z_{k+1}\odot |b_{k+1}|^{\odot -1}| < 0$, we get 
  \begin{equation*}
    \ominus 0 \oplus \epsilon_1 \odot |b_{k+1}|^{\odot -1} \odot z_{k+1} \oplus \dots \oplus \epsilon_{d-k} \odot |b_d|^{\odot -1} \odot z_{d-k} < \Zero  .
  \end{equation*}
  Hence, there is a halfspace among those in~\cref{eq:separating+halfs+box} not containing $z$. 
  
  \textbf{Case 3b} ($z$ has a different sign pattern from $w_1$ and $w_2$. ) 
  We let
  \begin{equation*}
    (\epsilon_1,\dots, \epsilon_{d-k}) = (\ominus \tsgn(z_{k+1}),\dots,\ominus \tsgn(z_{d-k}))
  \end{equation*}
  be the negative of the sequence of signs. Because of the relation
  \begin{equation*}
    \ominus 0 \oplus \epsilon_1 \odot |b_{k+1}|^{\odot -1} \odot z_{k+1} \oplus \dots \oplus \epsilon_{d-k} \odot |b_d|^{\odot -1} \odot z_{d-k} < \Zero 
  \end{equation*}
  there is a halfspace among those in~\cref{eq:separating+halfs+box} not containing $z$, as these include all sign patterns except for those of $x$ and $y$. 
\end{proof}

\begin{example} \label{ex:TO+TC-intervals}
\Cref{fig:line+segments} compares TC-intervals with TO-intervals in a plane.
If we take $x = (0,0)$ and $y = (\ominus-3, \ominus-2)$, then for every $\lambda, \mu \in \Tgeq$ such that $\lambda \oplus \mu = 0$ we have $\faces(\lambda \odot x, \mu \odot y) = \Uncomp(\lambda \odot x \oplus \mu \odot y)$.
In particular, in this case the intervals coincide, $\shull(x,y) = \whull(x,y)$.
If we take $y = (\ominus -2, \ominus -2)$ instead, then the equality no longer holds.
Indeed, for $\lambda = -2$, $\mu = 0$ we have $\Uncomp(\lambda \odot x \oplus \mu \odot y) = \Uncomp(\bullet 2, \bullet 2) = [\ominus -2, -2] \times [\ominus -2, -2]$ and $\faces(\lambda \odot x, \mu \odot y) = \{(\ominus -2, \ominus -2), (-2, -2)\}$.
\end{example}

We provide a tool for a local check of convexity based on left sum and cancellation in only one coordinate.

\begin{proposition} \label{prop:local-generation-TC-hull}
  A set $M \subseteq \TTpm^{d}$ is TC-convex if and only if it is closed under the following two operations:
  \begin{enumerate}[(i)]
  \item for $x,y \in M$ and $\lambda, \mu \in \Tgeq$ with $\lambda \oplus \mu = 0$, we have
    \begin{equation*}
      \lambda \odot x \lplus \mu \odot y \in M; \tag{weighted left sum}
    \end{equation*}
  \item if $(u,v,w), (u,\ominus v, w) \in \TTpm^k \times \TTpm \times \TTpm^{d-k-1}$ for $k \in [d]_0$ are contained in $M$ then 
    \begin{equation*}
      (u,\Zero,w) \in M . \tag{local elimination}
    \end{equation*}
  \end{enumerate}
  Furthermore, we can equivalently replace $(ii)$ by a stronger property
    \begin{enumerate}
  \item[(ii')] if $(u,v,w), (u,\ominus v, w) \in \TTpm^k \times \TTpm \times \TTpm^{d-k-1}$ for $k \in [d]_0$ are contained in $M$ then 
    \begin{equation*}
      \{u\} \times [\ominus |v|, |v|] \times \{w\} \subseteq M .
    \end{equation*}
  \end{enumerate}
\end{proposition}

\begin{proof}
  By \cref{le:weak_inter}, a TC-convex set fulfills $(i)$ and $(ii')$. Moreover, $(ii')$ is stronger than $(ii)$.

  It therefore suffices to prove that the TC-convex hull of two points can be generated by $(i)$ and $(ii)$.
  So we fix two points $x,y \in M \subseteq \TTpm^{d}$.
  Let $\lambda, \mu \in \Tgeq$ with $\lambda \oplus \mu = 0$.
  By the first property, $\vertices(\lambda \odot x, \mu \odot y)$ is contained in $M$.
  If $\vertices(\lambda \odot x, \mu \odot y)$ do not form the vertices of a face of the respective hypercube, it equals $\faces(\lambda \odot x, \mu \odot y)$. 
  Otherwise, we have $\lambda \odot x = (u,v,w)$ and $\mu \odot y = (u, \ominus v, w)$ for some $u,v,w$ as in $(ii)$, so that $(u,\Zero,w) \in M$.   Taking the combinations $\rho \odot (u,v,w) \lplus (u,\Zero,w)$ and $\rho \odot (u,\ominus v,w) \lplus (u,\Zero,w)$ for $\rho \leq 0$ via the weighted left sum property yields the whole interval $\faces(\lambda \odot x, \mu \odot y)$. 
  Now, \cref{le:weak_inter} concludes the proof. 
\end{proof}

\begin{example} \label{ex:vertical-lines}
  For $(p,r),(q,r) \in \TTpm \times \TTpm^{d-1}$ with $p < q$, one gets
  \begin{equation}
    \whull(\{(p,r),(q,r)\}) = [p,q] \times \{r\} . 
  \end{equation}
  To see this, we first assume that $p < q$ have the same sign, w.l.o.g. both are positive.
  Then $[p,q] \times \{r\} = \SetOf{(p,r) \lplus \lambda \odot (q,r)}{\Zero \leq \lambda \leq 0}$. 
  Otherwise, assume that $p < \Zero$ and $|p| < q$.
  Then $(\ominus p,r) = (q - |p|) \odot (q,r) \lplus (p,r)$ and the claim follows with the local elimination property. 
  The other cases follow by suitably adapting the signs and scalars.
\end{example}

In the study of the structure of $\faces( . )$, the boxes arising from $\Uncomp( . )$ play a crucial role.
For $p, q \in \TTpm^d$ with $p \leq q$ (defined component-wise), we denote
\begin{align*}
  [p,q] = \SetOf{x \in \TTpm^d}{p_i \leq x_i \leq q_i\; \forall i \in [d]}, \\
  (p,q) = \SetOf{x \in \TTpm^d}{p_i < x_i < q_i\; \forall i \in [d]}.
\end{align*}

\begin{example} \label{ex:vertical-boxes}
  For $p,q \in \TTpm^d$ with $p \leq q$, one gets
  \begin{equation}
    \whull(\{p_1,q_1\} \times \dots \times \{p_d,q_d\}) = [p,q] .
  \end{equation}
  To see this, we can fix all but one coordinate and then iteratively use \cref{ex:vertical-lines}. 

  \smallskip

  Note that it also holds that
  \begin{equation}
    \shull(\{p_1,q_1\} \times \dots \times \{p_d,q_d\}) = [p,q] .
  \end{equation}
\end{example}

Let $C$ be a subcomplex of the faces of $[-1,1]^d$ with the property: \emph{if all vertices of a face are contained in $C$ then so is the face.}
These are the kind of subcomplexes arising from the operator $\faces( . )$ as we will conclude in \cref{cor:faces_from_ineqs}. 

For ${\sim} \in \{{\le}, {\ge}\}^d$, we define $x \sim y\; \Leftrightarrow\; \left(x_k \sim_k y_k \text{ for all }k \in [d]\right)$.

\begin{lemma} \label{lem:containment-special-subcomplex-hypercube}
  A point $z \in [-1,1]^d$ is contained in $C$ if and only if for every ${\sim} \in \{{\le}, {\ge}\}^d$ there exists a vertex $w$ of $C$ such that $z \sim w$. 
\end{lemma}
\begin{proof}
  With a point $z \in [-1,1]^d$, we associate a partition $I \cup J \cup K$ of $[d]$ indexing the coordinates of $z$ which are $+1$, $-1$ or in the open interval $(-1,1)$.
  Then the set $I \cup J$ uniquely defines the face $F$ of the cube of smallest dimension containing $z$.

  \smallskip
  
  Assume that $z$ is not contained in $C$.
  By the crucial property of $C$, there is a vertex $v$ of $F$ which is not contained in $C$.
  With $v$ we associate a vector ${\sim}^{(v)} \in \{{\le}, {\ge}\}^d$ by converting $-1$ to $\leq$ and $+1$ to $\geq$.
  Then $v$ is the unique vertex of $[-1,1]^d$ with $v \sim^{(v)} z$ because there is a unique choice of $\tau \in \{-1,1\}$ with $\tau \sim^{(v)}_k z_k$ for $k \in K$. 
  As $v$ is not a vertex of $C$ this concludes the first direction. 

  \smallskip
  
  On the other hand, if a point $z$ is contained in $C$ then also the face $F$ is contained in $C$.
  Let ${\sim} \in \{{\le}, {\ge}\}^d$ be arbitrary.
  We define a vertex $v$ of $[-1,1]^d$ by $v_{\ell} = z_{\ell}$ for $\ell \in I \cup J$ and
  \[
  v_k = \begin{cases}
    1 & \text{ if  ${\sim_k}$ equals ${\geq}$}, \\
    -1 & \text{ if ${\sim_k}$ equals ${\leq}$} 
  \end{cases}
  \]
  for $k \in K$.
  By construction, $v$ is a vertex of $F$ and, hence, of $C$.
  Furthermore, we get $v \sim z$ proving the claim. 
\end{proof}

\begin{corollary}\label{cor:faces_from_ineqs}
We have $y \in \faces(x_1,\dots,x_n)$ if and only if for every ${\sim} \in \{{\le}, {\ge}\}^d$ there exists $w \in \vertices(x_1, \dots, x_n)$ such that $y_k \sim_k w_k$ for all $k \in [d]$.
\end{corollary}
\begin{proof}
 By \cref{def:vert+faces}, the set $\faces(x_1,\dots,x_n)$ forms a subcomplex of the faces of the (stretched) hypercube $\Uncomp(x_1 \oplus \dots \oplus x_n)$ exactly with the property required of $C$ in \cref{lem:containment-special-subcomplex-hypercube}.
 As $\vertices(x_1, \dots, x_n)$ select exactly the vertices of this subcomplex, \cref{lem:containment-special-subcomplex-hypercube} gives the required equivalence. 
\end{proof}

\subsection{Representation of TC-convex sets by combinations}

We extend the representation as an analog of a convex combination from TC-convex hulls of two points to arbitrary finite sets in \cref{le:inthull_from_faces}.
Then, we refine this to a Carath{\'e}odory-type result.  
We start with two technical lemmas.
While the first relates the combination of elements with the hull arising from $\faces( . )$, the second shows how $\faces( . )$ arises from $\vertices( . )$. 
Fix finite sets $X,Y \subseteq \TTpm^{d}$. 

\begin{lemma} \label{le:adding_boxes}
  Let $s,t \in \TSS$ with $|s| = |\bigoplus X|$ and $|t| = |\bigoplus Y|$ such that $\Uncomp(s) \subseteq \faces(X)$ and $\Uncomp(t) \subseteq \faces(Y)$.
  Then $\Uncomp(s \lplus t) \subseteq \faces(X \cup Y)$.

\end{lemma}
\begin{proof}
  Note that, by definition of $s$, we have the inclusion $\Uncomp(\bullet s) \subseteq \Uncomp(\bigoplus X)$ and further, using $\Uncomp(s) \subseteq \faces(X)$, that every vertex of $\Uncomp(s)$ is a vertex of $\Uncomp(\bigoplus X)$.
  Analogously, every vertex of $\Uncomp(t)$ is a vertex of $\Uncomp(\bigoplus Y)$.
  
  Therefore as part of $\faces(\bigoplus X)$, all vertices of $\Uncomp(s)$ are of the form $\lplus X^{\sigma}$ for some ordering $\sigma$ of $X$, and analogously, all vertices of $\Uncomp(t)$ are of the form $\lplus Y^{\tau}$ for some ordering $\tau$ of $Y$.
  Now, we consider a vertex $u$ of $\Uncomp(s \lplus t)$. 
  We partition $[d]$ in two sets, $I := \{i \colon |u_i| = |s_i|\}$ and its complement.

  By definition of the left sum, $u$ is given by (signed versions of) the entries of $s$ on $I$ and (signed versions of) the entries of $t$ on $[d] \setminus I$. 
  Let $v$ be a vertex of $\Uncomp(s)$ which agrees with $u$ on $I$ and let $w$ be a vertex of $\Uncomp(t)$ which agrees with $u$ on $[d] \setminus I$.
  This just means that $u = v \lplus w$.
  Furthermore, by construction, the latter left sum lies in $\vertices(X \cup Y)$. 
  As $u$ belongs to $\vertices(X \cup Y)$, we get the inclusion $\Uncomp(s \lplus t) \subseteq \faces(X \cup Y)$ by definition of $\faces(.)$.
\end{proof}

\begin{lemma}\label{le:faces_convex}
  The set $\faces(X)$ is TC-convex.

  In particular, we have $\faces(X) = \whull(\vertices(X))$
\end{lemma}
\begin{proof}
  Let $x = |\bigoplus X|$ and let $x^{\sigma}$ for $\sigma \in \{\ominus,\oplus\}^d$ denote a signed version of $x$.
  Recall that the points $\vertices(X)$ are of the form $x^{\sigma}$ where $\sigma$ ranges over a subset $\Sigma$ of $\{\ominus, \oplus\}^d$.
  Let $H_{\tau}^+ = \Hclosed^+\left(0,(x^{\tau}_1)^{\odot-1},\dots,(x^{\tau}_d)^{\odot-1}\right)$ be the halfspace with its `apex' at the point $x^{\tau}$ for $\tau \in \{\ominus,\oplus\}^d$ containing $\mathbf{\Zero}$ and $H_{\tau}^{-}$ the opposite closed one. 
  Then the intersection $\bigcap_{\tau \in \{\ominus,\oplus\}^d} H_{\tau}^+$ exactly yields the hypercube $\Uncomp(\bullet x)$. 
  If we further intersect this with all halfspaces $H_{\tau}^{-}$ for $\tau \in \{\ominus,\oplus\}^d \setminus \Sigma$ we get $\faces(X)$ as $H_{\tau}^{-}$ exactly cuts off all faces of the cube containing $x^{\tau}$ with $\tau$ not in $\Sigma$. 
  Hence, it is TC-convex as an intersection of TC-convex sets.
  
  \smallskip

  Now, we look at a face $F$ of $\Uncomp(\bullet x)$ for which all vertices $V$ are contained in $\vertices(X)$. 
  Applying~\cref{le:weak_inter} iteratively on pairs of points, which only differ in the sign of one component, we get $\faces(V) = F \subseteq \whull(V)$. 
  Ranging over all faces in $\faces(X)$ yields $\faces(X) \subseteq \whull(X)$. 
  But as $\faces(X)$ is TC-convex, we get an equality. 
\end{proof}

\begin{proposition}\label{le:inthull_from_faces}
Let $X \subseteq \TTpm^{d}$ be an $n$-element set interpreted as a matrix. Then
\begin{equation}\label{eq:int_hull}
\whull(X) = \bigcup\SetOf{\faces(X \odot \diag(\lambda))}{\lambda \in \Tgeq^n, \bigoplus_i \lambda_i = 0} \, .
\end{equation}
\end{proposition}
\begin{proof}
  Let $x_1, \dots, x_n \in X$ and $\lambda_1, \dots, \lambda_n \in \Tgeq$ be such that $\bigoplus_i \lambda_i = 0$.
  We will start by showing that $y = \lambda_1 \odot x_1 \lplus \dots \lplus \lambda_n \odot x_n$ belongs to $\whull(X)$.
  The proof goes by induction over $n$.

  The claim is trivial for $n  = 1$.
  For higher $n$, let $\mu = \bigoplus_{i > 1} \lambda_i$.
  The claim is trivial if $\mu = \Zero$.
  Otherwise, let $\mu_i = \lambda_i - \mu$ for every $i > 1$.
  Then, we have $\lambda_1 \odot x_1 \lplus \dots \lplus \lambda_n \odot x_n = \lambda_1 \odot x_1 \lplus \mu \odot (\mu_2 \odot x_2 \lplus \dots \lplus \mu_n \odot x_n)$.
  By the induction hypothesis, the point $z = \mu_2 \odot x_2 \lplus \dots \lplus \mu_n \odot x_n$ belongs to $\whull(X)$.
  Therefore, the point $y = \lambda_1 \odot x_1 \lplus \mu \odot z$ belongs to $\whull(X)$ by \cref{le:weak_inter}.
  Thus, we have proven that every point of the form $\lambda_1 \odot x_1 \lplus \dots \lplus \lambda_n \odot x_n$ belongs to $\whull(X)$.
  In particular, the set $\vertices(X \odot \diag(\lambda))$ belongs to $\whull(X)$.
  Hence,~\cref{le:faces_convex} implies that $\faces(X \odot \diag(\lambda)) = \whull\left(\vertices(X \odot \diag(\lambda))\right)$ is contained in $\whull(X)$. 

  \smallskip
  
  To finish the proof, we will show that the set on the right-hand side of \cref{eq:int_hull} is TC-convex. 
  To do so, suppose that $a \in \faces(X \odot \diag(\lambda))$ and $b \in \faces(X \odot \diag(\mu))$ for some $\lambda_i, \mu_i \in \Tgeq$ and $\bigoplus_i \lambda_i = \bigoplus_i \mu_i = 0$.
  Further, let $\alpha, \beta \in \Tgeq$ be such that $\alpha \oplus \beta = 0$.
  By~\cref{le:adding_boxes}, the set $\vertices(\alpha \odot a, \beta \odot b)$ is included in
  $\faces(\alpha \odot X \odot \diag(\lambda) \cup \beta \odot X \odot \diag(\mu))$. 
  Hence, $\faces(\alpha \odot a, \beta \odot b)$ is also included in this set by \cref{le:faces_convex}. 
  Thus, by \cref{le:weak_inter}, $\whull(a,b)$ is included in the set defined on the right-hand side of \cref{eq:int_hull} since $\faces(\alpha \odot X \odot \diag(\lambda) \cup \beta \odot X \odot \diag(\mu)) = \faces(X \odot \diag(\nu))$ with $\nu = \alpha \odot \diag(\lambda) \oplus \beta \odot \diag(\mu)$, which fulfills $\nu \in \Tgeq^n$ and $\bigoplus_i \nu_i = 0$.
\end{proof}

Using the representation of the convex hull as union of finite convex hulls stated in \cref{cor:convexity-finitary} we get the following.

\begin{corollary} \label{cor:whull-arbitrary-faces}
If $X \subseteq \TTpm^d$, then
\begin{align*}
\whull(X) &= \bigcup\SetOf{\faces(X \odot \diag(\lambda))}{\lambda \in \Tgeq^X, \bigoplus_i \lambda_i = 0, |\supp(\lambda)| < +\infty} \, .
\end{align*}
\end{corollary}

We now give estimates on the Carath{\'e}odory-number of TC-convexity.
The core case for this is the representation of vertices from a small set of generators. 

\begin{lemma}\label{le:vertex_caratheodory}
  Let $X = \{x_1, \dots, x_n\} \subseteq \TTpm^d$.
  Then there is a subset $Y \subseteq X$ with $|Y| \leq d2^d$ and $\vertices(X) \subseteq \vertices(Y)$, hence, $\vertices(X) = \vertices(Y)$. 
\end{lemma}
\begin{proof}
  Let $y \in \vertices(X)$ and let $\sigma \in \Sym(n)$ such that $y = x_{\sigma(1)} \lplus \dots \lplus x_{\sigma(n)}$. 
  For every coordinate $k \in [d]$ let $j_k \in [n]$ be the smallest number such that $y_k = (x_{\sigma(j_k)})_k$ and let $J_{y} = \SetOf{\sigma(j_k)}{k \in [d]} \subseteq [n]$.
  Note that $|J_{y}| \le d$.
  Now, let $I_{y} = (i_1,\dots,i_{|J_y|})$ be the ascending sequence of the elements in $J_{{y}}$. 
  Observe that for any set $J$ with $J_{y} \subseteq J \subseteq [n]$,
  it holds $y = x_{i_1} \lplus \dots \lplus x_{|J_y|} \lplus \biglplus_{i \in J \setminus J_{y}} x_i $ for any order of the summands indexed by $J \setminus J_{y}$.

  Finally, we define $J$ to be the union $\bigcup_{y \in \vertices(X)} J_y$.
  By the above reasoning, we have $|J| \leq d2^d$ and $\vertices(X) \subseteq \vertices(Y)$.
  As each left-sum of the elements in $Y$ already defines a point for which the absolute values of the components equal the components of $\bullet \bigoplus X$, we also get the reverse inclusion $\vertices(Y) \subseteq \vertices(X)$.
  
\end{proof}

\begin{proposition}\label{le:inter_caratheodory}
If $X \subseteq \TTpm^d$, then
\begin{equation*}
\whull(X) = \bigcup\SetOf{\faces(X \odot \diag(\lambda))}{\lambda \in \Tgeq^X, \bigoplus_i \lambda_i = 0, |\supp(\lambda)| \leq c_d} \, ,
\end{equation*}
where $c_d = d2^d + 1$.
\end{proposition} 
\begin{proof}
  Let $y \in \whull(X)$.
  By \cref{cor:whull-arbitrary-faces} 
  there is some $\lambda \in \Tgeq^X$ with $\bigoplus_i \lambda_i = 0$ such that $y \in \faces(X \odot \diag(\lambda))$. 

  Let $j_0$ be an index with $\lambda_{j_0} = 0$.
  We derive a new coefficient vector $\mu$ from $\lambda$ with $\mu_{j_0} = 0$, $\vertices(X \odot \diag(\lambda)) = \vertices(X \odot \diag(\mu))$ and $|\supp(\mu)| \leq c_d = d2^d+1$.
  \Cref{le:vertex_caratheodory} implies that we can achieve this by setting all but $c_d$ entries of $\lambda$ to $\Zero$. 

  With~\cref{le:faces_convex}, we obtain 
  \begin{align*}
    y \in \faces(X \odot \diag(\lambda)) &= \whull(\vertices(X \odot \diag(\lambda))) \\
    &\subseteq \whull(\vertices(X \odot \diag(\mu)) \\
    &= \faces(X \odot \diag(\mu)) \, . 
  \end{align*}
  Taking the subset $Y$ of $X$ given by the support of $\mu$, one sees $y \in \whull(Y)$. 
\end{proof}

While TC-convexity extends the `usual' tropical convexity whose Carath{\'e}odory number is $d+1$ as discussed in~\cite{GaubertMeunier:2010}, we already get an exponential lower bound from a simple example. 

\begin{example} \label{ex:caratheodory-number}
Let $X = \{\ominus 0, 0\}^d$.
Then, \cref{le:faces_convex} shows that $\whull(X) = [\ominus 0, 0]^d$.
In particular, we have $\Zero \in \whull(X)$.
Moreover, if $Y \subsetneq X$ is any strict subset of $X$, then \cref{le:faces_convex} shows that $\whull(Y)$ is the union of faces of $[\ominus 0, 0]^d$ whose vertices belong to $Y$.
In particular, $\Zero \notin \whull(Y)$.
This example shows the lower bound $c_d \ge 2^d$ for the Carath{\'e}odory number of TC-convexity.
We do not know what is the optimal value of $c_d$.
\end{example}

\subsection{TC-convex cones}

In usual convexity, the study of convex sets and convex cones is intimately connected. 
We define the TC-conic hull of a set and infer a representation by combinations of points for TC-convex cones from the representation of TC-convex sets given in \cref{le:inthull_from_faces}.
We also introduce TC-span based on the TC-conic hull and show a basic property for later use. 
We finish with a technical lemma for later use which represents the slice of a TC-convex cone at `height' $0$ as a TC-convex hull of finitely many points and `rays'.

\begin{definition}\label{def:TC_cone}
Given an arbitrary set $X \subseteq \TTpm^d$, we denote
\[
 \wcone(X) = \SetOf{\lambda \odot x}{\lambda \in \Tgeq, x \in \whull(X)} \, ,
 \]
the \emph{TC-conic hull} of $X$. 
\end{definition}

\begin{lemma}\label{le:TC_cone_from_faces}
If $X \subseteq \TTpm^d$, then
\begin{equation*}
\wcone(X) = \bigcup\SetOf{\faces(X \odot \diag(\lambda))}{\lambda \in \Tgeq^X, |\supp(\lambda)| < +\infty} \, .
\end{equation*}
\end{lemma}
\begin{proof}
By \cref{le:inthull_from_faces} we have the equality
\begin{align*}
&\wcone(X) \\
&= \bigcup\SetOf{\mu \odot \faces(X \odot \diag(\lambda))}{\mu \in \Tgeq, \lambda \in \Tgeq^X, \bigoplus_i \lambda_i = 0,  |\supp(\lambda)| < +\infty} \\
&= \bigcup\SetOf{\faces(X \odot \diag(\mu \odot \lambda))}{\mu \in \Tgeq, \lambda \in \Tgeq^X, \bigoplus_i \lambda_i = 0,  |\supp(\lambda)| < +\infty}
\end{align*}
Furthermore, if $\xi \in \Tgeq^X$ $\xi \neq \Zero$ has finite support, then we can write it as $\xi = \mu \odot \lambda$, where $\mu = \xi_1 \oplus \dots \oplus \xi_d \in \Tgeq$ and $\lambda \in \Tgeq^X$, $\bigoplus_i \lambda_i = 0$ is defined as $\lambda_i = \xi_i \odot \mu^{\odot -1}$ for all $i$. Therefore, we get
\begin{align*}
&\wcone(X) \\
&= \bigcup\SetOf{\faces(X \odot \diag(\mu \odot \lambda))}{\mu \in \Tgeq, \lambda \in \Tgeq^X, \bigoplus_i \lambda_i = 0, |\supp(\lambda)| < +\infty} \\
&= \bigcup\SetOf{\faces(X \odot \diag(\lambda))}{\lambda \in \Tgeq^X, |\supp(\lambda)| < +\infty} \, . \qedhere
\end{align*}
\end{proof}

Recall that a \emph{cone} is a set $X \subseteq \TTpm^d$ such that $\lambda \odot X \subseteq X$ for all $\lambda \in \Tgt$.

\begin{corollary} \label{le:convexity_TC_cone}
 The set $\wcone(X)$ is the smallest TC-convex cone that contains $X$ and $\Zero$. 
\end{corollary}
\begin{proof}
  We set $\tilde{X} = X \cup \Zero$ and use the representation from \cref{le:TC_cone_from_faces} to get
  \begin{align*}
    &\wcone(X) \\
    &= \bigcup_{\mu \in \Tgeq}\bigcup\SetOf{\faces(X \odot \diag(\lambda))}{\lambda \in \Tgeq^X, \bigoplus_i \lambda_i \leq \mu,|\supp(\lambda)| < +\infty} \\
    &= \bigcup_{\mu \in \Tgeq}\bigcup\SetOf{\faces(\tilde{X} \odot \diag(\lambda))}{\lambda \in \Tgeq^{\tilde{X}}, \bigoplus_i \lambda_i = \mu,|\supp(\lambda)| < +\infty} \\
    &= \bigcup_{\mu \in \Tgeq} \whull\left(\mu \odot \tilde{X}\right) \enspace .
  \end{align*}
  As this is a nested union for increasing $\mu$, it is a TC-convex set by \cref{cor:basic-properties-TC-convex}.

  Furthermore, let $Z$ be the smallest TC-convex cone containing $\tilde{X}$.
  By definition of a cone, one gets $\mu \odot \tilde{X} \subseteq Z$ for all $\mu > \Zero$ and, by TC-convexity, also  $\whull\left(\mu \odot \tilde{X}\right) \subseteq Z$ for all $\mu > \Zero$. 
  Since $\Zero$ belongs to $Z$, the equality above shows the minimality of $\wcone(X)$. 
  \end{proof}
  
\begin{corollary}\label{cor:hull_of_cone}
If $X \subseteq \TTpm^d$ is a cone, then $\wcone(X) = \whull(X) \cup \{\Zero\}$.
\end{corollary}
\begin{proof}
The inclusion $\whull(X) \cup \{\Zero\} \subseteq \wcone(X)$ follows from \cref{le:convexity_TC_cone}. Conversely, if $y \in \wcone(X)$, then \cref{le:TC_cone_from_faces} shows that $y \in \faces(X \odot \diag(\lambda))$ for some $\lambda \in \Tgeq^X, |\supp(\lambda)| < +\infty$. If $\lambda = \Zero$, then $y = \Zero$. Otherwise, we define $\mu \in \Tgeq^X$ as 
\[
\mu_z = \begin{cases}
0 &\text{if $z = \lambda_x \odot x$ for some $x \in \supp(\lambda)$},\\
\Zero &\text{otherwise}.
\end{cases}
\]
Then, $\bigoplus_{z} \mu_z = 0$ and $X \odot \diag(\lambda) = X \odot \diag(\mu)$ because $X$ is a cone. Hence $y \in \faces(X \odot \diag(\mu))$ and $y \in \whull(X)$ by \cref{cor:whull-arbitrary-faces}.
\end{proof}

We define the TC-span in analogy to the classical fact that a linear space can be written as a cone with antipodal generators: $\wspan(X) = \wcone(X,\ominus X)$. 
Observe that this is a monotonous operator, that is $\wspan(X) \subseteq \wspan(Y)$ for $X \subseteq Y \subseteq \TTpm^{k}$. 

We relate the TC-convex hull and TC-span by intersecting with a suitably scaled box. 

\begin{observation} \label{obs:intersection+cuboids}
  Let $Q \subseteq \RR^k$ be a (potentially low-dimensional) axis-parallel hypercube with nonempty intersection with $[-1,1]^k$. 
  Then the vertices of $Q \cap [-1,1]^k$ are obtained from the vertices of $Q$ by replacing all entries of absolute value bigger than $1$ with an entry of absolute value $1$ but with the same sign. 
\end{observation}

\begin{lemma} \label{lem:special-hull-unit-cube}
Suppose that $X \subseteq \{\ominus 0, \Zero, 0\}^k$ satisfies $X = \ominus X$ and $\Zero \in X$. Then $\whull(X) = \wspan(X) \cap [\ominus 0, 0]^k$.
\end{lemma}
\begin{proof}
  Denote $X = \{x_1, x_2, \dots, x_n\}$ where $x_1 = \Zero$.
  The inclusion $\whull(X) \subseteq \wspan(X) \cap [\ominus 0, 0]^k$ is trivial because the set $ [\ominus 0, 0]^k$ is TC-convex.

  To prove the opposite inclusion, let $y \in \wspan(X) \cap [\ominus 0, 0]^k$.
  Since $X = \ominus X$, we have $\wspan(X) = \wcone(X)$.
  Hence, by \cref{le:TC_cone_from_faces} we have $y \in \faces(X \odot \diag(\lambda))$ for some $\lambda \in \Tgeq^n$.
  If $\lambda = \Zero$, then $y = \Zero$ and so $y \in \whull(X)$. Otherwise, let $\mu \in \Tgeq^n$ be defined as $\mu_1 = 0$ and $\mu_i = \min\{0,\lambda_i\}$ for every $i \ge 2$.
  We next show that $y \in \faces(X \odot \diag(\mu))$ concluding the proof by  \cref{le:TC_cone_from_faces} . 

  To do so, recall \cref{def:vert+faces} and let $F$ be a face of $\faces(X \odot \diag(\lambda))$ that contains $y$. Any vertex of this face is of the form $\lambda_{\sigma(1)} \odot x_{\sigma(1)} \lplus \dots \lplus \lambda_{\sigma(n)} \odot x_{\sigma(n)}$ for some permutation $\sigma$ of $[n]$.
  Let $\tau$ be the permutation of $[n]$ such that $\lambda_{\tau(1)} \geq \dots \geq \lambda_{\tau(n)}$ and that $\tau(i) < \tau(j)$ for $i<j$ with $\lambda_{\sigma(i)} = \lambda_{\sigma(j)}$.
  Then, we have $\lambda_{\sigma(1)} \odot x_{\sigma(1)} \lplus \dots \lplus \lambda_{\sigma(n)} \odot x_{\sigma(n)} = \lambda_{\tau(1)} \odot x_{\tau(1)} \lplus \dots \lplus \lambda_{\tau(n)} \odot x_{\tau(n)}$.
  Indeed, consider this left sum componentwise.
  The value of a component is given by the first non-$\Zero$ entry in the sum with maximal absolute value.
  This agrees in both expressions.
  Moreover, since $\lambda_{\tau(1)} \geq \dots \geq \lambda_{\tau(n)}$, the point $\mu_{\tau(1)} \odot x_{\tau(1)} \lplus \dots \lplus \mu_{\tau(n)} \odot x_{\tau(n)}$ is a vertex of $\faces(X \odot \diag(\mu))$ with the same sign pattern.
  Therefore, applying \cref{obs:intersection+cuboids} to $F$ shows that $y \in F \cap [\ominus 0, 0]^k \subseteq \faces(X \odot \diag(\mu))$.
\end{proof}

We finish this section with the relation between TC-convex hull and TC-conic hull via homogenization.

\begin{lemma}\label{le:homog_halflines}
Suppose that $V,W \subset \TTpm^d$ are two nonempty finite sets. Let $\hat{V} = \SetOf{(v,0)}{v \in V}$, $\hat{W} = \SetOf{(w,\Zero)}{w \in \hat{W}}$, and 
\[
X = \whull\bigl(\SetOf{v \lplus \lambda \odot w}{v \in V, w \in W, \lambda \in \Tgeq}\bigr) \, . 
\]
Then, we have the equality
\[
\SetOf{(x, 0)}{x \in X} = \wcone(\hat{V} \cup \hat{W}) \cap \{x_{d+1} = 0\} \, .
\]
\end{lemma}
\begin{proof}
Let $x \in X$. By \cref{le:inter_caratheodory}, we have
\[
x \in \faces\bigl(\mu_1 \odot (v_1 \lplus \lambda_1 \odot w_1), \dots, \mu_{\ell} \odot (v_\ell \lplus \lambda_\ell \odot w_{\ell}) \bigr)
\]
for some $\mu_{i} \in \Tgeq$, $\bigoplus_{i} \mu_i = 0$, $\lambda_i \in \Tgeq$, $v_i \in V$, $w_i \in W$. Denote $\xi_i = \mu_i \odot \lambda_i$ for all $i$ and observe that
\begin{align*}
\vertices&\bigl(\mu_1 \odot (v_1 \lplus \lambda_1 \odot w_1), \dots, \mu_{\ell} \odot (v_\ell \lplus \lambda_\ell \odot w_{\ell}) \bigr) \\
&= \vertices(\mu_1 \odot v_1 \lplus \xi_1 \odot w_1, \dots, \mu_{\ell} \odot v_\ell \lplus \xi_\ell \odot w_{\ell}) \\
&\subseteq \vertices(\mu_1 \odot v_1, \dots, \mu_\ell \odot v_\ell, \xi_i \odot w_1, \dots, \xi_\ell \odot w_\ell) \, .
\end{align*}
By \cref{le:faces_convex} we get $z \in \faces(\mu_1 \odot v_1, \dots, \mu_\ell \odot v_\ell, \xi_i \odot w_1, \dots, \xi_\ell \odot w_\ell)$. Let $\hat{v}_i = (v_i,0) \in \hat{V}$ and $\hat{w}_i = (w_i, \Zero) \in \hat{W}$ for all $i$. Since $\bigoplus_{i} \mu_i = 0$ we get
\[
(x,0) \in \faces(\mu_1 \odot \hat{v}_1, \dots, \mu_\ell \odot \hat{v}_\ell, \xi_i \odot \hat{w}_1, \dots, \xi_\ell \odot \hat{w}_\ell) \, .
\]
In particular, \cref{le:TC_cone_from_faces} implies that $(x,0) \in \wcone(\hat{V} \cup \hat{W})$. Conversely, suppose that $(x,0) \in \wcone(\hat{V} \cup \hat{W})$. Then, \cref{le:TC_cone_from_faces} shows that there exist $\mu \in \Tgeq^{\ell}, \xi \in \Tgeq^{\ell}$, $\hat{v}_i \in \hat{V}$, $\hat{w}_i \in \hat{W}$ such that $(x,0) \in \faces(\mu_1 \odot \hat{v}_1, \dots, \mu_\ell \odot \hat{v}_\ell, \xi_i \odot \hat{w}_1, \dots, \xi_\ell \odot \hat{w}_\ell)$. Since the last coordinate of $(x,0)$ is $0$, we get $\bigoplus_{i} \mu_i = 0$. Without loss of generality, we can suppose that $\mu_1 = 0$. Observe that
\begin{align*}
\vertices&(\mu_1 \odot \hat{v}_1, \dots, \mu_\ell \odot \hat{v}_\ell, \xi_i \odot \hat{w}_1, \dots, \xi_\ell \odot \hat{w}_\ell) \\
&= \vertices(\mu_1 \odot \hat{v}_1, \dots, \mu_\ell \odot \hat{v}_\ell, (-1) \odot \hat{v}_1 \lplus \xi_i \odot \hat{w}_1, \dots, (-1) \odot \hat{v}_1 \lplus \xi_\ell \odot \hat{w}_\ell) \\
&= \vertices(\mu_1 \odot \hat{v}_1, \dots, \mu_\ell \odot \hat{v}_\ell, (-1) \odot \hat{z}_1, \dots, (-1) \odot \hat{z}_\ell) \, ,
\end{align*}
where $\hat{z} = \hat{v}_1 \lplus (\xi_i + 1) \odot \hat{w}_i$. Hence, by \cref{le:faces_convex} we get 
\[
(x,0) \in \faces(\mu_1 \odot \hat{v}_1, \dots, \mu_\ell \odot \hat{v}_\ell, (-1) \odot \hat{z}_1, \dots, (-1) \odot \hat{z}_\ell) \, .
\]
For every $i$, let  $v_i,z_i \in \TTpm^d$ be the projection of $\hat{v}_i,\hat{z}_i$ obtained by deleting the last coordinate. Then, we have 
\[
\{v_1, \dots, v_\ell, z_1, \dots, z_\ell\} \subseteq \SetOf{v \lplus \lambda \odot w}{v \in V, w \in W, \lambda \in \Tgeq} \, .
\]
Furthermore, since the last coordinate of every point $\hat{v}_i, \hat{z}_i$ is equal to $0$ and $\bigoplus_{i} \mu_i = 0$, we get $x \in \faces(\mu_1 \odot v_1, \dots, \mu_\ell \odot v_\ell, (-1) \odot z_1, \dots, (-1) \odot z_\ell)$. Therefore, by \cref{le:inter_caratheodory} we have $x \in X$.
\end{proof}

\section{Separation by Hemispaces}
\label{sec:separation+hemispaces}

A \emph{hemispace} is a convex set whose complement is also convex.
Separation by hemispaces has been studied extensively in the context of abstract convexity structures, see, e.g.,~\cite{VanDeVel:1993}. 
In this section, we prove theorems about separation by hemispaces in the context of TO- and TC-convexity.
These theorems will serve as a building blocks towards proving the separation by halfspaces discussed in \cref{sec:TC-hull+as+intersection}.

We say that a TO-convex set $X \subseteq \TTpm^d$ is a \emph{TO-hemispace} if $\TTpm^d \setminus X$ is also TO-convex. Likewise, we say that a TC-convex set $X \subseteq \TTpm^d$ is a \emph{TC-hemispace} if $\TTpm^d \setminus X$ is also TC-convex. By \cref{cor:TO-contain-TC}, a TO-hemispace is also a TC-hemispace. The first theorem of this section (proven in \cref{sec:separation_hemispace_TO}) shows that TO-convexity has very well-behaved separation properties: it fulfills the Pasch property and the Kakutani property, see \cref{fig:pasch,fig:to_pasch} for an illustration. A similar statement was proven for the unsigned case in~\cite{Horvath:2017}.

\begin{theorem}\label{thm:kakutani}
TO-convexity has the Pasch property and the Kakutani property. In other words,
\begin{enumerate}[(i)]
\item (Pasch property) if $(a,b_1,b_2,c_1,c_2) \in \TTpm^d$ are such that $b_1 \in \shull(a,c_1)$, $b_2 \in \shull(a,c_2)$, then $\shull(c_1,b_2) \cap \shull(c_2,b_1) \neq \emptyset$;
\item (Kakutani property) if $A, B \subseteq \TTpm^d$ are two disjoint TO-convex sets, then there exists a TO-hemispace $X \subseteq \TTpm^d$ such that $A \subseteq X$ and $B \subseteq (\TTpm^d \setminus X)$.
\end{enumerate}
In particular, any TO-convex set is an intersection of TO-hemispaces.
\end{theorem}

As we observe in \cref{ex:TC-non-Pasch-plane}, the TC-convexity does not satisfy the Pasch--Kakutani properties. However, we show in \cref{sec:sepatation_hemispace_TC} that it satisfies the following weaker separation property, sometimes called the \emph{$S_{3}$ property}.

\begin{theorem}\label{thm:separation_hemispaces}
  Every TC-convex set is an intersection of TC-hemispaces.

  More precisely, each of these TC-hemispaces can be chosen such that its complement is TO-convex. 
\end{theorem}

Subsequently, in \cref{sec:hemispace_nearly_halfspace} we study the structure of TC-hemispaces. The main result of this section states that TC-hemispaces are ``nearly'' halfspaces, in the sense that a TC-hemispace is sandwiched between a closed halfspace and its interior. This extends a similar result known for the unsigned case~\cite{BriecHorvath:2008,KatzNiticaSergeev:2014,EhrmannHigginsNitica:2016}.

\begin{theorem} \label{th:hemispace}
If $G \subseteq \TTpm^d$ is a TC-hemispace and $G \notin\{\emptyset, \TTpm^d\}$, then there exists a vector $(a_0, \dots, a_d) \in \TTpm^{d+1}$, $(a_1, \dots, a_d) \neq \Zero$, such that $\HC^+(a) \subseteq G \subseteq \Hclosed^{+}(a)$.
\end{theorem}

In \cref{sec:boundary_hemispace} we analyze more closely the boundary of TC-hemispaces. We prove that the boundary is made our of polyhedral pieces coming from the cell decomposition of $\TTpm^d$ induced by the associated hyperplane, see \cref{le:bigger_pos_argmax} for a precise statement.

\subsection{Separation of TO-convex sets}\label{sec:separation_hemispace_TO}

We now give a proof of \cref{thm:kakutani}. Our proof is based on the basic algebraic properties of $\TSS$ established in \cref{sec:basic_calculation} and combining them with results known for general convexity structures.

\begin{proof}[Proof of \cref{thm:kakutani}]
  For $(i)$, we follow the proof of the Pasch property over arbitrary ordered fields, see~\cite[I. Proposition~4.14.1]{VanDeVel:1993}.
  Let $(a,b_1,b_2,c_1,c_2)$ be as in $(i)$.
  Then there are $s_1, s_2, \bar{s}_1, \bar{s}_2 \in \Tgeq$ such that
  \begin{equation} \label{eq:Pasch-proof-convex-scalars}
    s_1 \oplus \bar{s}_1 = 0 \quad , \quad s_2 \oplus \bar{s}_2 = 0
  \end{equation}
    and 
  \begin{equation} \label{eq:Pasch-proof-intermediate-point}
    b_1 \in \Uncomp(s_1 \odot a \oplus \bar{s}_1 \odot c_1) \quad \text{ and } \quad b_2 \in \Uncomp(s_2 \odot a \oplus \bar{s}_2 \odot c_2) \enspace .
  \end{equation}
  We define
  \begin{equation} \label{eq:Pasch-proof-adapted-scalars}
    \begin{aligned}
      t_1 &= s_2 \odot (s_1 \oplus \bar{s}_1 \odot s_2)^{\odot -1}\\
      \bar{t}_1 &= s_1 \odot \bar{s}_2 \odot (s_1 \oplus \bar{s}_1 \odot s_2)^{\odot -1}\\
      t_2 &= \bar{s}_1 \odot s_2 \odot (s_1 \oplus \bar{s}_1 \odot s_2)^{\odot -1}\\
      \bar{t}_2 &= s_1 \odot (s_1 \oplus \bar{s}_1 \odot s_2)^{\odot -1} \enspace .
    \end{aligned}
  \end{equation}
  Using \cref{eq:Pasch-proof-convex-scalars}, we obtain
  \begin{equation*}
    \begin{aligned}
      s_2 \oplus s_1 \odot \bar{s}_2 &= s_1 \odot s_2 \oplus \bar{s}_1 \odot s_2 \oplus s_1 \odot \bar{s}_2 \\
      &= s_1 \odot s_2 \oplus s_1 \odot \bar{s}_2 \oplus \bar{s}_1 \odot s_2 = s_1 \oplus \bar{s}_1 \odot s_2,
    \end{aligned}
  \end{equation*}
  which implies
  \begin{equation*}
    t_1 \oplus \bar{t}_1 = (s_2 \oplus s_1 \odot \bar{s}_2) \odot (s_1 \oplus \bar{s}_1 \odot s_2)^{\odot -1} = 0 \enspace .
  \end{equation*}
  Note that we also have $t_2 \oplus \bar{t}_2 = 0$.
  Hence, we obtain
  \begin{equation*}
    \Uncomp(t_1 \odot b_1 \oplus \bar{t}_1 \odot c_2) \subseteq \shull(b_1,c_2) \quad \text{ and }
    \quad \Uncomp(t_2 \odot c_1 \oplus \bar{t}_2 \odot b_2) \subseteq \shull (c_1,b_2) \enspace .
  \end{equation*}
  Multiplying by the denominator $s_1 \oplus \bar{s}_1 \odot s_2$ in~\cref{eq:Pasch-proof-adapted-scalars}, we see that the intersection $\shull(b_1,c_2) \cap \shull(c_1,b_2)$ is not empty if 
  \begin{equation*}
    \Uncomp(s_2 \odot b_1 \oplus s_1 \odot \bar{s}_2 \odot c_2) \cap \Uncomp(\bar{s}_1 \odot s_2 \odot c_1 \oplus s_1 \odot b_2) \neq \emptyset \enspace . 
  \end{equation*}
  Scaling~\cref{eq:Pasch-proof-intermediate-point} to
  \begin{equation*} 
    s_2 \odot b_1 \in \Uncomp(s_2 \odot s_1 \odot a \oplus s_2 \odot \bar{s}_1 \odot c_1) \; \text{ and } \; s_1 \odot b_2 \in \Uncomp(s_1 \odot s_2 \odot a \oplus s_1 \odot \bar{s}_2 \odot c_2) 
  \end{equation*}
  allows to apply~\cref{cor:iterated-sums-uncomp-dim} by setting
  \begin{equation*}
    v = s_1 \odot s_2 \odot a, q = s_2 \odot b_1, w = \bar{s}_1 \odot s_2 \odot c_1, p = s_1 \odot b_2, u = s_1 \odot \bar{s}_2 \odot c_2 .
  \end{equation*}

  \smallskip
  
  To prove $(ii)$, we use \cite[Theorem~5]{Chepoi:1994}, which shows that the Pasch property and the Kakutani property are equivalent for $2$-ary convexities (see also \cite[Theorem~4.1]{Kubis:2002} for a more recent improvement).
  TO-convexity is $2$-ary by \cref{le:TO-hull-intervals}, so it satisfies the Kakutani property. The last claim follows from the Kakuatni property because a single point is TO-convex. 
\end{proof}

\begin{figure}[t]
\centering
\includegraphics{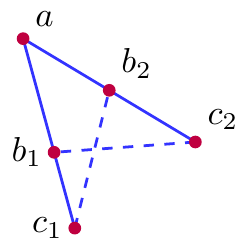}
\vspace*{-0.3cm}
\caption{The Pasch property in the real plane.}\label{fig:pasch}
\end{figure}

\begin{example}\label{ex:pasch_to}
Let $a = (0, \ominus 0)$, $b_1 = (\Zero, \ominus 0)$, $b_2 = (0, 0)$, $c_1 = (\ominus 0, \ominus 0)$, and $c_2 = (\Zero, 1)$.
\Cref{fig:to_pasch} depicts this configuration of points.
The TO-convex hull of $a$ and $c_1$ is the straight line connecting them, so $b_1 \in \shull(a,c_1)$.
Furthermore, $(-1)\odot c_2 \oplus a = (0, \bullet 0)$, so $b_2 \in \shull(a,c_2)$. 
One gets that the TO-convex hull of $b_1$ and $c_2$ is the dashed line connecting them and the TO-convex hull of $c_1$ and $b_2$ is just the shaded whole square.
Their nonempty intersection visualizes the Pasch property in this example. 
\end{example}

\begin{figure}[t]
\centering
\includegraphics{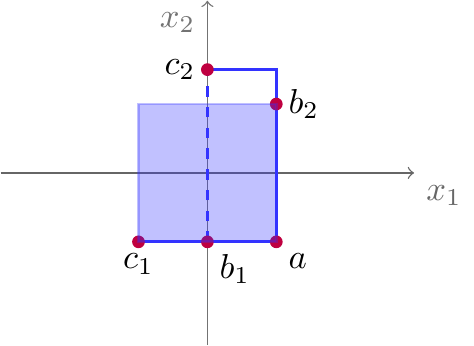}
\vspace*{-0.3cm}
\caption{The Pasch property in the TO-convexity is satisfied.}\label{fig:to_pasch}
\end{figure}

\begin{figure}[t]
\centering
\includegraphics{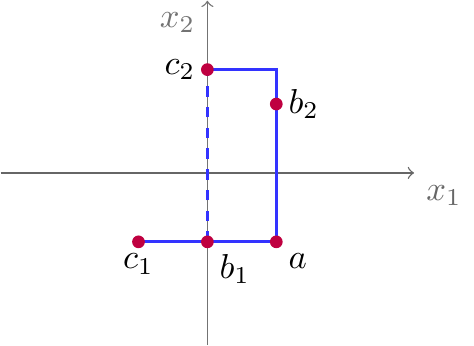}
\vspace*{-0.3cm}
\caption{The Pasch property in the TC-convexity is not satisfied.}\label{fig:tc_pasch}
\end{figure}

The next example shows that the TC-convexity does not have the Pasch property or the Kakutani property. Nevertheless, in \cref{sec:sepatation_hemispace_TC} we will show that the TC-convexity satisfies a weaker separation property.

\begin{example} \label{ex:TC-non-Pasch-plane}
As in \cref{ex:pasch_to}, let $a = (0, \ominus 0)$, $b_1 = (\Zero, \ominus 0)$, $b_2 = (0, 0)$, $c_1 = (\ominus 0, \ominus 0)$, and $c_2 = (\Zero, 1)$. 
\Cref{fig:tc_pasch} depicts this configuration of points.
Note that we have $b_1 \in \faces(a,c_1)$  and $b_2 \in \faces(a,-1 \odot c_2)$, so that $b_1 \in \whull(a,c_1)$, $b_2 \in \whull(a,c_2)$.
Nevertheless, the interval $\whull(c_1,b_2)$ is reduced to two points, $\whull(c_1,b_2) = \{c_1,b_2\}$, while $\whull(c_2,b_1) = \{\Zero\} \times [\ominus 0, 1]$.
In particular, the intersection $\whull(c_1,b_2) \cap \whull(c_2,b_1)$ is empty, so the TC-convexity does not have the Pasch property. This also shows that the TC-convexity does not have the Kakutani property: there is no TC-hemispace that separates $\whull(c_1,b_2)$ from $\whull(c_2,b_1)$, because the point $a$ could not belong to either side of this hemispace.
Even further, we note that $\whull(b_1,c_2) = \shull(b_1,c_2)$.
Hence, this example shows that the TO-convexity and TC-convexity do not satisfy the Kakutani property for pairs of convexities studied in~\cite{Kubis:2002}.
\end{example}

\subsection{Separation of TC-convex sets}\label{sec:sepatation_hemispace_TC}

The main ingredient for the proof of \cref{thm:separation_hemispaces} about the separation by TC-hemispaces will be the elimination property in \cref{prop:sand_glass}. 
This relies on a technical lemma that has the following geometric intuition: given two convex sets arising as the convex hull of finite sets $P_1$ and $P_2$, their intersection is contained in the convex hull of convex combinations of pairs of points in $P_1$ and $P_2$.

To use \cref{cor:faces_from_ineqs} for showing containment in $\faces( . )$, we introduce a helpful notion.  
For ${\sim} \in \{{\le}, {\ge}\}^d$, a point $w$ with $y_k \sim_k w_k$ for all $k \in [d]$ is said to \emph{dominate} $y$ with respect to $\sim$.

  \begin{lemma} \label{lem:combination-of-faces}
    If $\mu \le 0$ and $w \in \Uncomp(v^1 \oplus \mu \odot v^2)$, then
    \[
    \faces(X \cup v^1) \cap \faces(Y \cup v^2) \subseteq \faces(X \cup \mu \odot Y \cup w) \enspace
    \]
  \end{lemma}
  \begin{proof}
    Let $z$ be an arbitrary point in $\faces(X \cup v^1) \cap \faces(Y \cup v^2)$.
    Using \cref{cor:faces_from_ineqs}, for every ${\sim} \in \{{\le}, {\ge}\}^d$, there exist vertices $\tilde{x}_1 \lplus v^1 \lplus \tilde{x}_2$ and $\tilde{y}_1 \lplus  v^2 \lplus \tilde{y}_2$ in $\vertices(X \cup v^1)$ and $\vertices(Y \cup v^2)$, respectively, which dominate $z$ with respect to $\sim$. 
    Here, $\tilde{x}_1, \tilde{x}_2$ arise as left sum of points in $X$ and $\tilde{y}_1,\tilde{y}_2$ arise analogously. 
    As $\tilde{x}_1 \lplus \mu \odot \tilde{y}_1 \in \Uncomp(\tilde{x}_1 \oplus \mu \odot \tilde{y}_1)$, $w \in \Uncomp(v^1 \oplus \mu \odot v^2)$ and $\tilde{x}_2 \lplus \mu \odot \tilde{y}_2 \in \Uncomp(\tilde{x}_2 \oplus \mu \odot \tilde{y}_2)$, by \cref{lem:cancel-inequalities-left-sum}, $\tilde{x}_1 \lplus \mu \odot \tilde{y}_1 \lplus w \lplus \tilde{x}_2 \lplus \mu \odot \tilde{y}_2$ dominates $z$ with respect to $\sim$. 
    Ranging over all $\sim$, these points form a subset of $\vertices(X \cup \mu \odot Y \cup w)$.
    Again using \cref{cor:faces_from_ineqs}, $z$ is contained in $\faces(X \cup \mu \odot Y \cup w)$.
  \end{proof}

  This lemma serves as tool to show the following crucial elimination property of TC-convexity as visualized in \cref{fig:tc_sandglass}; the statement in classical geometry is visualized in~\cref{fig:real_sandglass}.  
  It is closely related to the `Sandglass Property' arising for interval convexities in~\cite[\S I.4]{VanDeVel:1993}. 
  
    \begin{figure}[t]
\centering
\includegraphics{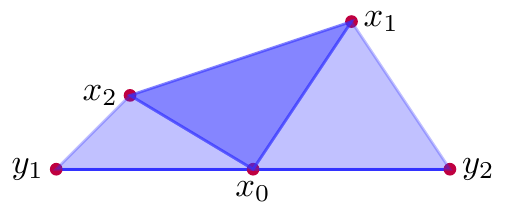}
\vspace*{-0.3cm}
\caption{Illustration of \cref{prop:sand_glass} in the standard convexity.}\label{fig:real_sandglass}
\end{figure}
  
  \begin{figure}[t]
\centering
\includegraphics{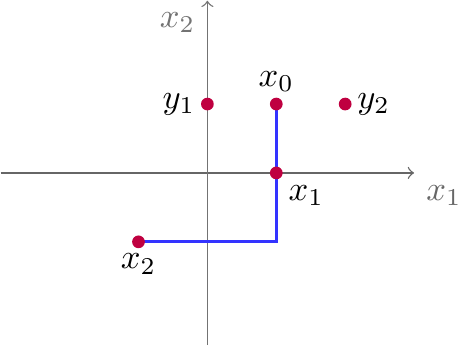}
\includegraphics{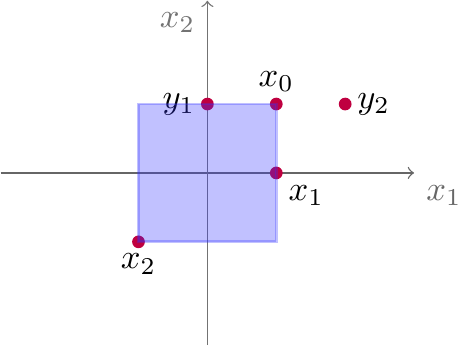}
\includegraphics{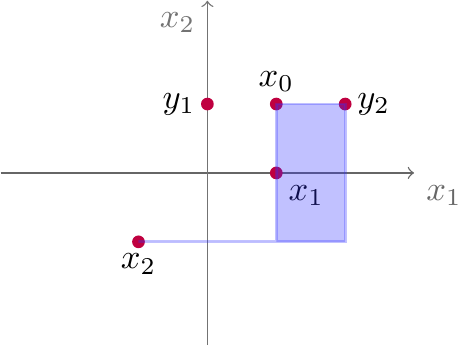}
\vspace*{-0.3cm}
\caption{Illustration of \cref{prop:sand_glass}.}\label{fig:tc_sandglass}
\end{figure}
  
\begin{theorem} \label{prop:sand_glass}
 Fix a set $X = \{x_0, \dots, x_n\} \subset \TTpm^d$. Then, for every $y_1,y_2 \in \TTpm^d$ such that $x_0 \in \shull(y_1, y_2)$ we have 
 \[
 \whull(X \cup \{y_1\}) \cap \whull(X \cup \{y_2\}) = \whull(X) \, .
\]
\end{theorem}
\begin{proof}
  The inclusion  $\whull(X \cup \{y_1\}) \cap \whull(X \cup \{y_2\}) \supseteq \whull(X)$ follows from monotonicity of the hull operator. Hence, it suffices to show that for every $z \in \TTpm^d$ such that $z \in \whull(X \cup \{y_1\}) \cap \whull(X \cup \{y_2\})$ we have $z \in \whull(X)$.

  Breaking the symmetry between $y_1$ and $y_2$, there is a $\mu \in \Tgeq$ with $\mu \leq 0$ such that $x_0 \in \Uncomp(y_1 \oplus \mu \odot y_2)$.
  So let $\lambda^{(1)} \in \Tgeq^{n+1}$, $\lambda^{(2)} \in \Tgeq^{n+1}$ and $\rho_1, \rho_2 \in \Tgeq$ with $\bigoplus \lambda^{(1)} \oplus \rho_1 = \bigoplus \lambda^{(2)} \oplus \rho_2 = 0$ be such that 
  \begin{equation*}
    z \in \faces(\diag(\lambda^{(1)}) \odot X \cup \rho_1 \odot y_1) \cap \faces(\diag(\lambda^{(2)}) \odot X \cup \rho_2 \odot y_2) .
  \end{equation*}

  First, if $\rho_2 = \Zero$ then $z \in \faces(\diag(\lambda^{(2)}) \odot X)$ yields the desired containment. 
  Furthermore, if $\mu = \Zero$ then $x_0 = y_1$ which concludes the claim with $z \in \faces(\diag(\lambda^{(1)}) \odot X \cup \rho_1 \odot y_1)$.
  
  Otherwise, let $\delta = \rho_1 - \rho_2 + \mu = \rho_1 \odot \rho_2^{\odot -1} \odot \mu$, or equivalently $\delta \odot \rho_2 = \rho_1 \odot \mu$. 
  We have to distinguish two cases, depending on the real sign of $\delta$.
  
  \textbf{Case 1} ($\delta \leq 0$.)
  By \Cref{lem:combination-of-faces}, we obtain that
  \begin{equation*}
    z \in \faces(\diag(\lambda^{(1)}) \odot X \cup \delta \odot \diag(\lambda^{(2)}) \odot X \cup w)
  \end{equation*}
 for an arbitrary $w \in \Uncomp(\rho_1 \odot y_1 \oplus \delta \odot \rho_2 \odot y_2)$. 
 With $\delta \odot \rho_2 = \rho_1 \odot \mu$, this implies that 
 \begin{equation*}
    z \in \faces(\diag(\lambda^{(1)}) \odot X \cup \delta \odot \diag(\lambda^{(2)}) \odot X \cup \rho_1 \odot x_0) \,.
 \end{equation*}
 In particular, this yields $z \in \whull(X)$ since $\delta, \rho_1 \leq 0$ and, therefore, $\bigoplus \lambda^{(1)} \oplus \delta \odot \bigoplus \lambda^{(2)} \oplus \rho_1 = \bigoplus \lambda^{(1)} \oplus \rho_1 = 0$. 

 \textbf{Case 2} ($\delta > 0$.) 
 By \Cref{lem:combination-of-faces}, we obtain that
  \begin{equation*}
    z \in \faces(\delta^{\odot -1} \odot \diag(\lambda^{(1)}) \odot X \cup \diag(\lambda^{(2)}) \odot X \cup w)
  \end{equation*}
 for an arbitrary $w \in \Uncomp(\delta^{\odot -1} \odot\rho_1 \odot y_1 \oplus \rho_2 \odot y_2)$. 
 With $\delta^{\odot -1} \odot \rho_1 = \rho_2 \odot \mu^{\odot -1}$, this implies that 
 \begin{equation} \label{eq:z-in-faces-scaled}
    z \in \faces(\delta^{\odot -1} \odot \diag(\lambda^{(1)}) \odot X \cup \diag(\lambda^{(2)}) \odot X \cup \rho_2 \odot \mu^{\odot -1} \odot x_0) \,.
 \end{equation}
 Note that $\delta^{\odot -1} < 0$ implying $\delta^{\odot -1} \odot (\bigoplus \lambda^{(1)}) < 0$ and that $\rho_2 \leq \rho_2 \odot \mu^{\odot -1} = \delta^{\odot -1} \odot \rho_1 < 0$. 
 With this, 
 \[
 \bigoplus \lambda^{(2)} \leq \delta^{\odot -1} \odot (\bigoplus \lambda^{(1)}) \oplus \bigoplus \lambda^{(2)} \oplus \rho_2 \odot \mu^{\odot -1} \leq \bigoplus \lambda^{(2)} \oplus \rho_2 = 0 = \bigoplus \lambda^{(2)} \;,
 \]
 Therefore, the scalars in \cref{eq:z-in-faces-scaled} constitute a proper combination, hence, $z \in \whull(X)$.
\end{proof}

The maximal set of elements which is tropically convex and does not contain a fixed point is called \emph{semispace} in~\cite{KatzNiticaSergeev:2014}.
This concept (for unsigned tropical convexity) was used to study hemispaces.
We use a similar idea in the following proof. 
While we do not have a proper separation theorem for TC-hemispaces, \cref{prop:sand_glass} suffices to show a representation of TC-convex sets by TC-hemispaces. 

\begin{proof}[Proof of \cref{thm:separation_hemispaces}]
  The proof is a variant of the argument given in \cite[Thm. 5.2]{Kubis:1999} and \cite[Thm. I.4.13]{VanDeVel:1993}. Let $G \subseteq \TTpm^d$ be a TC-convex set.  If $G \in \{\emptyset, \TTpm^d\}$, then the claim is trivial.
  Otherwise, let $z \notin G$.
  We want to prove that there exists a TC-hemispace $H^*$ such that $G \subseteq H^*$ and $z \notin H^*$.
  To do so, we consider the family $\mathcal{F}$ of TC-convex subsets $S$ of $\TTpm^d$ such that $G \subseteq S$ and $z \notin S$.
  We partially order $\mathcal{F}$ by inclusion.
  Any chain $\mathcal{C} \subseteq \mathcal{F}$ is upper bounded by $(\bigcup_{S \in \mathcal{C}} S) \in \mathcal{F}$ using that TC-convexity is closed under arbitrary nested unions. 
  Therefore, the Kuratowski--Zorn lemma implies that $\mathcal{F}$ has at least one maximal element $H^*$.
  By construction, $H^*$ is TC-convex; hence to prove that $H^*$ is a TC-hemispace we show that $\TTpm^d \setminus H^*$ is TO-convex.
  Note that this implies that it is also TC-convex.   
  Suppose that $\TTpm^d \setminus H^*$ is not TO-convex. 
  Then, there exist $y_1,y_2 \in \TTpm^d \setminus H^*$ and $x_0 \in \shull(y_1,y_2)$ such that $x_0 \in H^*$.
  Furthermore, by the maximality of $H^*$ we get $z \in \whull(H^* \cup \{y_1\})$ and $z \in \whull(H^* \cup \{y_2\})$.
  Hence, by \cref{cor:convexity-finitary}, there exist finite sets $X_1, X_2 \subseteq H^*$ such that $z \in \whull(X_1 \cup \{y_1\})$ and $z \in \whull(X_2 \cup \{y_2\})$.
  By putting $X = \{x_0\} \cup X_1 \cup X_2$, we get $x_0 \in \shull(y_1,y_2)$, $z \in \whull(X \cup \{y_1\})$, $z \in \whull(X \cup \{y_2\})$, and $z \notin \whull(X) \subseteq H^*$. 
  This gives a contradiction with \cref{prop:sand_glass}.
  Since $z \notin G$ was arbitrary, we obtain that $G$ is an intersection of TC-hemispaces.
\end{proof}

\subsection{TC-hemispaces are nearly halfspaces}\label{sec:hemispace_nearly_halfspace}

Next we aim to show that a TC-hemispace is basically a halfspace up to its boundary (\cref{th:hemispace}).
To achieve this, we consider its intersections with orthants.
This allows to use existing work for nonnegative tropical convexity and to distinguish cases based on the sign structure of the points in the TC-hemispace. 

For a cone $X \subseteq \TTpm^d$, a \emph{relative conic TC-hemispace \wrt $X$} is a subset $H \subseteq X \subseteq \TTpm^d$ such that $H$ and $X \setminus H$ are TC-convex cones. Note that if $G \subseteq \TTpm^d$ is a TC-hemispace as well as a cone and $X \subseteq \TTpm^d$ is a TC-convex cone, then $G \cap X$ is a relative conic TC-hemispace \wrt $X$.

We summarize a crucial insight on the structure of tropical hemispaces in the nonnegative tropical orthant $\Tgeq^d$; see \cite{BriecHorvath:2008}, \cite[Section~4]{KatzNiticaSergeev:2014}, \cite[Section~4]{EhrmannHigginsNitica:2016} for a detailed discussion.

\begin{proposition} \label{prop:hemispaces_open_orthant}
  If $G \subseteq \Tgt^d$ is a relative conic TC-hemispace \wrt $\Tgt^d$ and $G \notin\{\emptyset, \Tgt^d\}$, then there exists a vector $(a_1, \dots, a_d) \in \TTpm^{d}$, $(a_1, \dots, a_d) \neq \Zero$, such that $\HC^+(\Zero, a) \cap \Tgt^d \subseteq G \subseteq \Hclosed^{+}(\Zero, a) \cap \Tgt^d$. 
\end{proposition}

\Cref{prop:hemispaces_open_orthant} can be extended to characterize relative hemispaces that are included in the nonnegative orthant and contain (parts of) the boundary of $\Tgeq^d$. The next corollary gives one such extension that is needed in our proofs.

\begin{corollary}\label{prop:hemispaces_single_orthant}
  Let $X = \SetOf{x \in \Tgeq^d}{x_1 > \Zero}$. If $G \subseteq X$ is a relative conic TC-hemispace \wrt $X$ and $G \cap \Tgt^d \notin\{\emptyset, \Tgt^d\}$, then there exists a vector $(a_1, \dots, a_d) \in \TTpm^{d}$, $(a_1, \dots, a_d) \neq \Zero$, such that $\HC^+(\Zero, a) \cap X \subseteq G \subseteq \Hclosed^{+}(\Zero, a) \cap X$.
\end{corollary}
\begin{proof}
By \cref{prop:hemispaces_open_orthant}, there exists a vector $(a_1, \dots, a_d) \in \TTpm^{d}$, $(a_1, \dots, a_d) \neq \Zero$, such that $\HC^+(\Zero, a) \cap \Tgt^d \subseteq G \cap \Tgt^d \subseteq \Hclosed^{+}(\Zero, a) \cap \Tgt^d$. Let $x \in \HC^+(\Zero, a) \cap X$ be any point and let $y \in \Tgt^d$ be such that $y \notin G$. Then, for sufficiently small $\omega \in \Tgt$, the point $x \lplus \omega \odot y \in \Tgt^d$ belongs to $\HC^+(\Zero, a)$, which implies that it belongs to $G$. Hence, by the TC-convexity of $X \setminus G$, we have $x \in G$. This shows that $\HC^+(\Zero, a) \cap X \subseteq G$. Analogously, we get $\HC^{-}(\Zero, a) \cap X \subseteq (X \setminus G)$, which implies that $G \subseteq \Hclosed^{+}(\Zero,a) \cap X$.
\end{proof}

We note that the assumption $G \cap \TT_{>\Zero}^d \notin\{\emptyset, \TT_{>\Zero}^d\}$ implies that $\psupp(a) \neq \emptyset$ and $\nsupp(a) \neq \emptyset$.

In order to prove our characterization of TC-hemispaces we need to generalize \cref{prop:hemispaces_single_orthant} to handle multiple orthants.
We start by characterizing relative hemispaces \wrt the halfspace $X = \SetOf{x \in \TTpm^d}{x_1 > \Zero}$.
To do so, in \cref{lem:trivial+neighbouring+orthant,lem:hemispace+neighbouring+orthants,le:hemispace_in_halfspace} we suppose that $G$ is a relative conic TC-hemispace \wrt $X$.
Furthermore, we suppose that there exists an orthant $\ort$ in $\cl(X)$ such that $G \cap \inter(\ort) \notin \{\emptyset, \inter(\ort)\}$.
Then, $G \cap \ort$ is a relative conic TC-hemispace \wrt $\ort \cap X$.
In particular, by \cref{prop:hemispaces_single_orthant}, there is a vector $(a_1,\dots,a_d) \in \TTpm^{d}$ with $(a_1, \dots, a_d) \neq \Zero$ such that $\HC^+(\Zero, a) \cap \ort \cap X \subseteq G \cap \ort \subseteq \Hclosed^{+}(\Zero, a) \cap \ort \cap X$.
Let $\dort$ be a neighboring orthant of $\ort$ in $X$, i.e., an orthant obtained from $\ort$ by changing one sign (other than the sign of the first coordinate).
In the next two lemmas (\cref{lem:trivial+neighbouring+orthant} and \cref{lem:hemispace+neighbouring+orthants}), we show that the relative TC-hemispaces $G \cap \ort$ in $\ort$ and $G \cap \dort$ in $\dort$ are essentially determined by the same vector $a$.

By suitably flipping signs and permuting variables, we can assume that $\ort = \Tgeq^d$ and $\dort$ differs in the sign of the $d$-th component.
We denote $T = (\ort \cup \dort) \cap X$.
 
In the following proofs, $\omega \in \Tgt$ and $\Omega \in \Tgt$ will mean a sufficiently small and sufficiently big number, respectively. 

\begin{lemma} \label{lem:trivial+neighbouring+orthant}
  There are exactly two possibilities for the part of the relative TC-hemispace $G$ in $Q$ to be trivial: 
  \begin{enumerate}[(i)]
  \item $G \cap \inter(\dort) = \inter(\dort) \Leftrightarrow \nsupp(a) =\{d\} \Leftrightarrow \SetOf{y \in \dort \cap X}{y_d < \Zero} \subseteq G$,
  \item $G \cap \inter(\dort) = \emptyset \Leftrightarrow \psupp(a) = \{d\} \Leftrightarrow \SetOf{y \in \dort \cap X}{y_d < \Zero} \cap G = \emptyset$. 
  \end{enumerate}
  
\end{lemma}
\begin{proof}
To start, suppose that $G \cap \inter(\dort) = \inter(\dort)$ but there is $k \in \nsupp(a) \cap [d-1]$. 
As we assume that $G \cap \inter(\ort) \notin \{\emptyset, \inter(\ort)\}$, there exists $j \in \psupp(a)$. 
We define $x \in \ort, y \in \dort$ and derive $x \lplus y \in \ort$ via
\begin{equation*}
  x_i = \begin{cases} \Omega & i = j \\ \omega & \text{else} \end{cases} \qquad
  y_i = \begin{cases} 2\Omega & i = k \\ \ominus\omega & i = d \\ \omega & \text{else} \end{cases} \qquad
  (x \lplus y)_i = \begin{cases} 2\Omega & i = k \\ \Omega & i = j \\ \omega & \text{else} \end{cases}
\end{equation*}
By construction, $x,y \in G$ but $x \lplus y \not\in G$, as one can see from $a \odot x > \Zero$ and $a \odot (x\lplus y) < \Zero$. This contradicts the fact that $G$ is TC-convex.
Hence, $\nsupp(a) \subseteq \{d\}$ and as $G \cap \inter(\ort) \notin \{\emptyset, \inter(\ort)\}$ we get $\nsupp(a) = \{d\}$. 

Now, assume that $\nsupp(a) = \{d\}$ but there is a point $y \in (Q \cap X) \setminus G$ such that $y_d < \Zero$. Consider two cases: if $y_k > \Zero$ for some $k \in \psupp(a)$, then we define $x, z \in \ort \cap X$ via 
\begin{equation*}
  x_i = \begin{cases} \Omega & i = d \\ y_i & \text{else} \end{cases} \qquad
  z_i = \begin{cases} \omega & i = d \\ y_i & \text{else} \end{cases}
\end{equation*}
Then $z$ is on the line segment between $x$ and $y$, but $x,y \in T \setminus G$ while $z \in G$. 
This contradicts the fact that $G$ is a relative TC-hemispace.  If $y_k = \Zero$ for every $k \in \psupp(a)$, then we define $x, z \in \ort \cap X$ via
\begin{equation*}
  x_i = \begin{cases} |y_d| & i = d \\ \omega & \text{else} \end{cases} \qquad
  z_i = \begin{cases} 2\omega & i = d \\ y_i &i \neq d, y_i \neq \Zero \\ \omega & \text{else} \end{cases}
\end{equation*}
We note that $z \in \faces(x,y)$. Moreover, the assumption on $y$ implies that $a \odot z > \Zero$. Hence, as above we have $x,y \in T \setminus G$ while $z \in G$, contradicting the fact that $G$ is a relative TC-hemispace. The implication $\SetOf{y \in \dort \cap X}{y_d < \Zero} \subseteq G \Rightarrow G \cap \inter(\dort) = \inter(\dort)$ is trivial, proving the first point of the lemma.

\smallskip

The second point follows analogously as it can be obtained by considering the complement of $G$. 
\end{proof}

This leads us to describe the structure of a TC-hemispace with the union of two orthants. 

\begin{lemma} \label{lem:hemispace+neighbouring+orthants}
$\HC^+(\Zero, a) \cap T \subseteq G \cap T \subseteq \Hclosed^{+}(\Zero, a) \cap T$.
\end{lemma}
\begin{proof}
The case of a trivial intersection of $G$ with $\inter(\dort)$ follows directly from \cref{lem:trivial+neighbouring+orthant}. 
Hence, by \cref{prop:hemispaces_single_orthant}, we can assume that there exists a vector $(b_1,\dots,b_d) \in \TTpm^{d} \setminus \{\Zero\}$ such that $\HC^{+}(\Zero,b) \cap \dort \cap X \subseteq G \cap \dort \subseteq \Hclosed^{+}(\Zero, b) \cap \dort \cap X$.
We show with an exhaustive case distinction that the defining vectors $a$ and $b$ agree up to positive scaling; then this concludes the proof of the lemma.

\textbf{Case 1 } (Different sign pattern.)

Let $j$ be the smallest index for which $\tsgn(a_j) \neq \tsgn(b_j)$. 
Breaking the symmetry in $a$ and $b$, we can assume that $j \in \supp(a)$. 

By taking the complement in both orthants, we can assume that $j \in \psupp(a)$ which means $\tsgn(b_j) \in \{\Zero,\ominus\}$.

\textbf{Case 1a} ($j = d$.)

Since the sign  of $a$ and $b$ differ in $d = j \in \psupp(a)$, there is an $r \in \psupp(b) \setminus \{d\}$.
As the relative TC-hemispace in $\inter(\ort)$ is not trivial, there is an index $s \in \nsupp(a)$.
Note that $r,s,d$ are pairwise different due to the minimality assumption on $j$. 
We define $x \in \ort, y \in \dort$ and derive $x \oplus y$ via
\begin{equation*}
 x_i = \begin{cases} \Omega & i = s \\ 2\Omega & i = d \\ \omega & \text{else} \end{cases} \qquad
 y_i = \begin{cases} 0 & i = r \\ \ominus2\Omega & i = d \\ \omega & \text{else} \end{cases} \qquad
 (x \oplus y)_i = \begin{cases} 0 & i = r \\ \Omega  & i = s \\ \bullet2\Omega & i = d \\ \omega & \text{else} \end{cases} \enspace .
\end{equation*}
By construction, $a \odot x > \Zero$, $b \odot y > \Zero$. 
Choosing $z \in \Uncomp(x \oplus y)$ with $z_d = \omega$, we get a point in $G$. 
However, $z \in \ort$ and $a \odot z < \Zero$, a contradiction.

\textbf{Case 1b} ($j \neq d$.) 

Let $k = \min\nsupp(a)$ (which cannot equal $j$ then). 
As the relative TC-hemispace in $\inter(\dort)$ is not trivial, there is an index $r \in \nsupp(b)$ (which could equal $j$).

We define $x \in \ort, y \in \dort$ and derive $x \lplus y \in \ort$ via
\begin{equation*}
 x_i = \begin{cases} 0 & i = k \\ \omega & \text{else} \end{cases} \qquad
 y_i = \begin{cases} \Omega & i = j \\ 0 & i = r, i \neq j \\ \ominus\omega & i = d \\ \omega & \text{else} \end{cases} \qquad
 (x \lplus y)_i = \begin{cases} \Omega & i = j \\ 0 & i = r, i \neq j \\ 0 & i = k\\ \omega & \text{else} \end{cases} 
\end{equation*}
By construction, $a \odot x < \Zero$, $b \odot y < \Zero$ but $a \odot (x\lplus y) > \Zero$, a contradiction to the TC-convexity of $G$. 

\smallskip

From the cases considered so far, we deduce that $a$ and $b$ have the same sign pattern. 
Hence, only the following possibility is remaining. 

\textbf{Case 2}
Now, we have that the sign patterns of $a$ and $b$ are the same. 

If $a$ and $b$ are not the same up to scaling, then there are three indices $i,j,k \in [d]$ such that not all three signs of the respective components agree and among the quotients $a_i \odot b_i^{\odot -1}$, $a_j \odot b_j^{\odot -1}$, $a_k \odot b_k^{\odot -1}$ at least one has a different value from the other two. 
Then there are indices $p,q \in \{i,j,k\}$ with $\tsgn(a_p) = \tsgn(b_p) = \oplus$, $\tsgn(a_q) = \tsgn(b_q) = \ominus$, and $a_p \odot b_q \neq a_q \odot b_p$. 

\textbf{Case 2a} ($d \not\in \{p,q\}$. )

By scaling $a$ and $b$, we can assume that $a_p = b_p = 0$. 
By the symmetry in $a$ and $b$, we can assume that $|a_q| > |b_q|$.

We choose $\xi_a, \xi_b \in \Tgt$ with $\xi_a < |a_q^{\odot -1}| < \xi_b < |b_q^{\odot -1}|$.
With this, we define $x \in \ort, y \in \dort$ and derive $x \lplus y \in \ort$ via
\begin{equation*}
 x_i = \begin{cases} \xi_a & i = q \\ 0 & i = p \\ \omega & \text{else} \end{cases} \qquad
 y_i = \begin{cases} \xi_b & i = q \\ \ominus\omega & i = d \\ 0 & i = p \\ \omega & \text{else} \end{cases} \qquad
 (x \lplus y)_i = \begin{cases} \xi_b & i = q \\ 0 & i = p \\ \omega & \text{else} \end{cases}
\end{equation*}
The choice of $\xi_a$ and $\xi_b$ yields $|\xi_a \odot a_q| < 0$ and $|\xi_b \odot b_q| < 0$ as well as $|\xi_b \odot a_q| > 0$.
This implies that $a \odot x > \Zero$, $b \odot y > \Zero$ and $a \odot (x \lplus y) < \Zero$, a contradiction. 

\textbf{Case 2b} ($d \in \{p,q\}$. )

By taking complements in both orthants, we can assume that $d = p$, so $d \in \psupp(a) = \psupp(b)$. 
Since $\Hclosed^{+}(\Zero, b) \cap \inter(\dort) \neq \emptyset$, there is at least one $r \in \psupp(b) \setminus \{d\}$.
By scaling, we can assume that $a_q = b_q = \ominus 0$.
By the previous case, we have $a_r = b_r$.
By the symmetry in $a$ and $b$, we can suppose that $|a_d| > |b_d|$.
We choose $\xi_r, \xi_q \in \Tgt$ with
\[
a_d > \xi_q > \xi_r \odot a_r = \xi_r \odot b_r > b_d \enspace .
\]
With this, we define $x \in \ort, y \in \dort$ and $z \in \ort$ via
\begin{equation*}
 x_i = \begin{cases} \xi_q & i = q \\ 0 & i = d \\ \omega & \text{else} \end{cases} \qquad
 y_i = \begin{cases} \xi_r & i = r \\ \ominus 0 & i = d, \\ \omega & \text{else} \end{cases} \qquad
z_i = \begin{cases} \xi_q & i = q \\ \xi_r & i = r \\ \omega & \text{else} \end{cases}
\end{equation*}
We have $z \in \faces(x,y)$. Moreover, $a \odot x > \Zero$, $b \odot y > \Zero$, but $a \odot z < \Zero$, a contradiction. 

\end{proof}

We now extend \cref{lem:hemispace+neighbouring+orthants} from a union of two orthants to the entire set $X$.
\begin{lemma}\label{le:hemispace_in_halfspace}
We have $\HC^+(\Zero, a) \cap X \subseteq G \subseteq \Hclosed^{+}(\Zero, a) \cap X$.
\end{lemma}
\begin{proof}
Let $\dort$ be any orthant in $\cl(X)$.
We want to show that $\HC^+(a) \cap \dort \cap X \subseteq G \cap \dort \subseteq \Hclosed^{+}(a) \cap \dort \cap X$. Let $\ortsig \in \{\oplus, \ominus\}^{d}$ be the sign vector corresponding to $\dort$, i.e., the vector such that $\inter(\dort) = \SetOf{x \in \TTpm^d}{\forall k, \tsgn(x_k) = \ortsig_k}$.
For the purposes of this proof, we say that an orthant is \emph{good} if $\HC^+(\Zero, a)$ subdivides its interior in a nontrivial way; that is 
if there exists a pair $(k,l) \in \supp(a)$ such that $\ortsig_k \tsgn(a_k) \neq \ortsig_l \tsgn(a_l)$.
Let $r = |\SetOf{k \in [d]}{k \in \supp(a), \ortsig_k = \ominus}|$.
We divide the proof into two cases.

\textbf{Case 1} (We have $|\supp(a)| \ge 3$ or $r \in \{0,1\}$.)

In this case, we start by proving that there exists a sequence of orthants $\Tgeq^d = \dort_0, \dort_1, \dots, \dort_p = \dort$ in $\cl(X)$ such that $\dort_i, \dort_{i+1}$ differ by flipping one sign and the $\dort_i$ are good for all $i \le p-1$.
Such a sequence can be obtained in the following way.
The orthant $\dort_0$ is good by the assumption of the lemma.
To go from $\dort_0$ to $\dort$, we first flip the signs in $\SetOf{k \notin \supp(a)}{\ortsig_k = \ominus}$ (in any order) and note that all orthants obtained in this way are good.
Then, we flip the signs in $\SetOf{k \in \supp(a)}{\ortsig_k = \ominus}$.
If $r \in \{0,1\}$, then this already proves the existence of the sequence, because this step either does nothing or flips one sign.
If $|\supp(a)| \ge 3$ and $r \ge 3$, then there are at most two orthants that are not good and could be obtained at this step.
However, since the graph of the $r$-dimensional hypercube is $r$-vertex-connected and $r \ge 3$, there is a way of flipping the signs in $\SetOf{k \in \supp(a)}{\ortsig_k = \ominus}$ that avoids going through these two orthants (except, possibly, for the last step, since $\dort$ may be not good).
If $|\supp(a)| \ge 3$ and $r = 2$, then we let $k,l$ be the two indices in $\supp(a)$ such that $\ortsig_k = \ortsig_l = \ominus$ and $j$ be any index in $\supp(a)$ such that $\ortsig_j = \oplus$.
If $\tsgn(a_j) = \tsgn(a_k)$ or $\tsgn(a_k) = \tsgn(a_l)$, then we first flip the sign of the component indexed by $k$ and subsequently we flip the sign of the component indexed by $l$.
If $\tsgn(a_j) = \tsgn(a_l)$, then we start by flipping the sign of the component indexed by $l$ and subsequently we flip the sign of the component indexed by $k$.

Given the sequence $Q_0, \dots, Q_p$ we apply \cref{lem:hemispace+neighbouring+orthants} to $T = (\dort_0 \cup \dort_1) \cap X$. This gives $\HC^+(a) \cap \dort_1 \cap X \subseteq G \cap \dort_1 \subseteq \Hclosed^{+}(a) \cap \dort_1 \cap X$. Furthermore, since $\dort_1$ is good, \cref{lem:trivial+neighbouring+orthant} implies that $G \cap \inter(\dort_1) \neq \{\emptyset, \inter(\dort_1)\}$. Hence, by applying \cref{lem:hemispace+neighbouring+orthants,lem:trivial+neighbouring+orthant} to $T = (\dort_1 \cup \dort_2) \cap X$ we get that $\HC^+(a) \cap \dort_2 \cap X \subseteq G \cap \dort_2 \subseteq \Hclosed^{+}(a) \cap \dort_2 \cap X$ and $G \cap \inter(\dort_2) \neq \{\emptyset, \inter(\dort_2)\}$. By repeating this reasoning, we obtain the claim.

\textbf{Case 2} (We have $|\supp(a)| = r = 2$.)

In this case, let $\{k,l\} = \supp(a)$ and let $\dort_{\{k\}}, \dort_{\{l\}}, \dort_{\{k,l\}}$ be the orthants obtained from $\Tgeq^d$ by flipping the signs of the components indexed by $k$ and $l$.
Since $\Tgeq^d$ is good, we can suppose (up to permuting $k$ and $l$) that $\tsgn(a_k) = \oplus$ and $\tsgn(a_l) = \ominus$.
Then, \cref{lem:trivial+neighbouring+orthant} shows that $G \cap \inter(\dort_{\{k\}}) = \emptyset$ and $G \cap \inter(\dort_{\{l\}}) = \inter(\dort_{\{l\}})$.

We first show that $G \cap  \inter(\dort_{\{k,l\}}) \notin \{\emptyset, \inter(\dort_{\{k,l\}})\}$.
We define $x \in \dort_{\{k,l\}}, y \in \Tgeq^d$ and derive $x \lplus y \in \dort_{\{k\}}$ via
\begin{equation*}
  x_i = \begin{cases} \ominus \omega & i = k \\ \ominus \Omega & i = l \\ 0 & \text{else} \end{cases} \quad
  y_i = \begin{cases} 0 & i = k \\ \Omega  & i = l \\ 0 & \text{else} \end{cases} \quad
  (x \lplus y)_i = \begin{cases} 0 & i = k \\ \ominus \Omega & i = l \\ 0 & \text{else} \end{cases}
\end{equation*}
We have $y \notin G$ and $x \lplus y \in G$. Therefore, $x \in G \cap \inter(\dort_{\{k,l\}})$.
Similarly, if we define 
\begin{equation*}
  x_i = \begin{cases} \ominus \Omega & i = k \\ \ominus \omega & i = l \\ 0 & \text{else} \end{cases} \quad
  y_i = \begin{cases} \Omega & i = k \\ 0  & i = l \\ 0 & \text{else} \end{cases} \quad
  (x \lplus y)_i = \begin{cases} \ominus \Omega & i = k \\ 0 & i = l \\ 0 & \text{else} \end{cases}
\end{equation*}
then $y \in G$ and $x \lplus y \notin G$.
Hence, we have $x \notin G$ and $x \in \inter(\dort_{\{k,l\}})$.
This implies $G \cap  \inter(\dort_{\{k,l\}}) \notin \{\emptyset, \inter(\dort_{\{k,l\}})\}$.

Now, by \cref{prop:hemispaces_single_orthant}, there exists a vector $b \in \TTpm^d \setminus \{\Zero\}$ such that $\HC^{+}(\Zero,b) \cap \dort_{\{k,l\}} \subseteq G \cap \dort_{\{k,l\}} \subseteq \Hclosed^{+}(\Zero,b) \cap \dort_{\{k,l\}}$.
By applying \cref{lem:trivial+neighbouring+orthant} to $\dort_{\{k,l\}} \cup \dort_{\{k\}}$ and $\dort_{\{k,l\}} \cup \dort_{\{l\}}$ we get $\nsupp(b) = \nsupp(a) = \{l\}$ and $\psupp(b) = \psupp(a) = \{k\}$. 
By scaling $a$ and $b$, we can assume that $a_k = b_k = 0$.
We want to show that $a_l = b_l$.
For contradiction, assume that $|a_{l}| < \eta^{\odot -1} < \xi^{\odot -1} < |b_{l}|$ for some $\eta, \xi \in \Tgt$.
We define $x \in \dort_{\{k,\ell\}}, y \in \Tgeq$ and derive $x \lplus y \in \dort_{\{k\}}$ via
\begin{equation}\label{eq:dim_2_orthants}
  x_i = \begin{cases} \ominus 0 & i = k \\ \ominus \xi & i = l \\ \omega & \text{else} \end{cases} \quad
  y_i = \begin{cases} 0 & i = k \\ \eta  & i = l \\ \omega & \text{else} \end{cases} \quad
  (x \lplus y)_i = \begin{cases} \ominus 0 & i = k \\ \eta & i = l \\ \omega & \text{else} \end{cases} \enspace .
\end{equation}
We have $x,y \in G$ (as $b \odot x > \Zero$ and $a \odot y > \Zero$) but $x \lplus y \notin G$, which gives a contradiction.
Analogously, if $|a_l| > \eta^{\odot -1} > \xi^{\odot -1} > |b_l|$ for some $\xi,\eta \in \Tgt$, and we define $x,y$ as in \cref{eq:dim_2_orthants}, then $x,y\notin G$ but $y \lplus x \in G$, giving a contradiction.
Therefore, we have $a = b$.

To finish the proof, note that we can go from $\dort_{\{k,l\}}$ to $\dort$ by a sequence of good orthants obtained by flipping the signs that do not belong to $\supp(a)$.
This gives the claim by the same reasoning as in Case 1.
\end{proof}

\begin{figure}
\includegraphics{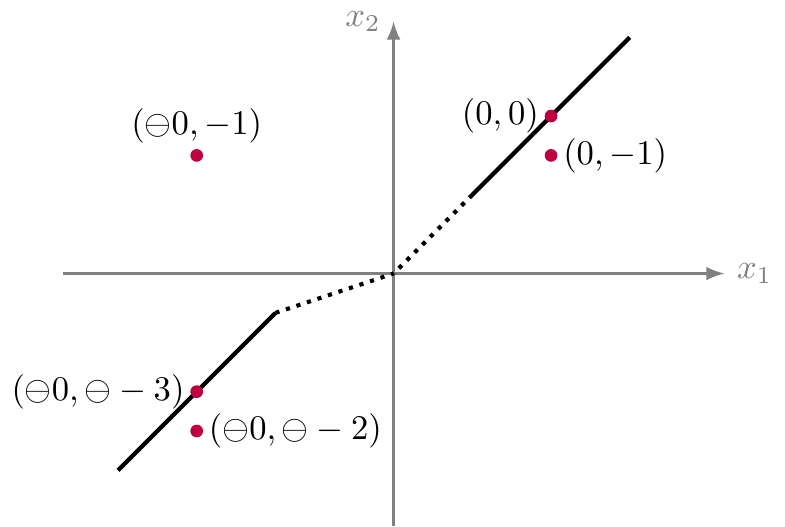}
    \caption{Sketch of the positions for the first part of \cref{ex:opposite-orthants-coefficients}. }
    \label{fig:opposite-orthants-coefficients}
    \end{figure}

\begin{example} \label{ex:opposite-orthants-coefficients}
We illustrate the last case in the proof of \cref{le:hemispace_in_halfspace}.

Let $a = (0,\ominus 0)$ and $b = (0, \ominus 3)$.
We take $\eta = -1$, $\xi = -2$, so $x = (\ominus 0, \ominus (-2))$, $y = (0,-1)$.
Then $a \odot y = 0$, $b \odot x = 1$ so $x,y \in G$ but $x \lplus y = (\ominus 0, -1)$ so $a \odot (x \lplus y) = \ominus 0, b \odot (x \lplus y) = \ominus 2$ and $x \lplus y \notin G$.

Likewise, if $a = (0,\ominus 3)$ and $b = (0, \ominus 0)$, then we take $\eta = -2, \xi = -1$, so $x = (\ominus 0, \ominus (-1)), y = (0,-2)$.
We have $a \odot y = \ominus 1$, $b \odot x = \ominus 0$ so $x,y \notin G$ but $y \lplus x = (0, \ominus (-1))$ satisfies $a \odot (y \lplus x) = 2$, $b \odot (y \lplus x) = 0$ and $y \lplus x \in G$.
 \end{example}

Now, we combine the last lemmas to prove that a TC-hemispace is nearly a halfspace, as stated in \cref{th:hemispace}.

\begin{proof}[Proof of \cref{th:hemispace}]
We consider two cases.
  
\textbf{Case 1} (There exists an orthant $\ort \in \TTpm^d$ such that $G \cap \inter(\ort) \notin \{\emptyset, \inter(\ort)\}$.)

By flipping signs, we assume that $\ort = \Tgeq^d$.
Let $X = \SetOf{x \in \TTpm^{d+1}}{x_{d+1} > \Zero}$.
We define two sets $\coG, \ccG \subseteq X$ by putting $\coG = \wcone \SetOf{(z,0)}{z \in G} \setminus \{\Zero\}$ and $\ccG = \wcone \SetOf{(z,0)}{z \notin G} \setminus \{\Zero\}$.
Since both $G$ and its complement are TC-convex, we have 
\begin{align*}
\coG &= \SetOf{\lambda \odot (z,0)}{\lambda \in \Tgt, z \in G} \, , \\
\ccG &= \SetOf{\lambda \odot (z,0)}{\lambda \in \Tgt, z \notin G} \, .
\end{align*}
In particular, $\coG = \wcone \SetOf{(z,0)}{z \in G} \cap X$, which implies that $\coG$ is a TC-convex cone. 
Analogously, $\ccG$ is a TC-convex cone.

Furthermore, we have $\coG \cup \ccG = X$ because any point $x = (x_1, \dots, x_{d+1})$ in $X$ can be written as $x = x_{d+1} \odot (y_1, \dots, y_d, 0)$, where $y_i = x_i \odot x_{d+1}^{\odot -1}$ and $y$ either belongs to $G$ or not.
Likewise, we have $\coG \cap \ccG = \emptyset$ because if $x \in \coG \cap \ccG$, then the point $y$ defined as above belongs both to $G$ and its complement, which gives a contradiction.
Thus, $\coG$ is a relative conic TC-hemispace \wrt $X$ and $\coG \cap \Tgt^{d+1} \notin \{\emptyset, \Tgt^{d+1}\}$.

Hence, by \cref{le:hemispace_in_halfspace}, there exists a vector $(a_1, \dots, a_{d+1}) \in \TTpm^{d+1}$ such that $(a_1, \dots, a_{d+1}) \neq \Zero$ and $\HC^{+}(\Zero, a) \cap X \subseteq \coG \subseteq \Hclosed^{+}(\Zero, a) \cap X$.
In particular,
\begin{align*}
&\SetOf{x \in \TTpm^d}{a_{d+1} \oplus a_{1} \odot x_1 \oplus \dots \oplus a_{d} \odot x_d  > \Zero } \subseteq G \\
&\subseteq \SetOf{x \in \TTpm^d}{a_{d+1} \oplus a_{1} \odot x_1 \oplus \dots \oplus a_{d} \odot x_d \teq \Zero} \, .
\end{align*}
Since $\nsupp(a) \neq \emptyset$ and $\psupp(a) \neq \emptyset$, we have $(a_1, \dots, a_d) \neq \Zero$.

  \smallskip

  \textbf{Case 2} ($G \cap \inter(\ort) \in \{\emptyset, \inter(\ort)\}$ for every orthant $\ort \subseteq \TTpm^d$.)
  
  By taking suitable combinations, we see that $G$ cannot only consist of parts of coordinate hyperplanes. 
  Hence, by flipping signs, we can assume that $\Tgt^d \subseteq G$.  
  We show that there is a $k \in [d]$ such that
  \begin{equation}
    \SetOf{x \in \TTpm^d}{x_k > \Zero} \subseteq G \subseteq \SetOf{x \in \TTpm^d}{x_k \geq \Zero} \; .
  \end{equation}

Suppose the right inclusion does not hold, then for each $\ell \in [d]$ there is an orthant $\ort^{(\ell)}$ with $\inter(\ort^{(\ell)}) \subseteq G$ whose $\ell$-th component is negative.
As each $\ort^{(\ell)}$ contains points close to the negative of the $\ell$-th tropical unit vector, taking convex combinations of these points and $(0,\dots,0) \in \Tgt^d$ yields points in the interior of every orthant in $\TTpm^d$.
This would imply $G = \TTpm^d$, which was excluded. 
Hence, $G \subseteq \SetOf{x \in \TTpm^d}{x_k \geq \Zero}$.
Therefore, for half of the orthants the interior is contained in $G$ and, by the same argument, for half of the orthants the interior is contained in the complement $G'$ of $G$.
But as $G$ and $G'$ cover the whole space, we get that $G' \subseteq \SetOf{x \in \TTpm^d}{x_k \leq \Zero}$.
By taking again the complement, this concludes the second case. 
\end{proof}

As TO-hemispaces are also TC-hemispaces we immediately get the following. 

\begin{corollary}
If $G \subseteq \TTpm^d$ is a TO-hemispace and $G \notin\{\emptyset, \TTpm^d\}$, then there exists a vector $(a_0, \dots, a_d) \in \TTpm^{d+1}$, $(a_1, \dots, a_d) \neq \Zero$, such that $\HC^+(a) \subseteq G \subseteq \Hclosed^{+}(a)$.
\end{corollary}

\smallskip

\subsection{The boundary of TC-hemispaces}\label{sec:boundary_hemispace}

\begin{figure}[tbh]
\includegraphics{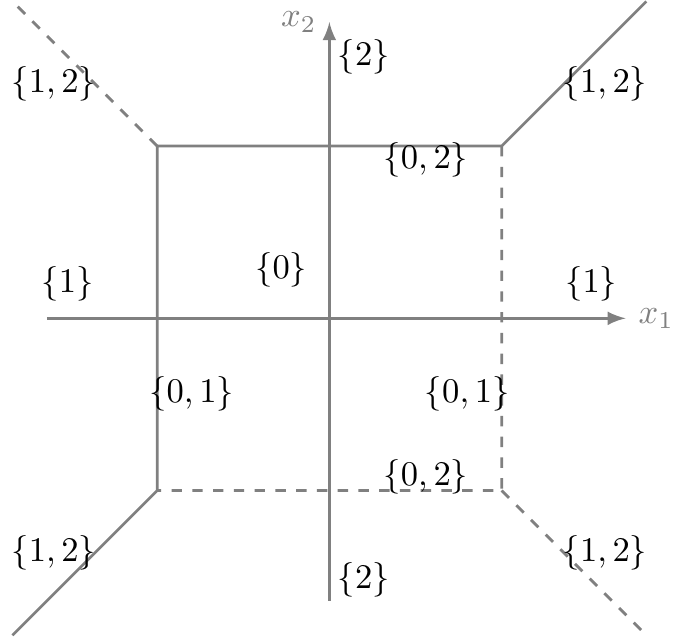}
\includegraphics{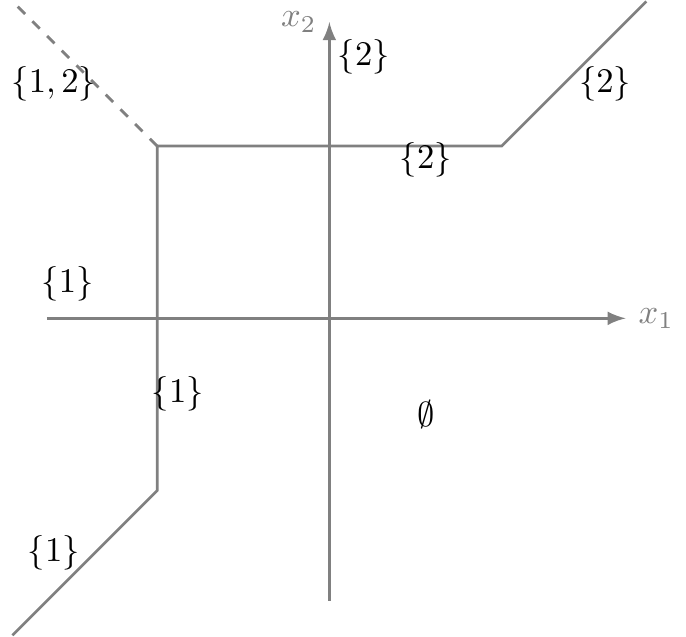}
  \caption{$\domin$ and $\posdomin$ for $(\ominus 0, \ominus -1, -1)$}
\end{figure}

The boundary of a halfspace is determined by the components where the maximum is attained in the tropical linear form defining the halfspace. 
Hence, for $a = (a_0, \dots, a_d) \in \TTpm^{d+1}$ and $x \in \TTpm^d$, we denote
\begin{align*}
\domin(a,x) &=\SetOf{k \in [d]_0}{a_k \odot x_k \neq \Zero \, \land \, \forall \ell \in [d]_0, |a_k| + |x_k| \ge |a_{\ell}| + |x_{\ell}|} \\
&\!= \argmax_{k \in [d]_0}{|a_k \odot x_k|} \cap \supp(a) \cap \supp(x)
\end{align*}
and
\[
\posdomin(a,x) = \SetOf{k \in \domin(a,x)}{a_k \odot x_k > \Zero} \, , 
\]
where we use the convention that $x_0 = 0$.
Note that
\[
\posdomin(a,x) \subseteq \domin(a,x) \subseteq \supp(x) \cup \{0\} .
\]

For a fixed $a$, comprising the points with \begin{enumerate*} \item the same set $\domin(a,x)$, \item in the same orthant \end{enumerate*} yields a cell decomposition of $\TTpm^d$. 
Taking the common refinement of the cell decompositions for several possible $a$ yields a generalization of the decomposition
which was studied under the name `type decomposition'~\cite{DevelinSturmfels:2004} or `covector decomposition'~\cite{JoswigLoho:2016}.

Right from the definition, we obtain some basic properties.
\begin{corollary} \label{cor:argmax_basic_properties}
We have
  \begin{enumerate}[(i)]
  \item If $0 \not\in \domin(a,x) \cup \domin(a,\rho\odot x)$, then $\domin(a, \rho \odot x) = \domin(a,x)$ for $\rho \in \TTpm \setminus \{\Zero\}$; 
  \item $\domin(a,x) = \emptyset$ if and only if $a \odot (0,x) = \Zero$; 
  \item $\domin(a,x) = [d]_0$ if and only if $|a_0| = |a_1| + |x_1| = \dots = |a_d| + |x_d| \neq \Zero$.
  \end{enumerate}
\end{corollary}

\smallskip

For this subsection, we fix a TC-hemispace $G \subseteq \TTpm^d$ with $G \neq \{\emptyset, \TTpm^d\}$ and let $(a_0, \dots, a_d) \in \TTpm^{d+1}$, $(a_1, \dots, a_d) \neq \Zero$, such that $\HC^+(a) \subseteq G \subseteq \Hclosed^{+}(a)$, which exists by \cref{th:hemispace}.
Furthermore, we fix a point $x \in G$ with $\domin(a,x) \neq \emptyset$.
Observe that we have $\posdomin(a,x) \neq \emptyset$.
Indeed, if $\domin(a,x) \neq \emptyset$ and $\posdomin(a,x) = \emptyset$, then $x \in \HC^{-}(a)$, giving a contradiction with $x \in G$.

We first state the main insight of this subsection which shows that the structure of hemispaces behaves rather well with respect to fixing $\domin$.  
In particular, it extends~\cite[Lemma 3.1]{EhrmannHigginsNitica:2016}.

\begin{proposition}\label{le:bigger_pos_argmax}
  Let $y \in \TTpm^d$ be such that $\domin(a,x) = \domin(a,y)$ and $\posdomin(a,x) \subseteq \posdomin(a,y)$.
  Then, we have $y \in G$.
\end{proposition}

\smallskip

To prove this proposition, we start with a statement that can be seen orthantwise.
It shows closure under positive-tropical scalar multiplication. 

\begin{lemma}\label{le:translation_hemispace}
For $\rho \in \RR$ with $\domin(a,x) = \domin(a,\rho \odot x)$, we have $\rho \odot x \in G$.
\end{lemma}
\begin{proof}
The claim is trivial for $\rho = 0$, so we start with $\rho < 0$. 
If $0 \in \domin(a,x)$, then the equality $\domin(a,x) = \domin(a,\rho \odot x)$ implies that $\domin(a,\rho \odot x) = \{0\} = \posdomin(a, \rho \odot x)$ and so $\rho \odot x \in \HC^{+}(a) \subseteq G$.
If $0 \notin \domin(a,x)$, then we fix $k \in\posdomin(a, x) \cap [d] = \posdomin(a, x)$. 
We define a point $y \in \TTpm^d$ as $y = \rho \odot x_k \odot e_k$ where $e_k$ is the $k$-th tropical unit vector. 
We have $y \in \HC^{+}(a)$ and so $y \in G$.
Furthermore, $\rho \odot x = \rho \odot x \lplus y$ and therefore $\rho \odot x \in G$ by the TC-convexity of $G$.

If $\rho > 0$, then we define $x' = \rho \odot x$, $G' = \TTpm^d \setminus G$, and $a' = \ominus a$.
Suppose that $x' \in G'$.
Then, $G'$ is a TC-hemispace such that $\HC^{+}(a') \subseteq G' \subseteq \Hclosed^{+}(a')$ and we have $x = (-\rho) \odot x'$.
Hence, the same reasoning as above shows that $x \in G'$, giving a contradiction. 
\end{proof}

Next, we demonstrate how the existence of a point in $G$ with a specific sign pattern implies that other points with a similar sign pattern also have to lie in $G$. 

\begin{lemma}\label{le:good_sup_slice}
For $y \in \TTpm^d$ with $\domin(a,y) = \domin(a,x)$ and $\tsgn(y_k) = \tsgn(x_k)$ for all $k \in \domin(a,x) \cap [d]$, we have $y \in G$. 
\end{lemma}
\begin{proof}
As noted above, $\posdomin(a,x) \neq \emptyset$, so we fix $k \in \posdomin(a,x)$ and consider the following two cases.

\textbf{Case 1} ($y_j = x_j$ for all $j \in \domin(a,x) \cap [d]$.) 

\textbf{Case 1a} ($|x_j| \le |y_j|$ for all $j \notin \domin(a,x)$.)

Define $M = 0$ if $k = 0$ and $M = 1$ otherwise.
Let $z \in \TTpm^d$ be the point defined as 
\begin{equation*}
z_i = \begin{cases}
\Zero &\text{if $i \in \domin(a,y) \setminus \{k\}$,}\\
M \odot y_{i} &\text{otherwise.}
\end{cases}
\end{equation*}
Then, $z$ satisfies $\domin(a,z) = \posdomin(a,z) = \{k\}$ as one sees by taking scalar product with $a$. In particular, we have $z \in \HC^{+}(a) \subseteq G$.
Moreover, $|x_j| \le |y_j|$ for all $j \notin \domin(a,x)$ implies that $y = (-M) \odot z \lplus x$, and so $y \in G$. 

\textbf{Case 1b} (There exists an $\ell \notin \domin(a,x)$ such that $|x_{\ell}| > |y_{\ell}|$.)

Let $u \in \TTpm^d$ be any point that satisfies $u_j = x_j$ for all $j \in \domin(a,x)$ and $|u_j| = |x_j|$ otherwise.
Then, we have $\domin(a,u) = \domin(a,x)$ and $u \in G$ by the previous case. 
Since $u$ was arbitrary, \cref{ex:vertical-boxes} implies that the point $v \in \TTpm^d$ defined as 
\[
v_i = \begin{cases}
x_i &\text{if $i \in \domin(a,x)$,}\\
\Zero &\text{otherwise}
\end{cases}
\]
belongs to $G$.
Now we can apply the Case 1a with $v$ instead of $x$ since $\Zero = |v_j| \leq |y_j|$ for all $j \not\in \domin(a,x)$. 
Hence, we obtain $y \in G$. 

\textbf{Case 2} (There exists an $\ell \in \domin(a,x) \cap [d]$ such that $y_{\ell} \neq x_{\ell}$.)

Let $\rho = |y_{\ell}| \odot |x_{\ell}|^{\odot -1} \in \RR$ and consider the point $x' = \rho \odot x$.
Then, for each $j \in \domin(a,x) \cap [d]$ we have 
\begin{equation} \label{eq:argmax-scaled-points}
|a_j| \odot |x'_j| = \rho \odot |a_j| \odot |x_j| = \rho \odot |a_{\ell}| \odot |x_{\ell}| = |a_{\ell}| \odot |y_{\ell}| = |a_j| \odot |y_j| \,.
\end{equation}
This yields $x'_j = y_j$ for $j \in \domin(a,x) \cap [d]$ because $\tsgn(x_j) = \tsgn(y_j)$. 

By the case assumption, we have $\max_{k \in [d]_0}{|a_k \odot x_k|} \neq \max_{k \in [d]_0}{|a_k \odot y_k|}$ as $y_{\ell} \neq x_{\ell}$ but $\domin(a,y) = \domin(a,x)$.
In particular, $0 \notin \domin(a,y)$.
Furthermore, \cref{eq:argmax-scaled-points} gives $|a_0| < |a_{\ell}| \odot |y_\ell| = |a_{\ell}| \odot |x'_{\ell}|$, so $0 \not\in \domin(a,x')$. 
This implies that $0 \not\in \domin(a,x) \cup \domin(a,x')$. 
Therefore, we have $\domin(a, x') = \domin(a,x)$ by \cref{cor:argmax_basic_properties} and \cref{le:translation_hemispace} implies $x' \in G$. 
Since $x'_j = y_j$ for all $j \in \domin(a,x) \cap [d]$, Case 1 shows that $y \in G$.
\end{proof}

The structure of the boundary of TC-hemispaces relies heavily on the support of the halfspaces sandwiching it.
We discuss this in a simple example which is visualized in \cref{fig:hemispace_boudnaries}.

\begin{example}
Suppose that $G \subset \TTpm^2$ is a TC-hemispace such that 
\begin{equation}\label{eq:simple_hemi1}
\SetOf{x \in \TTpm^2}{x_1 > 0} \subseteq G \subseteq \SetOf{x \in \TTpm^2}{x_1 \ge 0}\, .
\end{equation}
Furthermore, suppose that there is a point $x \in G$ such that $x_1 = 0$ and let $y \in \TTpm^2$ be any point such that $y_1 = 0$. Then, for $a = (\ominus 0, 0, \Zero)$ we get $\domin(a,x) = \{0,1\} = \domin(a,y)$ and $\tsgn(x_1) = \tsgn(y_1)$, so \cref{le:good_sup_slice} implies that $y \in G$. Hence, there are only two TC-hemispaces that satisfy~\cref{eq:simple_hemi1}, namely $G = \SetOf{x \in \TTpm^2}{x_1 > 0}$ and $G = \SetOf{x \in \TTpm^2}{x_1 \ge 0}$. By contrast, suppose that $G' \subset \TTpm^2$ is a TC-hemispace such that 
\begin{equation}\label{eq:simple_hemi2}
\SetOf{x \in \TTpm^2}{x_1 > \Zero} \subseteq G' \subseteq \SetOf{x \in \TTpm^2}{x_1 \ge \Zero}\, .
\end{equation}
Furthermore, suppose that there is a point $x' \in G'$ such that $x'_1 = \Zero$ and let $y \in \TTpm^2$ be any point such that $y_1 = \Zero$. Then, for $a = (\Zero, 0, \Zero)$ we get $\domin(a,x) = \emptyset$ so the assumption of \cref{le:good_sup_slice} is not satisfied and we cannot deduce anything about $y$ using this lemma. Indeed, if $y_2 < x'_2$, then
\[
G' = \SetOf{x \in \TTpm^2}{x_1 > \Zero} \cup \SetOf{x \in \TTpm^2}{x_1 = \Zero, x_2 \ge x'_2}
\]
is a TC-hemispace that does not contain $y$ and
\[
G' = \SetOf{x \in \TTpm^2}{x_1 > \Zero} \cup \SetOf{x \in \TTpm^2}{x_1 = \Zero, x_2 \le x'_2}
\]
Is a TC-hemispace that contains $y$. In particular, there are infinitely many TC-hemispaces that satisfy~\cref{eq:simple_hemi2}.
\end{example}

\begin{figure}[th]
\includegraphics{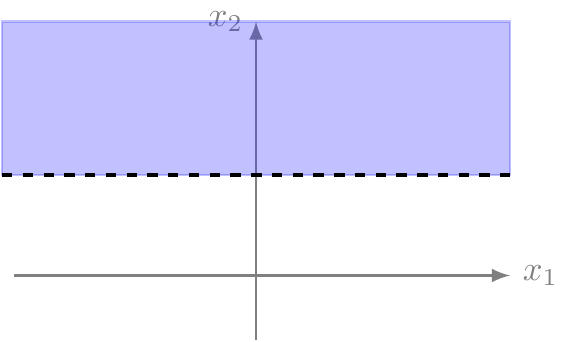}
\includegraphics{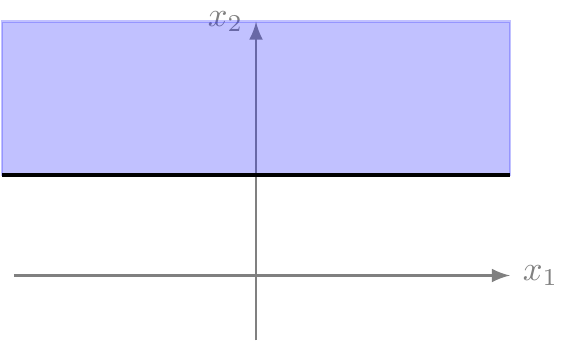}
\includegraphics{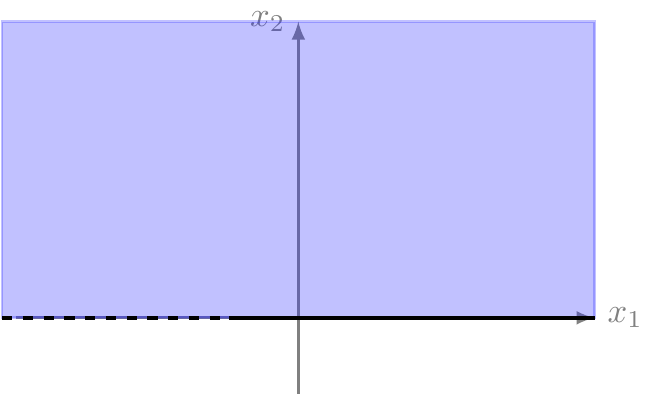}
\caption{Possibilities of boundary of TC-hemispace sandwiched between $x_1 \geq 0$ and $x_1 > 0$ or $x_1 \geq \Zero$ and $x_1 > \Zero$. }\label{fig:hemispace_boudnaries}
\end{figure}

We strengthen the former lemma to even more sign patterns. 

\begin{lemma}\label{le:pos_dir_hemispace}
Define a point $s \in \TSS^d$ as
\[
\forall \ell \in [d], \, s_{\ell} = \begin{cases}
x_{\ell} &\text{if $\ell \in \posdomin(a,x)$,}\\
{\bullet}x_{\ell} &\text{otherwise,}
\end{cases}
\]
then $\Uncomp(s) \subseteq G$.
\end{lemma}
\begin{proof}
By \cref{ex:vertical-boxes}, it is enough to show that the vertices of $\Uncomp(s)$ belong to $G$.
Note that the vertices of $\Uncomp(s)$ are precisely the points $y \in \TTpm^d$ that satisfy $|y_j| = |x_j|$ for all $j \in [d]$ and $\posdomin(a,x) \subseteq \posdomin(a,y)$.
To show that every such point belongs to $G$, we proceed by induction over $m = |\posdomin(a,y) \setminus  \posdomin(a,x)|$.

If $m = 0$, then we have $y_j = x_j$ for all $j \in \domin(a,x)$ and so the claim follows from \cref{le:good_sup_slice}.

If $m> 0$, there is a $k \in \posdomin(a,y) \setminus \posdomin(a,x)$. 
Consider the point $z \in \TTpm^d$ defined as
\[
 z_i = \begin{cases}
x_i &\text{if $i = k$,}\\
y_i &\text{otherwise.}
\end{cases}
\]
Then, we have $z \in G$ be the induction hypothesis. Furthermore, since $|a_0| < 1 + |a_k| + |y_k|$, the point $w \in \TTpm^d$ defined as
\begin{equation*}
 w_i = \begin{cases}
1 \odot y_k &\text{if $i = k$,}\\
\Zero &\text{otherwise}
\end{cases}
\end{equation*}
belongs to $\HC^{+}(a)$.
In particular, $w \in G$.
Since $y = (-1) \odot w \lplus z$, we get $y \in G$.
\end{proof}

Now, we are ready to prove \cref{le:bigger_pos_argmax}. 

\begin{proof}[Proof of \cref{le:bigger_pos_argmax}]
Consider the point $z \in \TTpm^d$ defined as $z_j = \tsgn(y_j)|x_j|$ for all $j \in [d]$.
Since $\posdomin(a,x) \subseteq \posdomin(a,y)$, we have $z_j = x_j$ for all $j \in \posdomin(a,x)$. 
With $|z_j| \leq |x_j|$ for all $j \in [d]$, \cref{le:pos_dir_hemispace} implies $z \in G$. 
Furthermore, $\domin(a,z) = \domin(a,x) = \domin(a,y)$ and $\tsgn(z_j) = \tsgn(y_j)$ for all $j \in \domin(a,z) \cap [d]$.
Hence, we have $y \in G$ by \cref{le:good_sup_slice}. 
\end{proof}

\section{Lifts of signed halfspaces}
\label{sec:signed+halfspaces+lifts}

We will make fundamental use of the correspondence between tropical halfspaces and halfspaces defined over non-Archimedean valued fields.
In the following, we consider the field of \emph{(generalized) real Puiseux series} $\KK = \puiseux{\RR}{t}$, whose elements 
\begin{equation}\label{eq:puiseux}
\bm \gamma = \sum c_i t^{a_i} \ , \ a_i,c_i \in \RR \ , \ a_0 > a_1 > a_2 > \dots
\end{equation}
are formal power series with real exponents and such that the sequence $(a_i)_i$ is either finite or unbounded.
The addition and multiplication of Puiseux series are defined in the natural way.
Furthermore, given a series $\bm \gamma$ as in~\cref{eq:puiseux}, we say that $c_0$ is its \emph{leading coefficient} and we denote by $\lc \colon \KK \to \RR$ the map that sends a Puiseux series to its leading coefficient, with the convention that $\lc(0) = 0$.
We also say $\bm \gamma$ is \emph{positive} if its leading coefficient is positive.
This makes $\KK$ an ordered field via $\bm \gamma > \bm \delta$ if and only if $\bm \gamma - \bm \delta$ is positive.
It is known that $\KK$ is a real closed field~\cite{Markwig:2010} and this remains true even if one considers a subfield formed by Puiseux series that are absolutely convergent for sufficiently large $t$~\cite{VanDenDriesSpeissegger:1998}.
All our results are valid for both of these fields.
The crucial property of a real closed field is that for `well-structured' statements (in the sense of model theory), they behave exactly as the `usual' real numbers via Tarski's principle~\cite{Tarski:1948},~\cite[Corollary~3.3.16]{Marker:2002}. 

The field of Puiseux series is linked with signed tropical numbers via the signed valuation map.

\begin{definition}
The map $\sval \colon \KK \to \TTpm$ sends a Puiseux series as in \cref{eq:puiseux} to its \emph{signed valuation},
\[
\sval(\bm \gamma) = \begin{cases}
a_0 &\text{if } \bm \gamma > 0\\
\Zero &\text{if } \bm \gamma  = 0\\
\ominus a_0 &\text{if } \bm \gamma < 0
\end{cases} \enspace .
\]
We extend $\sval$ to vectors in $\KK^d$ componentwise by putting $\sval(\bm x) = \bigl(\sval(\bm x_i)_{i \in [d]}\bigr)$. 
\end{definition}

General properties of the images of semialgebraic sets under the signed valuation map were studied in~\cite{JellScheidererYu:2018}, see also~\cite[Section~4]{AllamigeonGaubertSkomra:2020}.

We fix our notation for halfspaces over Puiseux series.

\begin{definition}
  For a vector $(\bm a_0, \bm a_1,\dots, \bm a_d) \in \KK^{d+1}$ such that $(\bm a_1,\dots, \bm a_d) \neq 0$ we define the \emph{(closed affine) halfspace} by
\begin{equation}
   \bHclosed^+(\bm a) = \SetOf{\bm x \in \KK^d}{\bm a \cdot \begin{pmatrix} 0 \\ \bm x \end{pmatrix} \ge 0} \enspace .
\end{equation}
Furthermore, we denote $\bHclosed^-(\bm a) = \bHclosed^+(- \bm a)$.
\end{definition}

After deriving basic properties and setting some terminology in \cref{sec:basic-tools-signed-valuation}, we give the following simple generalization of the characterization known for tropical halfspaces in one orthant in \cref{sec:tropicalization-of-halfspaces}.

\begin{lemma}\label{le:sval_hspace}
For every halfspace $\bHclosed^+(\bm a)$ we have $\sval\bigl(\bHclosed^+(\bm a) \bigr) = \Hclosed^+\bigl(\sval(\bm a)\bigr)$. 
\end{lemma}

Using a hyperplane separation theorem over Puiseux series, we derive the following separation theorem for images of closed convex semialgebraic sets in \cref{sec:separation-over-puiseux-series}. 

\begin{theorem} \label{thm:tropical_polar}
Suppose that $\bm X \subseteq \KK^d$ is a nonempty closed convex semialgebraic set.
Then, for every $y \notin \sval(\bm X)$ there exists $\bm a \in \KK^{d+1}$ such that $\bm X \subseteq \bHclosed^{+}(\bm a)$ and $y \notin \Hclosed^{+}(\sval(\bm a))$.
In particular, $\sval(\bm X)$ is equal to the intersection of the closed tropical halfspaces that contain it. 
\end{theorem}

We strengthen the representation as intersection of closed tropical halfspaces for polyhedral cones in \cref{sec:tropicalization+cones}.
The results from that subsection are mainly preparatory for the main results in \cref{sec:TC-hull+as+intersection}.

Finally, we show that TC-hemispaces are images of convex sets under signed valuation.
This is crucial for our main separation result for TC-convexity.

\begin{proposition}\label{prop:hemispace_puiseux}
If $G \subseteq \TTpm^d$ is a TC-hemispace, then $G = \sval(\bm G)$ for some convex set $\bm G \subseteq \KK^d$.
\end{proposition}

\subsection{Basic tools for the signed valuation}
\label{sec:basic-tools-signed-valuation}

The following lemma summarizes basic properties of the signed valuation.

\begin{lemma}\label{le:basic_sval}
The signed valuation map has the following properties: 
\begin{enumerate}[(i)]
\item if $\bm x_1, \bm x_2 \in \KK$ satisfy $\bm x_1 \ge \bm x_2$, then $\sval(\bm x_1) \ge \sval(\bm x_2)$;
\item if $\bm x_1, \dots, \bm x_n \in \KK$, then $\sval(\bm x_1 \cdots \bm x_n) = \sval(\bm x_1) \odot \dots \odot \sval(\bm x_n)$;
\item if $\bm x_1, \dots, \bm x_n \in \KK^d$, then $\sval(\bm x_1 + \dots + \bm x_n) \in \Uncomp\bigl(\sval(\bm x_1) \oplus \dots \oplus \sval(\bm x_n)\bigr)$.
\end{enumerate}
\end{lemma}
\begin{proof}
The first property is trivial if $\bm x_1 \ge 0 \ge \bm x_2$.
Suppose that $\bm x_1 \ge \bm x_2 > 0$ but $\sval(\bm x_1) < \sval(\bm x_2)$.
Then, the leading coefficient of $\bm x_2 - \bm x_1$ is equal to $\lc(\bm x_2) > 0$, so $\bm x_2 - \bm x_1 > 0$, which gives a contradiction.
Analogously, if $0 > \bm x_1 \ge \bm x_2$, but $\sval(\bm x_1) < \sval(\bm x_2)$, then the leading coefficient of $\bm x_2 - \bm x_1$ is equal to $-\lc(\bm x_1) > 0$, which gives a contradiction.

The second property for $n = 2$ follows from the definition of multiplication of Puiseux series.
The extension to $n > 2$ follows by an immediate induction.

To prove the third property, it is enough to consider the case $d = 1$ since $\sval$ and $\Uncomp$ are defined componentwise. 
Furthermore, observe that if $\bm y, \bm z \in \KK$ are nonnegative (or nonpositive), then $\sval(\bm y + \bm z) = \sval(\bm y) \oplus \sval(\bm z)$ by the definition of addition in Puiseux series.
We can order $\bm x_1, \dots, \bm x_n$ in such a way that $\bm x_1, \dots, \bm x_k$ are nonnegative and $\bm x_{k+1}, \dots, \bm x_n$ are negative.
Let $\bm x^{+} = \bm x_1 + \dots + \bm x_k$ and $\bm x^{-} = \bm x_{k+1} + \dots + \bm x_n$.
By the observation above, we have $\sval(\bm x^{+}) = \sval(\bm x_1) \oplus \dots \oplus \sval(\bm x_k)$ and $\sval(\bm x^{-}) = \sval(\bm x_{k+1}) \oplus \dots \oplus \sval(\bm x_n)$.
If $\sval(\bm x^+) \neq \ominus \sval(\bm x^-)$, then the definition of addition in Puiseux series implies that $\sval(\bm x^{+} + \bm x^{-}) = \sval(\bm x^{+}) \oplus \sval(\bm x^{-})$.
If $\sval(\bm x^+) = \ominus \sval(\bm x^-) = a$, then the definition of addition in Puiseux series implies that $\sval(\bm x^{+} + \bm x^{-}) \in [\ominus a, a] = \Uncomp(\sval(\bm x^{+}) \oplus \sval(\bm x^{-}))$.
\end{proof}

The signed valuation map allows us to study sets defined over $\KK^d$ by looking at their images under $\sval$.
Conversely, it is sometimes useful to study a set defined in $\TTpm^d$ by looking at its `lift' in the Puiseux series.
Since $\sval$ is not bijective, we have many possible choices for the lift.
We now introduce different types of lifts of points $x \in \TTpm^d$ into Puiseux series that are used in this work.
Similar lifts were used in \cite{AllamigeonBenchimolGaubertJoswig:2015} to derive properties about lifts of tropical halfspaces within $\Tgeq^d$. 
We start with the simplest one, the canonical lift.

\begin{definition}
Given a point $x \in \TTpm^d$ we define its \emph{canonical lift}, $\clif(x) \in \KK^d$, as the point
\[
\forall i \in [d], \ \bigl( \clif(x) \bigr)_i = \sigma t^{|x_{i}|}\, ,
\]
where $\sigma \in \{-,0,+\}$ is the ``de-tropicalized'' version of $\tsgn(x_i) \in \{\ominus, \Zero, \oplus\}$ and we use the convention that $t^{\Zero} = 0$. 
\end{definition}

The canonical lifts are very simple, but it turns out that they are not particularly well suited for our purposes. Instead, it is more useful for us to consider lifts that vary from one orthant of $\TTpm^d$ to another. To this end, we introduce the following definition.

\begin{definition} \label{def:li-type-J}
Given a point $x \in \TTpm^d$ and a set $\subs \subseteq [d]$ we define the \emph{lift of $x$ of type $\subs$}, denoted $\lif_{\subs}(x) \in \KK^d$, by
\[
\forall i \in [d], \ \bigl( \lif_{\subs}(x) \bigr)_i =
\begin{cases}
(d+1)t^{|x_{i}|} &\text{if $\tsgn(x_i) = \oplus$ and $i \in \subs$}, \\
-(d+1)t^{|x_{i}|} &\text{if $\tsgn(x_i) = \ominus$ and $i \notin \subs$}, \\
\sigma t^{|x_{i}|} &\text{otherwise, }
\end{cases}
\]
where $\sigma$ is the ``de-tropicalized'' version of $\tsgn(x_i)$. 
\end{definition}
In this way, $\lif_{\emptyset}(x)$ coincides with $\clif(x)$ on the nonnegative orthant of $\TTpm^d$, $\lif_{[d]}(x)$ coincides with $\clif(x)$ on the nonpositive orthant of $\TTpm^d$ and so on. Also, for every $x$ and $\subs$ we have $\sval\bigl( \lif_{\subs}(x) \bigr) = x$.

\subsection{Tropicalization of halfspaces}
\label{sec:tropicalization-of-halfspaces}

We are ready to prove the characterization of the signed valuation of halfspaces.

\begin{proof}[Proof of \cref{le:sval_hspace}]
  The proof proceeds as in the case of tropical halfspaces in one orthant, cf. \cite[Proposition~2.4]{DevelinYu:2007}.
  Let $a = \sval(\bm a)$.
  The inclusion $\sval\bigl(\bHclosed^+(\bm a) \bigr) \subseteq \Hclosed^+(a)$ follows from the arithmetic properties of $\sval$. Indeed, if $\bm x \in \KK^d$ is such that $\bm y = \bm a_0 + \bm a_1 \bm x_2 + \dots + \bm a_d \bm x_d \ge 0$ and we denote $x = \sval(\bm x)$, then \cref{le:basic_sval} shows that $\sval(\bm y) \ge \Zero$ and $\sval(\bm y) \in \Uncomp(a_0 \oplus a_1 \odot x_1 \oplus \dots \oplus a_d \odot x_d)$. Hence, we either have $a_0 \oplus a_1 \odot x_1 \oplus \dots \oplus a_d \odot x_d \in \Tzero$ or $a_0 \oplus a_1 \odot x_1 \oplus \dots \oplus a_d \odot x_d \in \Tgt$.
  
  Conversely, if $x \in \HC^+(a)$ and we let $\bm x = \clif(x)$, then \cref{le:basic_sval} shows that $\sval(\bm a_0 + \bm a_1 \bm x_2 + \dots + \bm a_d \bm x_d) \in \Uncomp(a_0 \oplus a_1 \odot x_1 \oplus \dots \oplus a_d \odot x_d)$, which is a singleton in $\Tgt$, so $\bm x \in \bHclosed^+(\bm a)$.
  Hence $\HC^+(a) \subseteq \sval\bigl (\bHclosed^+(\bm a) \bigr) \subseteq \Hclosed^+(a)$.
  Furthermore, the set $\sval\bigl (\bHclosed^+(\bm a) \bigr)$ is closed by \cite[Theorem~6.9]{JellScheidererYu:2018} or, equivalently, \cite[Corollary~4.11]{AllamigeonGaubertSkomra:2020}.
  Therefore, \cref{le:hspace_topo} implies that $\sval\bigl (\bHclosed^+(\bm a) \bigr) = \Hclosed^+(a)$.
\end{proof}

The next lemma gives a more explicit connection between tropical halfspaces of type $\subs$ and lifts of type $\subs$.

\begin{lemma}\label{le:lift_type_hspace}
  Suppose that $\Hclosed^+(a)$ is a tropical halfspace of type $\subs \subseteq [d]$. 
  Let $x \in \Hclosed^+(a)$ and denote $K = [d]_0 \setminus \subs$.
  Then $\lif_{\subs}(x) \in \bHclosed^+\bigl(\lif_{K}(a)\bigr)$. 
\end{lemma}
Before giving the proof, let us note that if $x \in \HC^+(a)$, then any lift of $x$ belongs to $\bHclosed^+(\bm a)$ for any lift $\bm a$ of $a$.
However, in order to lift the points that belong to the boundary of $\Hclosed^+(a)$ we need to be more careful. 

\begin{proof}
  Let $\bm x = \lif_{\subs}(x)$ and $\bm a  = \lif_{K}(a)$.
  By definition, we have
\[
\bm a_{0} = 
\begin{cases}
(d+1)t^{a_{0}} &\text{if $a_{0} \in \Tgt$,} \\
-t^{|a_{0}|} &\text{otherwise.}
\end{cases}
\]
Also, for every $i \in [d]$ we have $\bm a_{i} = \tsgn(a_i)t^{|a_{i}|}$ and
\[
\bm a_{i} \bm x_{i} =
\begin{cases}
(d+1)t^{|a_i| + |x_{i}|} &\text{if $\tsgn(x_i) = \tsgn(a_i)$}, \\
-t^{|a_i| + |x_{i}|} &\text{otherwise.}
\end{cases}
\]
Therefore, we have $\bm a \cdot \begin{pmatrix} 0 \\ \bm x \end{pmatrix} \ge 0$ if and only if
\begin{equation}\label{eq:lifts_in_hyp}
(d+1) \Bigl( \alpha t^{|a_0|} + \sum_{\tsgn(x_i) = \tsgn(a_i)}t^{|a_i| + |x_{i}|} \Bigr)\ge \beta t^{|a_0|} + \sum_{\tsgn(x_i) \neq \tsgn(a_i)} t^{|a_i| + |x_{i}|} \, ,
\end{equation}
where $(\alpha, \beta) = (1, 0)$ if $a_{0} \in \Tgt$ and $(\alpha, \beta) = (0, 1)$ otherwise.
Let $\bm y \in \KK$ denote the series on the left-hand side of~\cref{eq:hspace} and $\bm z \in \KK$ denote the series on the right-hand side of~\cref{eq:hspace}. By \cref{le:basic_sval} we have $\sval(\bm y - \bm z) \in \Uncomp\bigl(\sval(\bm y) \ominus \sval(\bm z)\bigr) = \Uncomp(a_0 \oplus a_1 \odot x_1 \oplus \dots \oplus a_d \odot x_d)$. If $\Uncomp(a_0 \oplus a_1 \odot x_1 \oplus \dots \oplus a_d \odot x_d)$ is a nonnegative singleton, then $\bm y - \bm z \ge 0$. Otherwise, we have $\sval(\bm y) = \sval(\bm z) > \Zero$. In this case, note that $\lc(\bm y) \ge d +1$ and $\lc(\bm z) \le d$. Hence $\bm y - \bm z \ge 0$.
\end{proof}

\subsection{Separation over Puiseux series}
\label{sec:separation-over-puiseux-series}

We also need the following version of the hyperplane separation theorem over Puiseux series.
We recall that a set $\bm X \subseteq \KK^d$ is a \emph{cone} if $\bm \lambda \bm x \in \bm X$ for all $\bm x \in \bm X$ and $\bm \lambda > 0$. We also recall that a set $\bm X$ is \emph{semialgebraic} if it is defined by a finite Boolean combination of polynomial inequalities. Since the field of Puiseux series is real closed, the semialgebraic sets over $\KK^d$ have similar properties to the semialgebraic sets over $\RR^d$---we refer to \cite{BasuPollackRoy:2006} for more information on this topic. For the sake of generality, we state the next two propositions for semialgebraic sets, but the familiarity with semialgebraic sets is not necessary to understand the other results of this paper---it is enough to admit that polyhedra are semialgebraic. The first proposition is a hyperplane separation theorem for convex semialgebraic sets.

\begin{proposition}\label{th:hyperplane_sep}
  Suppose that $\bm X, \bm Y \subseteq \KK^d$ are two nonempty convex semialgebraic sets such that $\bm X \cap \bm Y = \emptyset$.
  Then, there exists a halfspace $\bHclosed^+(\bm a)$ such that $\bm X \subseteq \bHclosed^+(\bm a)$ and $\bm Y \subseteq \bHclosed^-(\bm a)$.
\end{proposition}
\begin{proof}
  If we replace $\KK$ by $\RR$, then the claim follows from the hyperplane separation theorem in $\RR^{d}$, see, e.g.,~\cite[Theorem~11.3 and Theorem~11.7]{Rockafellar:1972}.
  Since the sets $\bm X, \bm Y$ are supposed to be semialgebraic, the claim for $\KK$ follows from the completeness of the theory of real closed fields, see, e.g.,~\cite[Corollary~3.3.16]{Marker:2002} or \cite[Theorem~2.80]{BasuPollackRoy:2006}.
  The argument is based on the existence of finite formulas for describing $\bm X$ and $\bm Y$ and the existence of a separating hyperplane in a real closed field due to their convexity; for more details see \cref{sec:details+semialgebraic+separation}.  
\end{proof}

\begin{remark}
We note that the assumption that $\bm X, \bm Y$ are semialgebraic cannot be entirely skipped. Indeed, it is shown in~\cite{Robson:1991} that $\RR$ is the only ordered field that admits the general hyperplane separation theorem. For the interested reader, we adapt the example from \cite{Robson:1991} to $\KK$ in \cref{sec:details+semialgebraic+separation}. We also note that a special case of \cref{th:hyperplane_sep} is still valid for sets that are definable in definably complete extensions of real closed fields, see~\cite[Corollary~2.20]{AschenbrennerFischer:2011}. Also, \cite{Robson:1991} gives a version of the separation theorem that is valid for arbitrary ordered fields.
\end{remark}

As an application of \cref{th:hyperplane_sep}, we get the following. 

\begin{proof}[Proof of \cref{thm:tropical_polar}]
  To prove the first part of the claim, let $y \notin \sval(\bm X)$.
  Since $\bm X$ is closed and semialgebraic, $\sval(\bm X)$ is also closed by \cite[Theorem~6.9]{JellScheidererYu:2018} or \cite[Corollary~4.11]{AllamigeonGaubertSkomra:2020}.
  Therefore, there exists an open neighborhood of $y$ that does not intersect $\sval(\bm X)$.
  Since the order on $\TTpm$ is dense, we can find $\ell_1,\dots,\ell_d,r_1,\dots,r_d \in \TTpm$ such that 
\[
\ell_1 < y_1 < r_1, \ \ell_2 < y_2 < r_2, \ \dots \ , \ell_d < y_d < r_d
\]
and such that the box $B = [\ell_1, r_1] \times [\ell_2, r_2] \times \dots \times [\ell_d, r_d]$ does not intersect $\sval(\bm X)$.
Consider the lifted box
\[
\bm B = [\clif(\ell_1), \clif(r_1)] \times [\clif(\ell_2), \clif(r_2)] \times \dots \times [\clif(\ell_d), \clif(r_d)] \subset \KK^d \, .
\]
The set $\bm B$ is convex and semialgebraic.
Moreover, we have $\bm X \cap \bm B = \emptyset$ because $\sval(\bm X) \cap \sval(\bm B) = \sval(\bm X) \cap B = \emptyset$.
Hence, by \cref{th:hyperplane_sep}, there exists $\bm a \in \KK^{d+1}$ such that $\bm X \subseteq \bHclosed^{+}(\bm a)$ and $\bm B \subseteq \bHclosed^-(\bm a)$.
Therefore $B \subseteq \Hclosed^-(\sval(\bm a))$ by \cref{le:sval_hspace}.
Since $y$ belongs to the interior of $B$, \cref{le:hspace_topo} shows that $y \in \HC^-(\sval(\bm a))$.
In particular, we have $y \notin \Hclosed^+(\sval(\bm a))$. To prove the second part of the claim, note that $\sval(\bm X) \subseteq \Hclosed^+(\sval(\bm a))$ by \cref{le:sval_hspace}. In other words, the closed tropical halfspace $\Hclosed^+(\sval(\bm a))$ contains $\sval(\bm X)$ but not $y$. Since $y$ was arbitrary, we get that $\sval(\bm X)$ is an intersection of some family of closed tropical halfspaces. Therefore, it is also an intersection of all the closed tropical halfspaces that contain it.
\end{proof}

\begin{corollary} \label{cor:linear+halfspaces+outer+cone}
  If $\bm X \subseteq \KK^d$ from \cref{thm:tropical_polar} is a cone, we can choose the tropical halfspaces to be linear.  
\end{corollary}
\begin{proof}
  To see that, suppose that $\bm a_0 \neq 0$.
  Since $0 \in \bm X$, we have $\bm a_0 > 0$. 
  Consider $\hat{\bm a} = (0,\bm a_1, \dots, \bm a_d)$.
  Since $\bHclosed^{+}(\hat{\bm a}) \subseteq \bHclosed^{+}(\bm a)$, we get $y \notin \Hclosed^{+}(\sval(\hat{\bm a}))$ by \cref{le:sval_hspace}.
  
  It remains to show that $\bm X \subseteq \bHclosed^{+}(\hat{\bm a})$.
  Indeed, if there exists $\bm x \in \bm X$ such that $\bm x^{T} \hat{\bm a} < 0$, then for any $\bm \lambda > -\bm a_0 / \bm x^{T} \hat{\bm a} > 0$ we have $\bm \lambda \bm x \in \bm P$ but $\bm \lambda \bm x \notin \bHclosed^{+}(\bm a)$, which is a contradiction.
  
  Therefore, we can suppose that $\bm a$ satisfies $\bm a_0 = 0$.
\end{proof}

\begin{example}
We note that neither of the assumptions of \cref{thm:tropical_polar} can be skipped.
Indeed, the set $\bm X = \SetOf{\bm x \in \KK}{\bm x > 0}$ is semialgebraic and convex but not closed.
We have $\sval(\bm X) = \Tgt$, which is not an intersection of closed tropical halfspaces.
Likewise, the set $\bm Y = \SetOf{\bm x \in \KK}{\sval(\bm x) > 0}$ is closed and convex, but not semialgebraic.
We have $\sval(\bm Y) = \SetOf{x \in \TTpm}{x > 0}$, which is not an intersection of closed tropical halfspaces.
\end{example}

The following lemma gives a partial characterization of the intersection of all closed tropical halfspaces that contain a given finite set.

\begin{lemma}\label{le:weak_hull_from_Puiseux}
Given a finite set $X = \{x_1, \dots, x_m\} \subset \TTpm^d$ we have
\begin{align*}
\bigcap \SetOf{\Hclosed^+(a)}{X \subseteq \Hclosed^+(a)} &= \bigcap \SetOf{\sval(\bm X)}{\bm X \subseteq \KK^d \text{ convex } \wedge X \subseteq \sval(\bm X)} \\
&= \bigcap_{\subs \subseteq [d]} \sval\Bigl( \conv\bigl( \lif_{\subs}(x_1), \dots, \lif_{\subs}(x_m) \bigr) \Bigr) \, .
\end{align*}
\end{lemma}
\begin{proof}
Denote 
\begin{align*}
  U &= \bigcap \SetOf{\Hclosed^+(a)}{X \subseteq \Hclosed^+(a)} \, , \\
  V &= \bigcap \SetOf{\sval(\bm X)}{\bm X \subseteq \KK^d \text{ convex } \wedge X \subseteq \sval(\bm X)} \, , \\
  W &= \bigcap_{\subs \subseteq [d]} \sval\Bigl( \conv\bigl( \lif_{\subs}(x_1), \dots, \lif_{\subs}(x_m) \bigr) \Bigr) \, .
\end{align*}
We start by showing that  
\begin{equation}\label{eq:simpler_inter_convex}
  V = \bigcap\SetOf{ \sval\Bigl( \conv\bigl( \bm x_1, \dots, \bm x_m \bigr) \Bigr)}{\forall i, \, \sval(\bm x_i) = x_i} \, .
\end{equation}
Indeed, the inclusion $\subseteq$ in \cref{eq:simpler_inter_convex} holds as we just range over a smaller set. 
To prove the opposite inclusion, it is enough to observe that any convex set $\bm X$ such that $\{x_1, \dots, x_m\} \subseteq \sval(\bm X)$ contains some set of the form $\conv\bigl( \bm x_1, \dots, \bm x_m \bigr)$.
Hence, the equality \cref{eq:simpler_inter_convex} holds.

As $\conv(\bm x_1, \dots, \bm x_m)$ is a closed convex semialgebraic set for each choice with $\sval(\bm x_i) = x_i$ for all $i \in [m]$, we can apply \cref{thm:tropical_polar} to each set in the intersection \cref{eq:simpler_inter_convex}.
This implies that $V$ is an intersection of a family of closed tropical halfspaces.
Since $X \subseteq V$, all of these halfspaces contain $X$.
Hence, $U \subseteq V$.

The inclusion $V \subseteq W$ is trivial.
Therefore, it remains to prove that $W \subseteq U$.
To this end, take a point $z \notin U$.
By definition, there is a closed tropical halfspace $\Hclosed^+(a)$ that contains $X$ with $z \notin \Hclosed^+(a)$.
Suppose that $\Hclosed^+(a)$ is of type $\subs \subseteq [d]$ and let $K = [d]_0 \setminus \subs$ be its complement. 
Then, \cref{le:lift_type_hspace} shows that $\lif_\subs(x_i) \in \bHclosed^+\bigl(\lif_{K}(a)\bigr)$ for every $i \in [m]$.
Since the set $\bHclosed^+\bigl(\lif_{K}(a)\bigr)$ is convex, we obtain
\[
\conv\bigl( \lif_{\subs}(x_1), \dots, \lif_{\subs}(x_m) \bigr) \subset \bHclosed^+\bigl(\lif_{K}(a)\bigr) \, .
\]
Using the representation of the valuation of halfspace from \cref{le:sval_hspace}, we get
\begin{equation*}
  W \subseteq \sval\left( \conv\bigl( \lif_{\subs}(x_1), \dots, \lif_{\subs}(x_m) \bigr) \right) \subset \sval\left(\bHclosed^+\bigl(\lif_{K}(a)\bigr)\right) = \Hclosed^+(a) .
\end{equation*}
In particular, we get $z \notin W$ and so $W \subseteq U$ as claimed.
\end{proof}

\subsection{Tropicalization and cones}
\label{sec:tropicalization+cones}

We provide refined representations of the tropicalization of subsets of $\KK^d$ in case they are cones.
These insights will be used to derive the main representation theorems in \cref{sec:TC-hull+as+intersection}. 

The next lemma strengthens the claim of \cref{cor:linear+halfspaces+outer+cone} for polyhedral cones.

\begin{lemma}\label{le:sval_cone}
Suppose that $\bm P \subseteq \KK^d$ is a polyhedral cone. Then, $\sval(\bm P)$ is an intersection of finitely many linear closed tropical halfspaces.
\end{lemma}
\begin{proof}
  Let $\bm P = \SetOf{\bm x \in \KK^d}{\bm A \bm x \ge 0}$ for some matrix $\bm A \in \KK^{n \times d}$.
  By Farkas' lemma~\cite[Corollary~7.1h]{Schrijver:1987}, for every $\bm a \in \KK^d$ we have the equivalence
  \begin{equation}\label{eq:farkas_puiseux}
  \bm P \subseteq \bHclosed^{+}(0,\bm a) \iff \bm a \in \SetOf{\bm y^{T}\bm A}{\bm y \ge 0} \, .
  \end{equation}
  The intersection of $\SetOf{\bm y^{T}\bm A}{\bm y \ge 0}$ with any closed orthant $\KK^d$ is a polyhedral cone, and so it is generated by a finite family of rays by the Minkowski--Weyl theorem~\cite[Corollary~7.1a]{Schrijver:1987}.
  For any closed orthant $\bort \subset \KK^d$, let $\bm U_{\bort} \subset \KK^d$ be a finite set such that $\SetOf{\bm y^{T}\bm A}{\bm y \ge 0} \cap \bort = \cone(\bm U_{\bort})$.
  Let $\bm U = \bigcup_{\bort} \bm U_{\bort}$ be the set of all rays obtained in this way.
  We claim that $\sval(\bm P) = \bigcap_{\bm u \in \bm U} \Hclosed^+(\Zero, \sval(\bm u))$.
  The inclusion ``$\subseteq$'' follows by combining \cref{eq:farkas_puiseux} with \cref{le:sval_hspace}.

  To prove the opposite inclusion, let $y \notin \sval(\bm P)$.
  By \cref{thm:tropical_polar}, there exists $\bm a \in \KK^{d+1}$ such that $\bm P \subseteq \bHclosed^{+}(\bm a)$ and $y \notin \Hclosed^{+}(\sval(\bm a))$.
  Since $\bm P$ is a cone, we can choose $\bm a$ in such a way that $\bm a_0 = 0$ by \cref{cor:linear+halfspaces+outer+cone}; hence, we let $\bm a \in \KK^{d}$. 
  
  Let $\bort \subset \KK^d$ be a closed orthant such that $\bm a = (\bm a_1, \dots, \bm a_d) \in \bort$. 
  By \cref{eq:farkas_puiseux} we have ${\bm a} \in \cone(\bm U_{\bort})$.
  Denote $\bm U_{\bort} = \{\bm u_1, \dots, \bm u_m\}$. 
  Since the orthant $\bort$ is fixed, we have $\sval({\bm a}) \in \wcone\left(\sval({\bm U_{\bort}})\right)$ 
  by \cite[Lemma~8]{AllamigeonGaubertSkomra:2019}. 
  Suppose that $y \in \Hclosed^{+}(\Zero,\sval(\bm u_i))$ for all $i \in [m]$.
  Then, we also have $\sval(\bm u_i) \in \Hclosed^{+}(\Zero,y)$ for all $i \in [m]$.
  As closed halfspaces are convex by \cref{cor:hspace_weakly_convex}, we get $\whull\left(\sval({\bm U_{\bort}})\right)=\whull\bigl(\sval(\bm u_1), \dots, \sval(\bm u_m)\bigr) \subseteq \Hclosed^{+}(\Zero,y)$.
  Furthermore, since the tropical halfspace $\Hclosed^{+}(\Zero,y)$ is linear, we get $\whull\left(\sval({\bm U_{\bort}})\right) \subseteq \Hclosed^{+}(\Zero,y)$. 
  Therefore $\sval({\bm a}) \in \Hclosed^{+}(\Zero,y)$, which implies that $y \in \Hclosed^{+}(\Zero,\sval(\bm a))$, giving a contradiction.
  
  Hence, there exists $i \in [m]$ such that $y \notin \Hclosed^{+}(\Zero,\sval(\bm u_i))$.
  Since $y$ was arbitrary, we get $\bigcap_{\bm u \in \bm U} \Hclosed^+(\Zero, \sval(\bm u)) \subseteq \sval(\bm P)$.
\end{proof}

Finally, we demonstrate how the tropicalization of a cone is composed of convex hulls. 
This will be used in the proof of the Minkowski-Weyl theorem for TC-convex cones. 

\begin{lemma}\label{le:multiplied_conv}
Let $\bm X \subseteq \KK^d$ be an arbitrary set. Take $\bm \lambda \in \KK$ such that $\bm \lambda \ge 0$ and denote $\sval(\bm \lambda) = \lambda$. Then, we have 
\begin{equation}\label{eq:multiplied_conv_puisuex}
\sval\bigl(\conv(\bm \lambda \bm X)\bigr) = \lambda \odot \sval\bigl(\conv( \bm X)\bigr) \, .
\end{equation}
In particular, we have the equality
\[
\sval\bigl(\cone(\bm X)\bigr) = \bigcup \SetOf{\sval\bigl(\conv(t^{\lambda}\bm X)\bigr)}{\lambda \in \Tgeq} \, .
\]
\end{lemma}
\begin{proof}
Let $x \in \sval\bigl(\conv(\bm \lambda \bm X)\bigr)$. Since $\conv(\bm \lambda \bm X) = \bm \lambda \conv(\bm X)$, there exists $\bm y \in \conv(\bm X)$ such that $x = \sval(\bm \lambda \bm y) = \sval(\bm \lambda) \odot \sval(\bm y) = \lambda \odot \sval(\bm y)$. Hence $x \in \lambda \odot \sval\bigl(\conv( \bm X)\bigr)$. Conversely, suppose that $x \in \lambda \odot \sval\bigl(\conv( \bm X)\bigr)$. Then, there exists $\bm y \in \conv(\bm X)$ such that $x = \lambda \odot \sval(\bm y) = \sval(\bm \lambda \bm y)$, so that $x \in \sval\bigl(\conv(\bm \lambda \bm X)\bigr)$. To prove the second claim, note that~\cref{eq:multiplied_conv_puisuex} gives $\sval\bigl(\conv(\bm \lambda \bm X)\bigr) = \sval\bigl(\conv(t^{\lambda}\bm X)\bigr)$. In particular, we get
\begin{align*}
\sval\bigl(\cone(\bm X)\bigr) &= \bigcup \SetOf{\sval\bigl(\conv(\bm \lambda \bm X)\bigr)}{\bm \lambda \ge 0} \\
&= \bigcup \SetOf{\sval\bigl(\conv(t^{\lambda}\bm X)\bigr)}{\lambda \in \Tgeq} \, . \qedhere
\end{align*}
\end{proof}

\subsection{Lifts of TC-hemispaces}

Now, we prove our crucial auxiliary proposition relating TC-hemispaces with convex lifts.

\begin{proof}[Proof of \cref{prop:hemispace_puiseux}]
  First, note that the claim is trivial if $G \in \{\emptyset, \TTpm^d\}$.
  Otherwise, \cref{th:hemispace} shows that there exists $a = (a_0, \dots, a_d) \in \TTpm^{d+1}$, $(a_1, \dots, a_d) \neq \Zero$ such that $\HC^{+}(a) \subseteq G \subseteq \Hclosed^{+}(a)$.

  We prove the claim by induction over $d$.
  If $d = 1$, then the claim follows from the fact that, for any $a \in \TTpm$, each set $\SetOf{\bm x \in \KK}{\sval(\bm x) {\sim} a}$ with ${\sim} \in \{\ge, >, \le, <\}$ is convex. 

  If $d > 1$, then we let $X = \SetOf{x \in \TTpm^d}{\domin(a,x) \neq \emptyset}$ and its complement $\Xperp = \TTpm^{d} \setminus X$.
  Note that $\Xperp$ is essentially the set of points whose support is disjoint from the support of $a$. 
  We will construct a lift for the intersection of $G$ with $X$ and $X'$ and finish by taking their convex hull. 
  
\smallskip

  We start by showing that there exists a convex set $\Gper \subseteq \KK^d$ with $\sval(\Gper) = G \cap \Xperp$.
  If $\Xperp = \emptyset$, then we set $\Gper = \emptyset$.
  Otherwise, let $K = \SetOf{k \in [d]}{a_k \neq \Zero}$ and note that we have $a_0 = \Zero$ and $\Xperp = \SetOf{x \in \TTpm^{d}}{x_k = \Zero \; \forall k \in K}$.
  
  If $K = [d]$, then we take $\Gper = \{0\}$ if $\Zero \in G$ and $\Gper = \emptyset$ otherwise.
  Note that $G \cap \Xperp$ and $\Xperp \setminus G$ are TC-convex as intersection of TC-convex sets. 
  If $K \subsetneq [d]$, then let $\pi \colon \TTpm^d \to \TTpm^{d - |K|}$ be the projection that forgets the coordinates from $K$.
  As the components in $K$ of $G \cap \Xperp, \Xperp \setminus G \subseteq \Xperp$ are all $\Zero$, \cref{prop:local-generation-TC-hull} implies that the set $\pi(G \cap \Xperp)$ is a TC-hemispace in $\TTpm^{d-|K|}$.
  Hence, by the induction hypothesis, we have $\pi(G \cap \Xperp) = \sval(\bm G')$ for some convex set $\bm G' \subseteq \KK^{d- |K|}$.
  We embed the set $\bm G'$ into $\KK^d$ by adding $0$ coordinates to every point and this gives a convex set $\Gper \subseteq \KK^{d}$ such that $\sval(\Gper) = G \cap \Xperp$.

  \smallskip

  Next, we show that there exists a convex set $\Gali \subseteq \KK^{d}$ with $\sval(\Gali) = G \cap X$.
  To construct an appropriate lift with \cref{def:li-type-J}, let $\subs = \SetOf{k \in [d]}{a_k > \Zero}$ be such that $\Hclosed^{+}(a)$ is of type $\subs$. 
  Set $\Gali = \conv(\SetOf{\lif_{\subs}(x) \in \KK^d}{x \in G \cap X})$.
  Then, $\Gali$ is convex and right from the definition we have $G \cap X \subseteq \sval(\Gali)$.  

  To prove that $\sval(\Gali) \subseteq G \cap X$, let $\bm y \in \Gali$ be arbitrary and $y = \sval(\bm y)$. 
  Then, the Carath{\'e}odory theorem in $\KK^d$ shows that
  \begin{equation} \label{eq:convex-combination-lift-y}
    \bm y = \bm \lambda_1 \lif_{\subs}(y^{(1)}) + \dots + \bm \lambda_m \lif_{\subs}(y^{(m)})
  \end{equation}
for some $m \le d+1$, $y^{(1)}, \dots, y^{(m)} \in G \cap X$, $\bm \lambda_1, \dots, \bm \lambda_m > 0$, and $\sum_i \bm \lambda_i = 1$.
As we aim to prove that $y \in G \cap X$, we will estimate the leading terms of the components in the sum~\cref{eq:convex-combination-lift-y}. 
  For every $i \in [m]$ let $\lambda_i = \sval(\bm \lambda_i) > \Zero$ and $\lclam_i = \lc(\bm \lambda_i) > 0$. 
  We assume that the points $y^{(1)}, \dots, y^{(m)}$ are labeled such that $\lclam_1 \ge \dots \ge \lclam_m$.
  To relate $y$ with the hemispace property of $G$, we set
  \begin{equation} \label{eq:leftsum-sval-y-points}
  z = \lambda_1 \odot y^{(1)} \lplus \dots \lplus \lambda_m \odot y^{(m)} \,.
\end{equation} 
  With $\bigoplus_i \lambda_i = 0$, the TC-convexity of $G$ implies $z \in G$. 
  Moreover, we also have $z \in X$.
  Indeed, this is trivial if $a_0 \neq \Zero$.
  Otherwise, because of $y^{(1)} \in X$, it has a component $k \in [d]$ such that $a_k \odot y^{(1)}_k \neq \Zero$ and so $a_k \odot z_k \neq \Zero$.
  Thus, $z \in G \cap X$. 
  
  We will prove that $y \in G$ using \cref{le:pos_dir_hemispace}. To do so, suppose that there exists $k \in \posdomin(a,z) \cap [d]$.  If $a_k > \Zero$ and $z_k > \Zero$, then we define a vector $u \in \RR^m$ by
  \begin{equation*}
    u_i =
    \begin{cases}
     \lc\Bigl(\bigl(\lif_{\subs}(y^{(i)})\bigr)_k\Bigr) & \text{ if } \left|\lambda_i \odot y^{(i)}_k\right| = |z_k|, \\
     0 & \text{ otherwise. }
   \end{cases}
\end{equation*}
Since $k \in J$, we have $u \in \{-1,0,d+1\}^m$.
Moreover, since $z_k> \Zero$, the vector $u$ is nonzero.
Let $p \in [m]$ be the smallest index such that $u_p \neq 0$, which is exactly the entry defining the $k$th component in the left sum~\cref{eq:leftsum-sval-y-points}. 
Hence, we have $z_k = \lambda_p \odot y^{(p)}_k$.
Now, $z_k > \Zero$ implies $y^{(p)}_k > \Zero$ and therefore $u_p = d+1$ by the definition of the lift of type $J$.

We are ready to estimate the sum of the leading coefficients of the dominating terms in the sum $\bm \lambda_1 \bigl(\lif_{\subs}(x_1)\bigr)_k + \dots + \bm \lambda_m \bigl(\lif_{\subs}(x_m)\bigr)_k$. 
This is just 
\begin{equation}\label{eq:lc_in_lifts}
\begin{aligned}
\lclam_1 u_1 + \dots + \lclam_m u_m &\ge (d+1)\lclam_p - \lclam_{p+1} - \dots - \lclam_{m} \\
&\ge (d+1)\lclam_p - (m - p)\lclam_{p} \ge p\lclam_p > 0 \, .
\end{aligned}
\end{equation}
Therefore, there is no cancellation of the leading terms in~\cref{eq:convex-combination-lift-y} which implies $y_k = \sval(\bm y_k) = z_k$. 
Analogously, we obtain $y_k = z_k$ if $a_k < \Zero$ and $z_k < \Zero$.
In particular, $y_k = z_k$ for every $k \in \posdomin(a,z) \cap [d]$.
Furthermore, for every $\ell \in [d]$ we have $\ominus |z_{\ell}| \le y_{\ell} \le |z_{\ell}|$ by \cref{le:basic_sval}. 
Therefore, $y \in G$ by \cref{le:pos_dir_hemispace}.

Moreover, we have $y \in X$.
Indeed, this is trivial if $a_0 \neq \Zero$.
Otherwise, $ \posdomin(a,z) \cap [d] \neq \emptyset$ and so $a_k \odot y_k \neq \Zero$ for at least one $k \in [d]$.
Thus, we have $\sval(\Gali) = G \cap X$.

\smallskip

To finish the proof, let $\bm G = \conv(\Gper \cup \Gali) \subseteq \KK^d$.
Then, $\bm G$ is convex and $G \subseteq \sval(\bm G)$.
To prove that $\sval(\bm G) \subseteq G$, note that this inclusion is trivial if $a_0 \neq \Zero$.
Otherwise, let $\bm z \in \bm G$.
Then, there exist $\bm x \in \Gper, \bm y \in \Gali, \bm \lambda, \bm \mu \ge 0, \bm \lambda + \bm \mu = 1$ such that $\bm z = \bm \lambda \bm x + \bm \mu \bm y$.
Denote $z = \sval(\bm z)$, $x = \sval(\bm x)$, $y = \sval(\bm y)$, $\lambda = \sval(\bm \lambda)$, $\mu = \sval(\bm \mu)$, and $w = \lambda \odot x \lplus \mu \odot y \in G$.
If $\lambda = \Zero$ or $\mu = \Zero$, then we trivially have $z \in G$.
Otherwise, $y \in X$ and $\mu \neq \Zero$ imply that $w \in X$.
Furthermore, $x \in \Xperp$ implies that $z_k = \mu \odot y_k = w_k$ for all $k \in K$.
Therefore, $z \in G$ by \cref{le:good_sup_slice}.
\end{proof}

\section{Separation and representation by halfspaces}
\label{sec:TC-hull+as+intersection}

We are now ready to prove our main results on separation in TO- and TC-convexity using closed tropical halfspaces.
To this end, we combine the results on separation by hemispaces established in \cref{sec:separation+hemispaces}, together with our analysis of the structure of hemispaces.
To obtain the separation result for TC-convexity we additionally use the representations of halfspaces and hemispaces using their lifts from \cref{sec:signed+halfspaces+lifts}.
As an application of the separation results, we derive analogs of the Minkowski--Weyl theorem for polyhedra and cones in \cref{sec:minkowski-weyl-theorem}. 

We start with a strengthening of the separation theorems in \cite[Section 5]{LohoVegh:2020} using the Pash property of TO-convexity shown in \cref{thm:kakutani}. 

\begin{theorem}\label{th:weak_separation_TO}
Suppose that $X,Y \subseteq \TTpm^d$ are nonempty and disjoint TO-convex sets. Then, there exists a hyperplane $\HC(a)$ such that $X \subseteq \Hclosed^+(a)$ and $Y \subseteq \Hclosed^-(a)$.
\end{theorem}
\begin{proof}
By \cref{thm:kakutani}, there exists a TO-hemispace $G \subseteq \TTpm^d$ such that $X \subseteq G$ and $Y \subseteq (\TTpm^{d} \setminus G)$. Since the sets $X,Y$ are nonempty, this hemispace is nontrivial. Hence, by \cref{th:hemispace}, there exists a hyperplane $\HC(a)$ such that $X \subseteq G \subseteq \Hclosed^+(a)$ and $Y \subseteq (\TTpm^{d} \setminus G) \subseteq \Hclosed^-(a)$.
\end{proof}

\begin{corollary}\label{th:TO-closure_separation}
Suppose that a set $X \subseteq \TTpm^d$ is TO-convex.
Then, the set $\cl(X)$ is equal to the intersection of closed tropical halfspaces that contain it.

Furthermore, if $X$ is a nonempty TO-convex cone, then $\cl(X)$ is equal to the intersection of linear closed tropical halfspaces that contain it.
\end{corollary}
\begin{proof}
Let $X$ be a TO-convex set. The claim is trivial if $X$ is empty. Otherwise, note that $\cl(X)$ is obviously included in the intersection of closed tropical halfspaces that contain it.
To prove the opposite inclusion, let $y \notin X$. We can find $\ell_1,\dots,\ell_d,r_1,\dots,r_d \in \TTpm$ such that 
\[
\ell_1 < y_1 < r_1, \ \ell_2 < y_2 < r_2, \ \dots \ , \ell_d < y_d < r_d
\]
and such that the box $B = [\ell_1, r_1] \times [\ell_2, r_2] \times \dots \times [\ell_d, r_d]$ does not intersect $\cl(X)$.
Since $B$ and $X$ are TO-convex and disjoint, \cref{th:weak_separation_TO} implies that there exists a closed tropical halfspace $\Hclosed^{+}(a)$ such that $\cl(X) \subseteq \Hclosed^{+}(a)$ and $B \subseteq \Hclosed^{-}(a)$.
Furthermore, $y$ belongs to the interior of $B$ and so we have $y \in \HC^{-}(a)$ by \cref{le:hspace_topo}.
Since $y$ was arbitrary, we get the first claim.

To prove the second claim, suppose that $X$ is a nonempty TO-convex cone and let $y \notin \cl(X)$.
By the first part of the theorem, there exists $a \in \TTpm^{d+1}$ such that $\cl(X) \subseteq \Hclosed^{+}(a)$ and $y \in \HC^{-}(a)$.
Since $X$ is a nonempty cone, we have $\Zero \in \cl(X)$.
In particular, $a_0 \ge \Zero$.
Let $\tilde{a} = (\Zero, a_1, \dots, a_d)$.
Then, $y \in \HC^{-}(\tilde{a})$.
Suppose that $x \in \cl(X)$ is such that $x \notin \Hclosed^{+}(\tilde{a})$.
Then, $a_1 \odot x_1 \oplus \dots \oplus a_d \odot x_d < \Zero$ and so $\lambda \odot x \in \HC^{-}(a)$ for a sufficiently large $\lambda > \Zero$.
This gives a contradiction with $\cl(X) \subseteq \Hclosed^{+}(a)$.
Hence, we have $\cl(X) \subseteq \Hclosed^{+}(\tilde{a})$, which finishes the proof.
\end{proof}

Before proceeding with the separation theorem for the TC-convexity, let us note that the analog of \cref{th:weak_separation_TO} fails for the TC-convexity as shown by \cref{ex:TC-non-Pasch-plane}. However, the analog of \cref{th:TO-closure_separation} could still be true for the TC-convexity. Our techniques do not allow us to prove this result in full, and we leave it as an open question for future research. Nevertheless, we are able to show that this separation results is true for finitely generated sets. 

\begin{theorem} \label{thm:TC-hull-intersection-halfspaces}
For every finite set $X \subseteq \TTpm^d$, its TC-convex hull $\whull(X)$ is equal to the intersection of the closed tropical halfspaces that contain $X$. 
\end{theorem}
\begin{proof}
  \cref{cor:hspace_weakly_convex} shows that $\whull(X)$ is contained in said intersection. 

  For the reverse direction, 
  \cref{le:weak_hull_from_Puiseux} shows the equality
  \begin{align*}
    &\bigcap_{a \in \TTpm^{d+1}} \SetOf{\Hclosed^+(a)}{X \subseteq \Hclosed^+(a)} \\
    &= \bigcap \SetOf{\sval(\bm Y)}{\bm Y \subseteq \KK^d \text{ convex, } X \subseteq \sval(\bm Y)} \\
  &\subseteq \bigcap \SetOf{\sval(\bm Y)}{\bm Y \subseteq \KK^d \text{ convex, }  X \subseteq \sval(\bm Y),\, \sval(\bm Y) \text{ is a TC-hemispace }} \; .
\end{align*}
By \cref{prop:hemispace_puiseux}, every TC-hemispace containing $X$ arises in the latter intersection.
Combined with \cref{thm:separation_hemispaces}, this yields
\begin{align*}
  &\bigcap \SetOf{\sval(\bm Y)}{\bm Y \subseteq \KK^d \text{ convex, } X \text{ in TC-hemispace } \sval(\bm Y)} \\
    &= \bigcap \SetOf{G}{G \text{ TC-hemispace with } X \subseteq G} \\
& = \whull(X) \, . \qedhere
\end{align*}
\end{proof}

Combining \cref{thm:TC-hull-intersection-halfspaces} and \cref{le:weak_hull_from_Puiseux} we obtain the following corollary, which is an analog of \cite[Theorem~3.14]{LohoVegh:2020} for TC-convexity. 

\begin{corollary} \label{cor:intersection-of-lifts}
For every $x_1, \dots, x_n \in \TTpm^d$ we have 
\begin{align*}
\whull(x_1, \dots, x_n) &=  \bigcap \SetOf{\sval \bigl( \conv( \bm x_1, \dots, \bm x_n)\bigr)}{ \forall i, \, \bm x_i \in \sval^{-1}(x_i) } \\
&= \bigcap_{\subs \subseteq [d]} \sval\Bigl( \conv\bigl( \lif_{\subs}(x_1), \dots, \lif_{\subs}(x_n) \bigr) \Bigr) \, .
\end{align*}
\end{corollary}

\subsection{Minkowski--Weyl theorem}
\label{sec:minkowski-weyl-theorem}

Based on the separation results, we are able to derive various representations of finitely generated TC-convex sets. 
We start with a conic version. 

\begin{theorem}\label{th:conic_M-W}
Let $X \subseteq \TTpm^d$ be an arbitrary set. Then, the following are equivalent: 
\begin{enumerate}[(i)]
\item $X$ is an intersection of finitely many closed linear tropical halfspaces. 
\item $X = \wcone(x_1, \dots, x_m)$ for a finite collection of points $\{x_1, \dots, x_m\}$.
\item There exists a finite collection of polyhedral cones $\bm P_1, \dots, \bm P_m \subseteq \KK^d$ such that
\[
\bigcap_i \sval(\bm P_i) = X \, .
\]
\item There exists a finite collection of linear halfspaces $\bm H_1, \dots, \bm H_m \subseteq \KK^d$ such that
\[
\bigcap_i \sval(\bm H_i) = X \, .
\]
\end{enumerate}
\end{theorem}
\begin{proof}
To prove the implication $(i) \Rightarrow (iv)$, suppose that $X = \bigcap_{i = 1}^{m} \Hclosed^{+}(a_i)$ for some $a_1, \dots, a_m \in \TTpm^{d+1}$.
For every $i \in [m]$, let $\bm H_i = \bHclosed^+\bigl(\clif(a_i)\bigr)$.
By \cref{le:sval_hspace} we have $X = \bigcap_{i = 1}^{m} \Hclosed^{+}(a_i) = \bigcap_{i = 1}^{m} \sval(\bm H_i)$.

The implication $(iv) \Rightarrow (iii)$ is trivial and the implication $(iii) \Rightarrow (i)$ follows from \cref{le:sval_cone}.

To prove the implication $(i) \Rightarrow (ii)$, we use the tropical Minkowski--Weyl theorem for unsigned tropical convexity~\cite[Theorem~1]{GaubertKatz:2011}.
This theorem implies that for every closed orthant $\ort$ of $\TTpm^d$ there exists a finite set $U_{\ort} \subset \ort$ such that $X \cap \ort = \wcone(U_{\ort})$.
Let $U = \bigcup_{\ort} U_{\ort}$.
We will show that $X = \wcone(U)$.
Since $\wcone(U_{\ort}) \subseteq \wcone(U)$ for every $\ort$, we get $X \subseteq \wcone(U)$.
To prove the opposite inclusion, note that $U \subseteq X$ implies $\whull(U) \subseteq X$ by \cref{cor:hspace_weakly_convex}.
Since $X$ is an intersection of linear tropical halfspaces, we get $\wcone(U) \subseteq X$.

To finish the proof, it is enough to show the implication $(ii) \Rightarrow (iii)$.
To do so, let 
\[
\bm P_{J} = \cone\bigl(\lif_{J}(x_1), \dots, \lif_{J}(x_m)\bigr)
\]
for all $J \subseteq [d]$. We will show that $X = \bigcap_{J \subseteq [d]} \sval(\bm P_{J})$.
If $x \in X$, there exists $\lambda \in \Tgeq$ such that $x \in \lambda \odot \whull(x_1, \dots, x_m)$.
By combining \cref{thm:TC-hull-intersection-halfspaces} with \cref{le:weak_hull_from_Puiseux}, for all $J \subseteq [d]$, we get 
\[
\whull(x_1, \dots, x_m) \subseteq \sval\Bigl(\conv\bigl(\lif_{J}(x_1), \dots, \lif_{J}(x_m)\bigr)\Bigr) \, .
\]
Hence, by \cref{le:multiplied_conv} we have
\begin{align*}
\lambda \odot \whull(x_1, \dots, x_m) &\subseteq \lambda \odot \sval\Bigl(\conv\bigl(\lif_{J}(x_1), \dots, \lif_{J}(x_m)\bigr)\Bigr) \\
&= \sval\Bigl(\conv\bigl(t^{\lambda}\lif_{J}(x_1), \dots, t^{\lambda}\lif_{J}(x_m)\bigr)\Bigr) \subseteq \sval(\bm P_J) \,
\end{align*}
and $x \in \bigcap_{J \subseteq [d]} \sval(\bm P_{J})$.
Conversely, if $x \in \bigcap_{J \subseteq [d]} \sval(\bm P_{J})$ then, by \cref{le:multiplied_conv}, for every $J\subseteq [d]$ there exists $\lambda_{J} \in \Tgeq$ such that
\begin{align*}
x &\in \sval\Bigl(\conv\bigl(t^{\lambda_{J}}\lif_{J}(x_1), \dots, t^{\lambda_{J}}\lif_{J}(x_m)\bigr)\Bigr) \\
&= \sval\Bigl(\conv\bigl(\lif_{J}(\lambda_{J} \odot x_1), \dots, \lif_{J}(\lambda_{J} \odot x_m)\bigr)\Bigr) \, .
\end{align*} 
Hence, by combining \cref{thm:TC-hull-intersection-halfspaces} with \cref{le:weak_hull_from_Puiseux} and \cref{le:TC_cone_from_faces} we get
\begin{align*}
x &\in \bigcap_{J} \sval\Bigl(\conv\SetOf{\lif_{J}(\lambda_{\tilde{J}} \odot x_i)}{\tilde{J} \subseteq [d], i \in [m]}\Bigr) \\
&= \whull\Bigl(\SetOf{\lambda_{\tilde{J}} \odot x_i}{\tilde{J} \subseteq [d], i \in [m]}\Bigr) \\
&\subseteq \wcone(x_1, \dots, x_m)  = X\, . \qedhere
\end{align*}
\end{proof}

Establishing an affine version of the former theorem requires us to come up with an appropriate concept of dehomogenization. 

\begin{theorem} \label{thm:affine_M-W}
Let $X \subseteq \TTpm^d$ be an arbitrary set.
Then, the following are equivalent:
\begin{enumerate}[(i)]
\item $X$ is an intersection of finitely many closed tropical halfspaces.
\item There exist two finite sets $V,W \subset \TTpm^d$ such that 
\[
\whull\bigl(\SetOf{v \lplus \lambda \odot w }{v \in V, w \in W, \lambda \in \Tgeq}\bigr) = X \, .
\]
\item There exists a finite collection of polyhedra $\bm P_1, \dots, \bm P_m \subset \KK^d$ such that 
\[
\bigcap_i \sval(\bm P_i) = X \, .
\]
\item There exists a finite collection of affine halfspaces $\bm H_1, \dots, \bm H_m \subseteq \KK^d$ such that
\[
\bigcap_i \sval(\bm H_i) = X \, .
\]
\end{enumerate}
\end{theorem}
\begin{proof}
The theorem is trivial if $X$ is empty.
From now on we suppose that $X$ is nonempty.

The implication $(i) \Rightarrow (iv)$ follows by the same argument as in the proof of \cref{th:conic_M-W}.

The implication $(iv) \Rightarrow (iii)$ is trivial.

To prove that $(iii) \Rightarrow (ii)$, for every $i \in [m]$, let $\bm Q_i = \{(\bm \lambda \bm x, \bm \lambda) \colon \bm \lambda \ge 0, \bm x \in \bm P_i\} \subseteq \KK^{d+1}$.
The set $\bm Q_i$ is a polyhedral cone.
Furthermore, we have the equality 
\begin{equation}\label{eq:homog_puiseux}
\sval(\bm Q_i) \cap \{x_{d+1} = 0\} = \SetOf{(x,0)}{ x \in \sval(\bm P_i)} \, .
\end{equation}
Indeed, if $x \in \sval(\bm P_i)$, then there exists $\bm x \in \bm P_i$ such that $\sval(\bm x) = x$.
Therefore, $(\bm x, 1) \in \bm Q_i$ and $(x,0) \in \sval(\bm Q_i)$.
Conversely, if $(x,0) \in \sval(\bm Q_i)$, then there exist $\bm \lambda \ge 0$ and $\bm x \in \bm P_i$ such that $\sval(\bm \lambda) = 0$ and $x = \sval(\bm \lambda \bm x) = \sval(\bm \lambda) + \sval(\bm x) = \sval(\bm x)$.
Therefore $x \in \sval(\bm P_i)$ and~\cref{eq:homog_puiseux} is satisfied.
Let $Y = \bigcap_{i = 1}^{m} \sval(\bm Q_i) \subseteq \TTpm^{d+1}$.
By~\cref{eq:homog_puiseux}, we get
\begin{equation}\label{eq:homog_puiseux2}
Y \cap \{x_{d+1} = 0\} = \SetOf{(x, 0)}{x \in X} \, .
\end{equation}
Applying \cref{th:conic_M-W} to the set $Y$, there exists a finite set $U = \{u_1, \dots, u_n\} \subset \TTpm^{d+1}$ such that $Y = \wcone(U)$.
Since $Y$ is a cone, we can suppose that $U$ contains $\Zero$.
Moreover, we have $y_{d+1} \ge \Zero$ for every $y \in Y$.
Hence, we can scale every point $u^{(i)} \in U$ in such a way that $u^{(i)}_{d+1} \in \{\Zero, 0\}$.
Thus, we can write $U  = \hat{V} \cup \hat{W}$ where $\hat{V}$ contains the elements of $U$ whose last coordinate is $0$ and $\hat{W}$ contains the elements whose last coordinate is $\Zero$.
Since $U$ contains $\Zero$, the set $\hat{W}$ is nonempty.
Since $X$ is nonempty,~\cref{eq:homog_puiseux2} implies that $\hat{V}$ is also nonempty.
Therefore, by combining~\cref{eq:homog_puiseux2} with \cref{le:homog_halflines} we get
\[
X = \whull\bigl(\SetOf{v \lplus \lambda \odot w }{v \in V, w \in W, \lambda \in \Tgeq}\bigr) \, ,
\]
where $V,W$ are the projections of $\hat{V}, \hat{W}$ obtained by deleting the last coordinate.

To prove the implication $(ii) \Rightarrow (i)$, we define $\hat{V} = \{(v,0) \colon v \in V\}$ and $\hat{W} = \{(w,\Zero) \colon w \in \hat{W}\}$.
By \cref{th:conic_M-W}, we have 
\[
\wcone(\hat{V} \cup \hat{W}) = \bigcap_{i = 1}^{n} \SetOf{x \in \TTpm^{d+1}}{a_{i,1} \odot x_1 \oplus \dots \oplus a_{i,d+1} \odot x_{d+1} \teq \Zero } \, 
\]
for some finite set $\{a_1, \dots, a_n\} \in \TTpm^{d+1} \setminus \{\Zero\}$. By \cref{le:homog_halflines} we have the equality $\{(x, 0) \colon x \in X\} = \wcone(\hat{V} \cup \hat{W}) \cap \{x_{d+1} = 0\}$ and therefore
\[
X = \bigcap_{i = 1}^{n} \SetOf{x \in \TTpm^{d+1}}{a_{i,d+1} \oplus a_{i,1} \odot x_1 \oplus \dots \oplus a_{i,d} \odot x_d \teq \Zero } \, . \qedhere
\]
\end{proof}

\section{Oriented (valuated) matroids and Bergman fans}
\label{sec:oriented-matroids}

We give a brief introduction of the necessary notions to connect oriented matroids and oriented valuated matroids with TC-convexity. 

Two vectors $u,v \in \TTpm^k$ are \emph{orthogonal}, denoted by $u \perp v$, if $u \odot v \in \Tzero$.
For a set $U \subseteq \TTpm^k$, we set $U^{\perp} \coloneqq \{v \in \TTpm^k \mid u \perp v\ \forall u \in U\}$. 

\begin{lemma} \label{lem:orthogonal-as-cone}
  For a subset $U \subseteq \TTpm^k$ we have $(U^{\perp})^{\perp} \supseteq \wspan(U)$. 
  For finite $U$, equality holds: $(U^{\perp})^{\perp} = \wspan(U)$. 
\end{lemma}
\begin{proof}
  By definition of orthogonality, we have $s \odot u \in (U^{\perp})^{\perp}$ for all $s \in \TTpm$ and $u \in U$.   
  As an intersection of hyperplanes, the set $(U^{\perp})^{\perp}$ is TC-convex.
  Therefore, $(U^{\perp})^{\perp} \supseteq \wcone(U,\ominus U)$. 
  For finite $U$, assume that there is an element $p \in (U^{\perp})^{\perp} \setminus \wspan(U)$.
  Then \cref{thm:TC-hull-intersection-halfspaces} shows that there is an element $w \in \TTpm^k$ such that $w \odot p < \Zero$ and $w \odot q \models \Zero$ for all $q \in \wcone(U,\ominus U)$.
  Though, then $w \odot u \models \Zero$ and $w \odot \ominus u \models \Zero$ for all $u \in U$ which implies $w \in U^{\perp}$.
  By definition of the double orthogonal, this shows $w \odot p \in \Tzero$, a contradiction. 
\end{proof}

\begin{remark} \label{rem:span+orthogonal}
  Motivated by the former lemma we clarify the interplay between the span and the orthogonal of a set.
  It turns out that $\wspan(U)^{\perp} = \wspan(U^{\perp})$ for an arbitrary subset $U \subseteq \TTpm^k$. 

  For this, note that $U^{\perp} = \wspan(U^{\perp})$ for any $U \subseteq \TTpm^k$ as $U^{\perp}$ is a TC-convex set as intersection of TC-convex sets. 
  Using the inclusion of \cref{lem:orthogonal-as-cone} on $U^{\perp}$ and the orthogonal of the inclusion for $U$ gives $U^{\perp} = \wspan(U^{\perp}) \subseteq U^{\perp \perp \perp} \subseteq \wspan(U)^{\perp}$.
  As $U \subseteq \wspan(U)$ implies $\wspan(U)^{\perp} \subseteq U^{\perp}$, we obtain the equality $U^{\perp} = \wspan(U^{\perp}) = \wspan(U)^{\perp} = U^{\perp \perp \perp}$. 
\end{remark}

Recently, there has been several advances on \emph{matroids over hyperfields} starting with \cite{BakerBowler:2019} building on the investigation of hyperfields in the context of tropical geometry~\cite{Viro:2010} among others. 
This framework generalizes matroids, oriented matroids, and valuated matroids.
In this terminology, oriented matroids are matroids over the \emph{sign hyperfield}; that is the set $\SignH = \{\ominus 0,\Zero,0\}$ equipped with a hyperfield structure.
We refer to~\cite{BLSWZ:1993,Anderson:2019,Celaya:2019} for further reading on oriented matroids and the interpretation in the context of hyperfields. 

All three examples arise naturally as images from matroids over valued fields (at least the `realizable' ones as discussed below) and, as such, they have additional useful properties.
This is captured by the class of \emph{stringent} hyperfields.  
Several cryptomorphic definitions of matroids over stringent hyperfields attain a simpler form as derived in~\cite[Section 4]{BowlerPendavingh-arxiv:2023}.
We consider two special classes, namely matroids over the hyperfield of signs (aka oriented matroids) and matroids over the real tropical hyperfield (aka oriented valuated matroids). 

Let $\HH \in \{\SignH,\TTpm\}$. 
Since $y \oplus x = x \lplus y$ for all $x \neq \ominus y \in \HH$, oriented matroids and oriented valuated matroids fall into the category of matroids over stringent hyperfields.
Therefore, \cite[Theorem 44]{BowlerPendavingh-arxiv:2023} allows one to state a unified concise definition of \emph{vectors} for oriented matroids and oriented valuated matroids.
We obtain an even simpler version using $\supp(x \lplus y) = \supp(x) \cup \supp(y)$ for all $x,y \in \HH$. 
Here, the operation $\circ$ defined at the beginning of \cite[Section 4.3]{BowlerPendavingh-arxiv:2023} boils down to the left sum. 
In the statement of the following definition, oriented matroids are $\SignH$-matroids and oriented valuated matroids are $\TTpm$-matroids. 

\begin{definition}\label{def:matroid_vectors}
  Let $\HH \in \{\SignH,\TTpm\}$. 
  A subset $\mathcal{V} \subseteq \HH^k$ forms the \emph{vectors} of a $\HH$-\emph{matroid} $\cM$ if and only if the following properties hold:
  \begin{itemize}
  \item[(V0)] $\Zero \in \mathcal{V}$. 
  \item[(V1)] If $s \in \HH$ and $v \in \mathcal{V}$ then $s \odot v \in \mathcal{V}$. 
  \item[(V2)] If $x,y \in \mathcal{V}$ and $z = x \lplus y$, then $z \in \mathcal{V}$. 
  \item[(V3)] If $x,y \in \mathcal{V}$, $i \in [k]$ with $x_i = \ominus y_i \neq \Zero$, then there is a $z \in \mathcal{V}$ with $z \in \Uncomp(x \oplus y) \cap \HH^k$ and $z_i = \Zero$. 
  \end{itemize}
\end{definition}

\begin{definition}[{\cite[Lemma 5]{BowlerPendavingh-arxiv:2023}}] \label{def:others-from-vectors}
The \emph{circuits} of $\cM$ are the support-minimal non-zero vectors.
  The \emph{covectors} of $\cM$ are the elements of $\HH^k$ orthogonal to the circuits of $\cM$.
  The \emph{cocircuits} of $\cM$ are the support-minimal non-zero covectors. 
  The covectors form the vectors of the \emph{dual matroid}. 
\end{definition}

Note that for a field $F$, the $F$-matroids just correspond to linear spaces over $F$. 
In particular, the vectors are the elements in a linear space over $F$ and the circuits are the support-minimal non-zero vectors of the linear space.
The covectors and cocircuits then arise from the dual linear space.
We will use this in the following for the fields $\RR$ and $\puiseux{\RR}{t}$. 

We summarize a few statements of~\cite{BowlerPendavingh-arxiv:2023} to derive the following lemma.

\begin{lemma} \label{lem:orthogonal-closure-circuits}
  Let $M$ be a $\TTpm$-matroid with vectors $\mathcal{V}$, covectors $\mathcal{W}$, circuits  $\mathcal{C}$ and cocircuits $\mathcal{D}$.
  Then $\mathcal{V} = \mathcal{D}^{\perp}$, $\mathcal{W} = \mathcal{C}^{\perp}$ and $\mathcal{V} = (\mathcal{C}^{\perp})^{\perp}$. 
\end{lemma}
\begin{proof}
  The representation of the vectors by cocircuits and of covectors by circuits via orthogonality is the original definition from \cite{BowlerPendavingh-arxiv:2023} which was shown to be equivalent to \cref{def:matroid_vectors}. 
  
  By~\cite[Corollary 34]{BowlerPendavingh-arxiv:2023}, stringent hyperfields are also perfect. 
  Therefore, one gets $\mathcal{W} \subseteq \mathcal{V}^{\perp}$.
  As circuits are special vectors, the inclusion-reversing structure of orthogonality yields $\mathcal{C}^{\perp} \supseteq \mathcal{V}^{\perp}$.
  Together, we get $\mathcal{W} = \mathcal{V}^{\perp} = \mathcal{C}^{\perp}$ as the first and the last of the sets in this chain are equal by the first part.
  Analogously, we can deduce $\mathcal{V} = \mathcal{W}^{\perp}$.
  Using again $\mathcal{W} = \mathcal{C}^{\perp}$ implies $\mathcal{V} = (\mathcal{C}^{\perp})^{\perp}$. 
\end{proof}

A matroid $M_0$ over a hyperfield $H_0$ is \emph{realizable} over another hyperfield $H_1$ if there is a matroid $M_1$ over $H_1$ and a hyperfield morphism $H_1 \to H_0$ mapping $M_1$ to $M_0$.
We do not need this generality but only a more special case, and we refer to \cite{BakerBowler:2019} for more details.
Recall that $\sval$ maps vectors over Puiseux series $\puiseux{\RR}{t}$
to vectors over $\TTpm$.
Furthermore, $\tsgn$ maps vectors over $\TTpm$ or vectors over $\puiseux{\RR}{t}$
to vectors over $\SignH$.
Furthermore, the sign map $\sgn$ maps vectors over $\RR$ to vectors over $\SignH$.
Here, we have a slight abuse of notation as we do not distinguish between the ways we represent the three elements of $\SignH$ as $\{\ominus 0,\Zero,0\}$, $\{\ominus,\Zero,\oplus\}$ or $\{-,0,+\}$. 
All three maps indeed give rise to hyperfield morphisms.  

\begin{definition}
  Let $\mathcal{R} \in \{\RR,\puiseux{\RR}{t}\}$ and $\HH \in \{\SignH,\TTpm\}$.
  Then an $\HH$-matroid $M$ is \emph{realizable} over $\mathcal{R}$ if the set of vectors of $M$ is the image of the set of vectors of a linear space over $\mathcal{R}$.
\end{definition}

\begin{example}

  Consider the matrix 
  \begin{align*}
    T =
    \begin{pmatrix}
      1 & 1 & 1 & 1\\
      t^2 & t^{-1} & -t^3 & 1\\
      -t^{-1} & -t^{1} & 0 & t^{-2}\\
    \end{pmatrix} \in \puiseux{\RR}{t}^{3 \times 4} \enspace . 
  \end{align*}
  One can think of $t$ as a very big number, so $t^{-\Omega}$ for a sufficiently big $\Omega$ is a very small number.
  Taking the row span of $T$ yields a linear subspace of $\puiseux{\RR}{t}^{4}$.
  The image of the vectors in this linear subspace under the map $\sval$ form the vectors of an $\TTpm$-matroid, and under the map $\tsgn$, they form the vectors of an $\SignH$-matroid. 
  As an example, we obtain
  \begin{align*}
    \sval(t^2, t^{-1}, -t^3, 1) = (2, -1, \ominus 3, 0), \quad \sval(-t^{-1}, -t^{1}, 0, t^{-2}) = (\ominus -1, \ominus 1, \Zero, -2)
  \end{align*}
  and
  \begin{align*}
    \tsgn\left( (-t^{-1}, -t^{1}, 0, t^{-2}) + t^{-5} \cdot (t^2, t^{-1}, -t^3, 1) \right) = (\ominus 0,\ominus 0,\ominus 0,\oplus 0)\\
    = (\ominus 0,\ominus 0,\Zero,\oplus 0) \lplus (\oplus 0,\oplus 0,\ominus 0,\oplus 0) \\
    = \tsgn\left( (-t^{-1}, -t^{1}, 0, t^{-2}) \right) \lplus \tsgn\left(t^{-5} \cdot (t^2, t^{-1}, -t^3, 1) \right)
  \end{align*}
  For a nice discussion on the interpretation of $\lplus$ see \cite[Example 6.5]{Anderson:2019}. 
  \end{example}

The image of a linear space over complex Puiseux series under the valuation is a \emph{tropical linear space}.
If the coefficients of the defining equations of the linear space are constant, then this gives rise to the \emph{Bergman fan} of the underlying linear matroid; see, e.g.,~\cite{Joswig:2022} for more details.
This idea was generalized to the setting of oriented matroids in~\cite[Proposition~2.4.7]{Celaya:2019} under the name \emph{real Bergman fan}.
In the general setting of matroids over hyperfields, the set of vectors forms the generalization of the Bergman fan.
Already for tropical linear spaces, a connection to tropical convexity was established in that it is the tropical convex hull of the circuits~\cite{YuYuster:2006}.
Furthermore, \cite[Theorem 3.14]{Tabera:2015} basically states that the set of vectors of a $\TTpm$-matroid realizable by a linear space $L$ over $\puiseux{\RR}{t}$ is the orthogonal set to the image of the cocircuits of $L$.
We give a unified generalization of these insights. 

\begin{theorem} \label{thm:representation-vectors-tropical-hyperfields}
  Let $\mathcal{C}, \mathcal{D}, \mathcal{V} \subseteq \TTpm^k$ be the set of circuits, cocircuits and vectors of a $\HH$-matroid $\cM$ where $\HH \in \{\TTpm,\SignH\}$.

  For $\HH = \TTpm$,
  \begin{equation}\label{eq:matroids_ttpm}
    \mathcal{V} = \whull(\mathcal{C})\cup \{\Zero\} = \bigcap_{d \in \mathcal{D}} \HC\left((\Zero,d)\right) = \mathcal{D}^{\perp} \enspace .
  \end{equation}
  
  Furthermore, for $\HH = \SignH$,
  \begin{equation}
    \wspan(\mathcal{V}) = \wspan(\mathcal{C}) = \bigcap_{d \in \mathcal{D}} \HC\left((\Zero,d)\right)  \enspace ,
  \end{equation}
    \begin{equation}
    \whull(\mathcal{V}) = \whull(\mathcal{C} \cup \{\Zero\}) = \bigcap_{d \in \mathcal{D}} \HC\left((\Zero,d)\right) \cap [\ominus 0, 0]^k  \enspace ,
    \end{equation}
  \begin{equation}
   \mathcal{V} = \{\ominus 0,\Zero, 0\}^k \cap \whull(\mathcal{C} \cup \{\Zero\}) = \{\ominus 0,\Zero, 0\}^k \cap \bigcap_{d \in \mathcal{D}} \HC\left((\Zero,d)\right)  \enspace .
  \end{equation}
  
  If $\cM$ is realizable by a linear space $L$ over $\RR$ or over $\puiseux{\RR}{t}$, then the image of $L$ is the TC-convex hull of the image of the circuits of $L$ and the orthogonal to the images of the cocircuits of $L$ under the signed valuation. 
\end{theorem}

The proof relies on the following lemma, which allows us to transfer the results from matroids over $\TTpm$ to matroids over $\SignH$.

\begin{lemma}\label{le:embedding}
Suppose that $\mathcal{C}, \mathcal{D} \subseteq \TTpm^k$ are the sets of circuits and cocircuits of a matroid over $\SignH$. Consider the sets $\widetilde{\mathcal{C}} = \SetOf{\lambda \odot v}{\lambda \in \TTpm \setminus \{\Zero\},\, v \in \mathcal{C}}$, $\widetilde{\mathcal{D}} = \SetOf{\lambda \odot u}{\lambda \in \TTpm \setminus \{\Zero\},\, u \in \mathcal{D}}$. Then, the sets $\widetilde{\mathcal{C}},\widetilde{\mathcal{D}}$ are the sets of circuits and cocircuits of a matroid over $\TTpm$.
\end{lemma}
\begin{proof}
Observe that the scaling yields a correct signature of the circuits and cocircuits. 
The orthogonality of the circuits and cocircuits over $\TTpm$ follows from the orthogonality over $\SignH$.
Now, the claim follows directly from the characterization of a pair of circuits of a matroid and the circuits of its dual matroid given in \cite[Theorem 4.17]{BakerBowler:2019} and \cite[Theorem~3]{BowlerPendavingh-arxiv:2023}.
\end{proof}

\begin{proof}[Proof of \cref{thm:representation-vectors-tropical-hyperfields}]  
  We start by proving the claim for $\HH = \TTpm$. By \cref{lem:orthogonal-closure-circuits}, we directly get $\mathcal{V} = \mathcal{D}^{\perp} = \bigcap_{d \in \mathcal{D}} \HC\left((\Zero,d)\right)$.
  In particular, $\mathcal{V}$ is TC-convex, hence, due to (V1), we have $\mathcal{V} = \wspan(\mathcal{V})$. 
  
  By the monotonicity of $\wspan$, we obtain $\wspan(\mathcal{C}) \subseteq \wspan(\mathcal{V})$.
  Recall that circuits of fixed support are unique up to scaling (see e.g. property (C2) in \cite[\S 2.2]{BowlerPendavingh-arxiv:2023}).
  Applying Lemma~\ref{lem:orthogonal-as-cone} to a finite representative set of $\mathcal{C}$ up to scaling yields $\left(\mathcal{C}^{\perp}\right)^{\perp} = \wspan(\mathcal{C})$.
  Together with \cref{lem:orthogonal-closure-circuits} we get $\wspan(\mathcal{C}) = \left(\mathcal{C}^{\perp}\right)^{\perp} = \mathcal{V}$.
  Since the set of circuits forms a cone, \cref{cor:hull_of_cone} gives $\wspan(\mathcal{C}) = \wcone(\mathcal{C}) = \whull(\mathcal{C})\cup \{\Zero\}$.  

  \smallskip

  To prove the claims for $\SignH$, note that the equality $\wspan(\mathcal{C}) = \bigcap_{d \in \mathcal{D}} \HC\left((\Zero,d)\right)$ follows from \cref{le:embedding} combined with \cref{eq:matroids_ttpm}, since $\wspan(\mathcal{C}) = \wspan(\widetilde{\mathcal{C}})$ and $ \bigcap_{d \in \mathcal{D}} \HC\left((\Zero,d)\right) =  \bigcap_{d \in \widetilde{\mathcal{D}}} \HC\left((\Zero,d)\right)$.
  Moreover, since every non-$\Zero$ vector of an oriented matroid is a composition of circuits, we have $\mathcal{C} \cup \{\Zero\} \subseteq \mathcal{V} \subseteq \whull(\mathcal{C}) \cup \{\Zero\}$, which gives $\wspan(\mathcal{V}) = \wspan(\mathcal{C})$, $\whull(\mathcal{V}) = \whull(\mathcal{C} \cup \{\Zero\})$, and shows the inclusion
\[
   \mathcal{V} \subseteq \{\ominus 0,\Zero, 0\}^d \cap \whull(\mathcal{C} \cup \{\Zero\}) \subseteq \{\ominus 0,\Zero, 0\}^d \cap \bigcap_{d \in \mathcal{D}} \HC\left((\Zero,d)\right)  \enspace .
\]

The first and the last set are equal by the definition of covectors in terms of orthogonality applied to the dual matroid (\cref{def:others-from-vectors}).
Therefore, the chain of inclusions is actually an equality. Applying \cref{lem:special-hull-unit-cube} to $\mathcal{V}$ gives $\whull(\mathcal{V}) = \wspan(\mathcal{V}) \cap [\ominus 0, 0]^k$, which finishes the proof of all of the equalities.

  \smallskip

  For the second part, we rely on insights about hyperfield morphisms from~\cite{BakerBowler:2019}. 
  Signed valuation gives rise to a hyperfield morphism to the real tropical hyperfield; forgetting the valuation further gives rise to a hyperfield morphism to the sign hyperfield. 
  Since hyperfield morphisms transfer matroids to matroids, realizability is preserved. 
\end{proof}

\begin{remark}
The hyperfield morphism from the real tropical hyperfield to the tropical hyperfield transfers vectors of a matroid over $\TTpm$ to vectors of an unsigned tropical linear space.
So, while taking the tropical absolute value of an arbitrary TC-tropically convex set in general does not result in a tropically convex set, this holds for linear spaces over the real tropical hyperfield. 
\end{remark}

\begin{remark}
  Note that we have $\mathcal{V} = \left(\mathcal{V}^{\perp}\right)^{\perp}$ over $\TTpm$.
  Over a field this characterizes linear spaces. 
  Over $\TTpm$, by \cref{rem:span+orthogonal}, this equality actually holds exactly for those sets of the from $U^{\perp}$ for some subset $U \subseteq \TTpm^k$.
  It remains open if sets of this form are finitely generated.
  The analogous question for unsigned tropical convexity was answered negatively in \cite{GrigorievVorobjov:2021}.
\end{remark}

\section{Conclusion}

One of our main results is the representation of a finitely generated TC-convex set as an intersection of closed tropical halfspaces in \cref{thm:TC-hull-intersection-halfspaces}. 
Furthermore, we show in \cref{thm:tropical_polar} that also the valuation of a closed convex semialgebraic set has such a representation. 
This motivates the following. 
\begin{conjecture}
  For all closed TC-convex sets, we have
  \begin{equation}
    X = \bigcap_{a \in \TTpm^{d+1}} \SetOf{\Hclosed^+(a)}{X \subseteq \Hclosed^+(a)} \,.
  \end{equation}
\end{conjecture}

Such a more general statement could be deduced from a more direct proof of the representation by halfspaces.
On one hand, it would be interesting to get stronger separation without relying on the separation results over Puiseux series by using \cref{le:weak_hull_from_Puiseux} in the proof of \cref{thm:TC-hull-intersection-halfspaces}.  
On the other hand, the proof of \cref{thm:separation_hemispaces} uses the Kuratowski--Zorn lemma which is highly non-constructive. 
\begin{question}
 How can one deduce \cref{thm:TC-hull-intersection-halfspaces} in a more constructive way? 
\end{question}

\Cref{ex:caratheodory-number} discusses the Carath{\'e}odory number $c_d$ of TC-convexity.
The examples gives a lower bound $c_d \ge 2^d$ complementing the upper bound $c_d \leq d2^d + 1$ given in \cref{le:inter_caratheodory}.
Recall that TC-convexity extends the `usual' tropical convexity which has Carath{\'e}odory number $d+1$ as already shown in~\cite{Helbig:1988,DevelinSturmfels:2004,BriecHorvath:2004}. 

\begin{question}
 What is the Carath{\'e}odory number of TC-convexity?
\end{question}

An approach for the former question would be via a better understanding of the operator $\faces(\cdot)$.  
In \cref{def:vert+faces}, we introduced it as a crucial building block for the structure of TC-convex sets.
On a more abstract level, the operator produces certain cubical subcomplexes of the cubical complex formed by the faces of a hypercube. 
A better understanding of these complexes might lead to a better bound. 

\smallskip

Our connection to matroids over hyperfield in \cref{sec:oriented-matroids} together with the recent work~\cite{MaxwellSmith:2023} motives the following.

\begin{question}
 Can one extend the theory of tropical convexity to ordered perfect hyperfields preserving many of the features of TC-convexity? 
\end{question}

Furthermore, (unsigned) tropical linear spaces were characterized as those tropical varieties which are also tropically convex in~\cite{Hampe:2015}.
This may also be extended to the signed setting.

\begin{question}
 Are linear spaces over the signed tropical hyperfield exactly those signed tropical varieties which are TC-convex? 
\end{question}

Here, by linear space, we mean the sets of vectors of a matroid over the signed tropical hyperfield. 
These sets of vectors were described in \cref{thm:representation-vectors-tropical-hyperfields}.
We leave it open, what exactly we mean by signed tropical variety; one may even be lead to define a signed tropical variety as a polyhedral complex that looks locally like a linear space over the signed tropical hyperfield.

\section*{Acknowledgements}

We thank the 2021 HIM program Discrete Optimization during which part of this work was developed.
This research benefitted from the support of the FMJH Program PGMO.
We thank Ben Smith for helpful comments on preliminary versions.

\bibliographystyle{plain}
 \newcommand{\noop}[1]{}

\appendix

\section{Details on semialgebraic sets}
\label{sec:details+semialgebraic+separation}

We give more explanations for the statement and proof of \cref{th:hyperplane_sep}. We refer to \cite[Corollary~3.3.20]{Marker:2002} for another example of a proof that works in the same way.

\begin{proof}[Extended proof of \cref{th:hyperplane_sep}]
Recall that $\bm X, \bm Y \subseteq \KK^d$ are two nonempty convex semialgebraic sets such that $\bm X \cap \bm Y = \emptyset$.
Since $\bm X$, $\bm Y$ are semialgebraic, there exist two first-order formulas $\phi(x_1, \dots, x_d, y_1, \dots y_{n_1}), \psi(x_1, \dots, x_d, y_1, \dots, y_{n_2})$ in the language of ordered fields and two vectors $\bm r \in \KK^{n_1}, \bm s \in \KK^{n_2}$ such that 
\begin{align*}
\bm X &= \{\bm x \in \KK^d \colon \phi(\bm x, \bm r) \text{ is true in } \KK \} \, , \\
\bm Y &= \{\bm x \in \KK^d \colon \psi(\bm x, \bm s) \text{ is true in } \KK \} \, .
\end{align*}
  For every $\bm w \in \KK^{n_1}, \bm z \in \KK^{n_2}$, let $\bm X_{\bm w}, \bm Y_{\bm z}$ be the semialgebraic sets defined as 
\begin{align*}
\bm X_{\bm w} &= \{\bm x \in \KK^d \colon \phi(\bm x, \bm w) \text{ is true in } \KK \} \, , \\
\bm Y_{\bm z} &= \{\bm x \in \KK^d \colon \psi(\bm x, \bm z) \text{ is true in } \KK \} \, .
\end{align*}
Now, the statement 
\begin{quote}
``For every $\bm w, \bm z$ such that $\bm X_{\bm w}, \bm Y_{\bm z}$ are nonempty, convex, and disjoint, there exists a vector $\bm a$ such that $\bm X_{\bm w} \subseteq \bHclosed^+(\bm a)$ and $\bm Y_{\bm z} \subseteq \bHclosed^-(\bm a)$''
\end{quote}
  can be written as sentence in the language of ordered fields.
  This sentence is true over $\RR$ by the hyperplane separation theorem, and so it is true over $\KK$ by the completeness of the theory of real closed fields.
  By taking $(\bm w, \bm z) = (\bm r, \bm s)$ we obtain the claim.
  \end{proof}
  
  \begin{example}
  The assumption that both sets $\bm X$, $\bm Y$ are semialgebraic cannot be dropped from \cref{th:hyperplane_sep} as the following example, taken from \cite{Robson:1991}, shows. 
  Let $\bm A = \SetOf{\bm z \in \KK}{\bm z \ge 0 \land \val(\bm z) < 0}$ and $\bm B = \SetOf{\bm z \in \KK}{\bm z \ge 0 \land \val(\bm z) \ge 0}$.
  In this way, $\bm A \cup \bm B$ is the set of nonnegative Puiseux series, $\bm A$ does not have a least upper bound in $\KK$, $\bm B$ does not have a greatest lower bound in $\KK$, and $\bm a < \bm b$ for all $\bm a \in \bm A, \bm b \in \bm B$.
  Let $\bm X = \bm X_1 \cup \bm X_2$, where
\begin{align*}
\bm X_1 &= \SetOf{(\bm x, \bm y) \in \KK^2}{\bm x \ge 0 \, \land \, (\forall \bm z \in \bm B , \ \bm y \le \bm z \bm x)} \, . \\
\bm X_2 &= \SetOf{(\bm x, \bm y) \in \KK^2}{\bm x \le 0 \, \land \, (\forall \bm z \in \bm A , \ \bm y \le \bm z \bm x )} \, . 
\end{align*}
We note that $\bm X$ is nonempty because it contains $(0,0)$.
Also, both $\bm X_{1}$ and $\bm X_2$ arise as intersections of closed convex sets, so $\bm X_1$ and $\bm X_2$ are both convex and closed.
Hence $\bm X$ is closed.
To see that $\bm X$ is also convex, let $(\bm x_1, \bm y_1) \in \bm X_1$ and $(\bm x_2, \bm y_2) \in \bm X_2$ be such that $\bm x_1 > 0$ and $\bm x_2 < 0$.
Pick $\bm \lambda \in [0,1]$ such that $\bm \lambda \bm x_1 + (1 - \bm \lambda) \bm x_2 = 0$ and let $\bm z = \max(0, \bm y_1 / \bm x_1)$.
Since $\bm z$ is a lower bound for $\bm B$, we have $\bm z \in \bm A$.
Hence $\bm \lambda \bm y_1 + (1-\bm \lambda) \bm y_2 \le \bm z (\bm \lambda \bm x_1 + (1 - \bm \lambda) \bm x_2) = 0$.
Therefore, the point $\bm \lambda (\bm x_1, \bm y_1) + (1 - \bm \lambda)(\bm x_2, \bm y_2)$ belongs to both $\bm X_1$ and $\bm X_2$.
Hence, the whole segment between these two points belongs to $\bm X$.

Despite the fact that $\bm X$ is closed and convex, it cannot be separated by a hyperplane from \emph{any} point $\bm w \notin \bm X$.
To see that, suppose that $(\bm c_0, \bm c_1, \bm c_2)$ is such that $(\bm c_1, \bm c_2) \neq (0,0)$ and $\bm c_0 + \bm c_1 \bm x + \bm c_2 \bm y \ge 0$ for all $(\bm x, \bm y) \in \bm X$.
Since $(t^x, 0),(0,-t^x), (-t^x, -t^x), (t^x, t^{x - 0.5}) \in \bm X$ for arbitrarily large values of $x > 0$, we have $\bm c_1 > 0$ and $\bm c_2 < 0$.
Let $\bm z = \bm c_1 / |\bm c_2|$. If $\bm z \in \bm A$, then by taking any $\bm z' \in \bm A$ that is greater than $\bm z$ and considering the points $(t^x, \bm z' t^{x}) \in \bm X$ we get $\bm c_0 + \bm c_1 t^{x} + \bm c_2 \bm z' t^{x} \ge 0 \iff \bm z' \le \bm z + (\bm c_0 / |\bm c_2|) t^{-x}$, which is a contradiction for $x$ large enough.
Likewise, if $\bm z \in \bm B$, then by taking any $\bm z' \in \bm B$ that is smaller than $\bm z$ and considering the points $(-t^{x}, -\bm z' t^{x}) \in \bm X$ we get $\bm c_0 - \bm c_1 t^{x} - \bm c_2 \bm z' t^{x} \ge 0 \iff \bm z' \ge \bm z - (\bm c_0/|\bm c_2|) t^{-x}$, which is a contradiction for $x$ large enough.
\end{example}

\end{document}